%% file: ms.tex
\documentclass{article}

\usepackage{arxiv}

\usepackage[utf8]{inputenc} 
\usepackage[T1]{fontenc}    
\usepackage{hyperref}       
\usepackage{url}            
\usepackage{booktabs}       
\usepackage{amsfonts}       
\usepackage{nicefrac}       
\usepackage{microtype}      

\usepackage{natbib}
 \bibpunct[, ]{(}{)}{,}{a}{}{,}%

\usepackage[toc,page]{appendix}

\usepackage{mathtools}
\newcommand{\bhline}[1]{\noalign{\hrule height #1}}
\newcommand\twoline[2]{\mathclap{\substack{#1\\#2}}}
\newcommand\threeline[3]{\mathclap{\substack{#1\\#2\\#3}}}
 
\usepackage{amsmath,amsthm}
\newtheorem{lemma}{Lemma}[section]
\newtheorem{remark}{Remark}[section]

\title{An Iterated Dual Substitution Approach for \\
Binary Integer Programming Problems under\\
the Min--Max Regret Criterion}

\author{
    Wei Wu\\
    Shizuoka University\\
    Hamamatsu, Japan\\
    \texttt{goi@shizuoka.ac.jp}\\
    \And
    Manuel~Iori\\
    University of Modena and Reggio Emilia\\
    Reggio Emilia, Italy\\
    \texttt{manuel.iori@unimore.it}\\
    \AND
    Silvano~Martello\\
    University of Bologna\\
    Bologna, Italy\\
    \texttt{silvano.martello@unibo.it}\\
    \And
    Mutsunori~Yagiura\\
    Nagoya University\\
    Nagoya, Japan\\
    \texttt{yagiura@nagoya-u.jp}\\
}

\begin{document}
\maketitle

\begin{abstract}
We consider binary integer programming problems with the min--max regret objective function under interval objective coefficients.
We propose a new heuristic framework, which we call the iterated dual substitution (iDS) algorithm.
The iDS algorithm iteratively invokes a dual substitution heuristic and excludes from the search space any solution already checked in previous iterations.
In iDS, we use a best-scenario-based lemma to improve performance.
We apply iDS to four typical combinatorial optimization problems:
the knapsack problem, the multidimensional knapsack problem, the generalized assignment problem, and the set covering problem.
For the multidimensional knapsack problem, we compare the iDS approach with two algorithms widely used for problems with the min--max regret criterion:
a fixed-scenario approach, and a branch-and-cut approach.
The results of computational experiments on a broad set of benchmark instances show that the proposed iDS approach performs best on most tested instances.
For the knapsack problem, the generalized assignment problem, and the set covering problem,
we compare iDS with state-of-the-art results.
The iDS algorithm successfully updates best known records for a number of benchmark instances.
\end{abstract}


\section{Introduction}
Incomplete information, poor predictions of future events, or mismeasurements can make it difficult to provide accurate input parameters for real-world optimization problems.
Robust optimization is a common approach for solving optimization problems under uncertainty.
The min--max regret criterion is a representative criterion for robust optimization
in which value or cost parameters in the objective function are inaccurate in the optimization phase.
When a set of value or cost parameters (called a \emph{scenario}) is revealed,
the \emph{regret} of a solution is defined as the difference between the resulting objective value of the solution and the value of the optimal solution to the revealed scenario.
In this paper, we focus on the minimization of the maximum regret (min--max regret) over all scenarios.

The study of min--max regret strategy is motivated by practical applications
in which the decision maker has to make a decision and the regret for the decision in case of a wrong choice is crucial.
Recently, models and solution methods to the min--max regret problems have been applied to a number of real-world cases in many fields, such as
water resources management \citep[see][]{PKN12},
greenhouse gas abatement \citep[see][]{LK99},
semiconductor manufacturing \citep[see][]{CZJ12},
assembly line worker assignment and balancing \citep[see][]{Pe18},
power management \citep[see][]{DHCX11},
revenue management \citep[see][]{AKM16},
and portfolio optimization \citep[see][]{XMHZ17}.

The min--max regret versions of important combinatorial optimization problems such as 
the assignment problem \citep{PA11},
the shortest path problem \citep{KPY01},
the 0-1 knapsack problem \citep{FIMY15},
the generalized assignment problem \citep{WIMY18},
the traveling salesman problem \citep{MBMG07},
and the set covering problem \citep{PA13} have been recently studied,
and both exact and heuristic algorithms have been proposed for this kind of problems.
Several attempts at heuristic methods have been made,
including constructive heuristics based on the solution for a fixed scenario \citep[see][]{KZ06} or on the inclusion of a dual relaxation component \citep[see][]{FIMY15},
as well as more elaborated metaheuristics such as genetic algorithms or filter-and-fan methods \citep[see][]{PA13}. 
Concerning exact algorithms, Benders-like decomposition methods iteratively solve a relaxed master problem in which only a subset of scenarios is considered for constraints.
Scenarios corresponding to violated constraints are computed by solving a slave problem and possibly added to its master problem until a solution satisfying all constraints with respect to the scenario is found \citep[see, e.g.,][]{M06}.
Branch-and-cut methods, which extend Benders-like decomposition by including the solution of the slave problem at all nodes of the enumeration tree,
have been applied to several min--max regret problems \citep[see][]{PA13}.
Other authors have instead used the structure of a min--max regret problem to devise specifically tailored combinatorial branch-and-bound algorithms \citep[see][]{MG05-2}.
For general introductions to this area, we refer to \citet{ABV09}, \citet{CAM11}, \citet{K08}, and \citet{KY97}.

In this paper, we consider binary integer programming problems with the min--max regret objective function under interval objective coefficients.
In Section~\ref{sec:PD}, we describe the problem and we introduce a mixed integer programming model for it.
In Section~\ref{sec:CH}, we provide a brief introduction to classical heuristics for this kind of problems.
In Section~~\ref{sec:iDSH}, we propose a new heuristic framework, which we call the \emph{iterated dual substitution} (iDS) algorithm.
To exclude already checked solutions from the search space, the iDS approach considers two kinds of constraints:
those for pruning checked solutions based on the Hamming distance
and those using a lemma based on the best scenario.
We apply the proposed approach to the min--max regret version of four  representative combinatorial optimization problems:
the knapsack problem, the multidimensional knapsack problem, the generalized assignment problem, and the set covering problem.
In Section~\ref{sec:TP}, we describe these four problems, and we present computational results for their min--max regret versions in Section~\ref{sec:CE}. 
For the multidimensional knapsack problem, we compare the iDS approach with two algorithms widely applied to problems with the min--max regret criterion: 
a fixed-scenario approach and a branch-and-cut approach.
We evaluate these algorithms through computational experiments on a broad set of benchmark instances (available online).
The proposed iDS algorithm performs best for most tested instances.
For the knapsack problem, the generalized assignment problem, and the set covering problem,
we compare iDS with state-of-the-art results.
The results show that the iDS algorithm is highly efficient as a heuristic, successfully updating the best known records for a number of benchmark instances.

\section{Problem Description}\label{sec:PD}
Let $I=\{1,2,\ldots,m\}$ and $J=\{1,2,\ldots,n\}$. In this paper, we consider a \emph{binary integer programming} problem (BIP) with a linear objective function and linear constraints defined as
\begin{align}
\label{std:Obj}     \max\hspace{2mm}            &\sum^n_{j=1}c_{j}x_j                               \\
\label{std:InqCon}  \mathrm{s.t.}\hspace{4mm}   &\sum^n_{j=1}a_{ij}x_j\le b_{i}  &\forall i \in I,	\\
\label{std:IteCon}                              &x_j\in \{0,1\}                  &\forall j \in J,
\end{align}
where $\boldsymbol{x}$ represents the vector of variables,
$\boldsymbol{c}$ and $\boldsymbol{b}$ are vectors of coefficients,
and $\boldsymbol{A}=(a_{ij})$ is a matrix of coefficients.
For convenience, we define the set of all feasible solutions as 
\begin{align}
X_0 = \left\{\boldsymbol{x} \mid \boldsymbol{x} \text{ satisfies constraints \mbox{\eqref{std:InqCon}--\eqref{std:IteCon}}}\right\}.
\end{align}
Note that minimization problems and problems with equality constraints can be rewritten as an equivalent problem in the form of Eqs.~\eqref{std:Obj}--\eqref{std:IteCon}.
This form includes several classical combinatorial optimization problems
such as the knapsack problem, the assignment problem, the shortest path problem, and the set covering problem.

\subsection{Min--Max Regret Criterion}\label{sec:MMRC}
In many real-world applications, the objective coefficients $c_{j}$ are affected by various factors and therefore may be inaccurate.
Here, we assume that every $c_{j}$ can take any value within the range $[c^-_{j}, c^+_{j}]$.
A scenario $s$ is defined as a vector of costs $c^s_{j}$ satisfying $c^s_{j}\in [c^-_{j}, c^+_{j}], \forall j\in J$.
We denote the objective value of a solution $\boldsymbol{x}$ under scenario $s$ as
\begin{align*}
z^s(\boldsymbol{x})=\sum^n_{j=1}c^s_{j}x_j,
\end{align*}
and the optimal value under scenario $s$ as
\begin{align*}
z^s_*=\max_{\boldsymbol{y}\in X_0}z^s(\boldsymbol{y}).
\end{align*}
The \emph{regret} $r^s(\boldsymbol{x})$ corresponding to a solution $\boldsymbol{x}$ under scenario $s$ is then the difference between these two values:
\begin{align*}
r^s(\boldsymbol{x}) = z^s_* - z^s(\boldsymbol{x}).
\end{align*}
We define $S=\left\{s \ \big|\  c^s_{j}\in[c^-_{j}, c^+_{j}], \forall j\in J\right\}$ to be the set of all possible scenarios.
The \emph{maximum regret} $r_{\max}(\boldsymbol{x})$ of a solution $\boldsymbol{x}$ is then the maximum $r^s(\boldsymbol{x})$ value over all scenarios:
\begin{align}
\label{def:rmax}	r_{\max}(\boldsymbol{x}) = \max_{s\in S}  r^s(\boldsymbol{x}).
\end{align}
The \emph{binary integer programming problem under min--max regret} criterion (\mbox{MMR-BIP}) is the problem of finding a feasible solution $\boldsymbol{x}$ that minimizes the maximum regret:
\begin{align}
\nonumber	\min_{\boldsymbol{x}\in X_0} r_{\max}(\boldsymbol{x})	&=\min_{\boldsymbol{x}\in X_0}\max_{s\in S}  r^s(\boldsymbol{x})\\
\label{form:old}									&=\min_{\boldsymbol{x}\in X_0}\max_{s\in S} \left\{\max_{\boldsymbol{y}\in X_0}\sum^n_{j=1}c^s_{j}y_j - \sum^n_{j=1}c^s_{j}x_j\right\}.
\end{align}
The following lemma is a classic general result proposed by~\citet{ABV09} (based on~\citet{YKP01}).
\begin{lemma}\label{LemWorstCase}
The regret of a solution $\boldsymbol{x}\in X_0$ is maximized under the scenario $\sigma(\boldsymbol{x})$ defined as
\begin{align}\label{def:worstSen}
c^{\sigma(\boldsymbol{x})}_j=
\begin{cases}
c^-_j		& \textrm{if}\ \ x_j=1 \\
c^+_j	& \textrm{otherwise}
\end{cases}&&\forall j \in J.
\end{align}
\end{lemma}

From Lemma~\ref{LemWorstCase}, the \mbox{MMR-BIP} can be rewritten as
\begin{align}
\label{form:mmrbipBasic}	&\min_{\boldsymbol{x}\in X_0}r_{\max}(\boldsymbol{x})
						=\min_{\boldsymbol{x}\in X_0}\left\{\max_{\boldsymbol{y}\in X_0}\sum^n_{j=1}c^{\sigma(\boldsymbol{x})}_jy_j - \sum^n_{j=1}c^-_jx_j\right\}.
\end{align}

The MMR-BIP is $\Sigma^p_2$-hard,
because it has as a special case the interval min--max regret knapsack problem \citep{FIMY15}, which has been proved by \citet{DW10}
to be $\Sigma_2^p$-hard \citep[see][Chap. 7]{GJ79}.

\subsection{Mixed Integer Programming Model} \label{sec:MIP}
This section presents a \emph{mixed integer programming} (MIP) model for the \mbox{MMR-BIP}. 
By introducing a new continuous variable $\lambda$, along with constraints that force $\lambda$ to satisfy $\lambda \ge z_{*}^{s}$, $\forall s\in S$,
the \mbox{MMR-BIP} can be formulated as the following MIP model (\mbox{MIP-MMR-BIP}):
\begin{align}
\label{form:mipObj}		\min\hspace{2mm}		&\lambda-\sum_{j=1}^{n} c^-_jx_j																							\\
\label{form:mipBenCon}	\mathrm{s.t.}\hspace{3mm} 	&\lambda \ge \sum^n_{j=1}c^{\sigma(\boldsymbol{x})}_jy_j = \sum_{j=1}^{n}\left(c^+_j+\left(c^-_j-c^+_j\right)x_j\right)y_j	&\forall \boldsymbol{y} \in X_0,	\\
\label{form:mipXCon}								&\boldsymbol{x}\in X_0.
\end{align}

Observe that this model has an exponential number of constraints~\eqref{form:mipBenCon}.
In the following, we call constraints~\eqref{form:mipBenCon} \emph{Benders cuts}. 

\section{Classical Heuristics} \label{sec:CH}
In this section, we present classical algorithms for the \mbox{MMR-BIP}, including a fixed-scenario algorithm and a branch-and-cut framework.

\subsection{Fixed-Scenario Algorithm}\label{sec:FSA}
The fixed-scenario algorithm is designed based on the observation that the feasible region of the \mbox{MMR-BIP} is the same as that of the classical BIP. 
This implies that we can obtain a feasible solution to an \mbox{MMR-BIP} instance by fixing a scenario, solving the resulting BIP instance to optimality,
and using \eqref{form:mmrbipBasic} to evaluate the maximum regret of the obtained solution.
This approach has been used for solving a number of interval min--max regret problems (see, e.g., \citealp{KZ06}, \citealp{PA13}, and \citealp{WIMY18}).

Among the three most commonly used scenarios, namely, lowest value $\boldsymbol{c^-}$, highest value $\boldsymbol{c^+}$, and median value $(\boldsymbol{c^-}+\boldsymbol{c^+})/2$, 
in this paper we consider the median-value scenario, for which the following general result, proved by~\cite{KZ06}, can be directly applied to the \mbox{MMR-BIP}. 
\begin{lemma}\label{LemMidPnt}
Let $\widetilde{s}$ be the median-value scenario, i.e.,\ $c^{\widetilde{s}}_j = (c^-_j+c^+_j)/2,\ \forall j\in J$, and let $\boldsymbol{\widetilde{x}}$ be an optimal solution to the BIP under $\widetilde{s}$. Then, $r_{\max}(\boldsymbol{\widetilde{x}})\le 2r_{\max}(\boldsymbol{x})$ holds for all $\boldsymbol{x}\in X_0$.
\end{lemma}
\noindent Note that Lemma~\ref{LemMidPnt} also holds if the original BIP is a minimization problem, that is, 
\begin{align}
\label{std:MinObj}    \min\hspace{2mm} 
&\sum^n_{j=1}c_{j}x_j                               \\
\nonumber\mathrm{s.t.}\hspace{3mm}   &\text{\eqref{std:InqCon} and \eqref{std:IteCon}}.
\end{align} 
\subsection{Branch-and-Cut Algorithm} \label{sec:BC}
Branch-and-cut is a standard exact approach widely applied to interval min--max regret problems (see \citealp{FIMY15}, \citealp{MBMG07}, \citealp{PA13}, and \citealp{WIMY18}).
We introduce an implementation of this algorithm using Benders cuts~\eqref{form:mipBenCon} based on a basic branch-and-cut framework.

We define $P(X)$ as a relaxation of the \mbox{MIP-MMR-BIP}, in which set $X_0$ in constraints~\eqref{form:mipBenCon} is replaced by a subset $X\subseteq X_0$:
\begin{align}\label{MipMmrOpConWithX}
\lambda \ge \sum_{j=1}^{n}\left(c^+_j+\left(c^-_j-c^+_j\right)x_j\right)y_j	&&\forall y \in X.
\end{align}
For a feasible solution $(\boldsymbol{x^*}, \lambda^*)$ to problem $P(X)$, we define a \emph{slave problem} $Q(\boldsymbol{x^*})$ as
\begin{align}\label{Benders'Cut}
\max_{\boldsymbol{y}\in X_0}\sum_{j=1}^{n}\left(c^+_j+\left(c^-_j-c^+_j\right)x^*_j\right)y_j.
\end{align}
Let $\boldsymbol{y^*}$ be an optimal solution to $Q(\boldsymbol{x^*})$ and $q(\boldsymbol{y^*})$ be the corresponding objective value.
If $q(\boldsymbol{y^*})>\lambda^*$ holds, then the constraint~\eqref{form:mipBenCon} induced by $\boldsymbol{y^*}$ is violated by the current optimal solution $(\boldsymbol{x^*}, \lambda^*)$ of $P(X)$.

We solve $P(X)$ by a branch-and-cut framework.
When an integer-feasible solution candidate $\boldsymbol{x^*}$ has been identified,
we check a violated constraint~\eqref{form:mipBenCon} by solving the slave problem $Q(\boldsymbol{x^*})$.
If such a violated constraint is found, we discard the candidate solution $\boldsymbol{x^*}$, adding the solution $\boldsymbol{y^*}$ to $X$ and continuing the branch-and-cut with the new $P(X)$;
otherwise, we update the incumbent solution and prune the node.

The general branch-and-cut framework can be implemented in many ways.
In our implementation, we maintain the set $X$ for constraints~\eqref{MipMmrOpConWithX} as follows.
To prevent the LP relaxation $\bar{P}(X)$ at the root node from being unbounded,
we start with set $X = \{\boldsymbol{\widetilde{x}}\}$,
where $\boldsymbol{\widetilde{x}}$ is the optimal solution obtained by the fixed-scenario heuristic under the median-value scenario.
All cuts added to $X$ at any active node are used throughout the remaining computations, and so contribute to all other active nodes.

The branch-and-cut algorithm usually obtains feasible solutions during the search process,
and hence it can also be used as a heuristic by prematurely terminating its execution.
\section{Iterated Dual Substitution Heuristic} \label{sec:iDSH}
In this section, we propose an approach obtained by iterating a dual substitution heuristic.
\subsection{Dual Substitution Heuristic}\label{sec:DSH}
The \emph{dual substitution} (DS) heuristic is an algorithm based on a MIP model in which a subproblem is replaced by the dual counterpart of its linear relaxation.
\citet{K08} provided a general technique,
and similar ideas of using MIP models were used in other min--max regret problems (see \citealp{FIMY15}, \citealp{KPY01}, \citealp{KZ06}, \citealp{WIMY18}, and \citealp{YKP01}).
In this paper, we use it as a heuristic for the \mbox{MMR-BIP}.

For every fixed $\boldsymbol{x}$, the maximization problem over $\boldsymbol{y}$ in \eqref{form:mmrbipBasic} is a classical BIP that can be expressed as 
\begin{align}
\label{form:yObj}		\max\hspace{2mm}		&\sum^n_{j=1}c^{\sigma{(\boldsymbol{x})}}_jy_j					\\
\label{form:yInqCon}	\mathrm{s.t.}\hspace{4mm}	&\sum^n_{j=1}a_{ij}y_j\le b_{i}			&\forall i \in I,		\\
\label{form:yIteCon}							&y_j\in \{0,1\}						&\forall j \in J.
\end{align}
We consider a corresponding linear program, obtained by replacing the integrality constraint~\eqref{form:yIteCon} with
\begin{align}
\label{con:yCon}0\le y_j\le 1	&&\forall j \in J.
\end{align}

We define two types of dual variables: $u_i\ (i\in I)$ for constraints~\eqref{form:yInqCon} and $v_j\ (j\in J)$ for the $\le 1$ constraints in~\eqref{con:yCon}.
The dual of problem~\eqref{form:yObj}, \eqref{form:yInqCon}, \eqref{con:yCon} can be formulated as
\begin{align}
\label{form:yDualObj}		\min\hspace{2mm}		&\sum^m_{i=1}b_iu_i+\sum^n_{j=1}v_j\\
\label{form:yDualConS}	\mathrm{s.t.}\hspace{3mm}	&\sum^m_{i=1}a_{ij}u_i+v_j\ge c^{\sigma{(\boldsymbol{x})}}_j = c^+_j+\left(c^-_j-c^+_j\right)x_j	&\forall j \in J,	\\
												&u_i\ge0 																	&\forall i \in I,		\\
\label{form:yDualConE}								&v_j\ge0 																	&\forall j \in J.
\end{align}

We then embed the dual counterpart into \eqref{form:mmrbipBasic}.
By replacing $\boldsymbol{y}\in X_0$ with constraints~\eqref{form:yDualConS}--\eqref{form:yDualConE}, and by replacing the inner maximization problem~\eqref{form:yObj} with~~\eqref{form:yDualObj}, 
we obtain the following \emph{dual substitution model} (\mbox{D-MMR-BIP}):
\begin{align}
\label{form:idsObj}		\min\hspace{2mm}		&\sum^m_{i=1}b_iu_i+\sum^n_{j=1}v_j-\sum^n_{j=1}c^-_jx_j\\
						\mathrm{s.t.}\hspace{3mm}	&\sum^m_{i=1}a_{ij}u_i+v_j\ge c^+_j+\left(c^-_j-c^+_j\right)x_j	&\forall j \in J,	\\
												&u_i\ge0												&\forall i \in I,		\\
												&v_j\ge0												&\forall j \in J,	\\
\label{form:idsEnd}								&\boldsymbol{x}\in X_0.
\end{align}

The dual substitution heuristic exactly solves the \mbox{D-MMR-BIP} and outputs the obtained solution.
The \mbox{D-MMR-BIP} is not easier than the BIP, since it contains all the BIP constraints.
\begin{remark}
\label{WeakUpperBound}The optimal value of the \mbox{D-MMR-BIP} is an upper bound on the optimal value for the \mbox{MMR-BIP}.
\end{remark}

In addition, for any instance, a tighter upper bound can be obtained as follows.
\begin{remark}
The bound obtained by evaluating the maximum regret of any optimal solution to the \mbox{D-MMR-BIP} is at least as good as the optimal value to the \mbox{D-MMR-BIP}.
\end{remark}

Our computational experiments showed that the DS heuristic obtains better solutions than does the fixed-scenario heuristic for most instances, although it does not guarantee solution quality for some problems, including the generalized assignment problem \citep{WIMY18} and the multidimensional knapsack problem \citep{WIMY16}.
\subsection{Best Scenario}\label{sec:Bs}
Before describing the iterated dual substitution heuristic, we introduce some relevant properties of the BIP. 
For a solution $\boldsymbol{x}$, let us consider the scenario
\begin{align*}
\phi(\boldsymbol{x}) = 
\begin{cases}
c^+_j &\mathrm{if\ }x_j=1\\
c^-_j &\mathrm{otherwise},
\end{cases}
\end{align*}
opposite to the worst-case scenario~\eqref{def:worstSen},
which we call the \emph{best scenario} for solution $\boldsymbol{x}$.
\begin{lemma}\label{lem:bst1}
If $\boldsymbol{x^*}$ is an optimal solution of the BIP~\eqref{std:Obj}--\eqref{std:IteCon} under a scenario $s \in S$,
then $\boldsymbol{x^*}$ is also an optimal solution under its best scenario $\phi(\boldsymbol{x^*})$.
\end{lemma}
\begin{proof}
In $\phi(\boldsymbol{x^*})$ the $c_j$ value of the items selected (resp. not selected) in $\boldsymbol{x^*}$ can only increase (resp. decrease).
\end{proof}
\noindent Note that even if the original BIP is a minimization problem~\eqref{std:MinObj}, \eqref{std:InqCon}, \eqref{std:IteCon}, Lemma~\ref{lem:bst1} holds with a slight modification, that is, the best scenario of $\boldsymbol{x^*}$ becomes $\sigma(\boldsymbol{x^*})$  in~\eqref{def:worstSen}. 

\begin{lemma}\label{lem:bst2}
For any feasible pair of solutions $\boldsymbol{\bar{x}},\boldsymbol{\hat{x}}\in X_0$ of the BIP~\eqref{std:Obj}--\eqref{std:IteCon},
if
\begin{align}\label{Lem:ConMAX}
z^{\phi(\boldsymbol{\bar{x}})}(\boldsymbol{\hat{x}}) \ge z^{\phi(\boldsymbol{\bar{x}})}(\boldsymbol{\bar{x}})
\end{align}
holds, then we have $r_{\max}(\boldsymbol{\hat{x}})\le r_{\max}(\boldsymbol{\bar{x}})$.
\end{lemma}

\begin{proof}
From~\eqref{Lem:ConMAX} we obtain the following inequalities:
\begin{align}
\nonumber\sum_{\twoline{j:\,\bar{x}_j=1,}{\ \,\hat{x}_j=1}}c^+_j + \sum_{\twoline{j:\,\bar{x}_j=0,}{\ \,\hat{x}_j=1}}c^-_j &\ge \sum_{j:\,\bar{x}_j=1}c_j^+\\
\label{LemProofMAX:Con}\sum_{\twoline{j:\,\bar{x}_j=0,}{\ \,\hat{x}_j=1}}c^-_j &\ge \sum_{\twoline{j:\,\bar{x}_j=1,}{\ \,\hat{x}_j=0}}c^+_j.
\end{align}
From~\eqref{def:rmax} we get
\begin{align}
\label{LemProofMAX:rmax} r_{\max}(\boldsymbol{\hat{x}}) = z_*^{\sigma(\boldsymbol{\hat{x}})} - z^{\sigma(\boldsymbol{\hat{x}})}(\boldsymbol{\hat{x}}) = z_*^{\sigma(\boldsymbol{\hat{x}})} - \sum_{j:\,\hat{x}_j=1}c_j^-
= z_*^{\sigma(\boldsymbol{\hat{x}})} - \sum_{\twoline{j:\,\bar{x}_j=0,}{\ \,\hat{x}_j=1}}c_j^- - \sum_{\twoline{j:\,\bar{x}_j=1,}{\ \,\hat{x}_j=1}}c_j^-.
\end{align}
From~\eqref{LemProofMAX:Con} we obtain
\begin{align}
\label{LemProofMAX:rmax2} r_{\max}(\boldsymbol{\hat{x}}) &\le z_*^{\sigma(\boldsymbol{\hat{x}})} - \sum_{\twoline{j:\,\bar{x}_j=1,}{\ \,\hat{x}_j=0}}c_j^+ - \sum_{\twoline{j:\,\bar{x}_j=1,}{\ \,\hat{x}_j=1}}c_j^-
= z_*^{\sigma(\boldsymbol{\hat{x}})} - \sum_{\twoline{j:\,\bar{x}_j=1,}{\ \,\hat{x}_j=0}}\left(c_j^+ - c_j^-\right) - \sum_{j:\,\bar{x}_j=1}c_j^-.
\end{align}

Letting $x^*$ be an optimal solution to the BIP under $\sigma(\boldsymbol{\hat{x}})$, inequality~\eqref{LemProofMAX:rmax2} can be rewritten as
\begin{align*}
 r_{\max}(\boldsymbol{\hat{x}}) &\le  \sum_{\twoline{j:\,x^*_j=1,}{\ \,\hat{x}_j=1}}c_j^- + \sum_{\twoline{j:\,x^*_j=1,}{\ \,\hat{x}_j=0}}c_j^+ - \sum_{\twoline{j:\,\bar{x}_j=1,}{\ \,\hat{x}_j=0}}\left(c_j^+ - c_j^-\right) - \sum_{j:\,\bar{x}_j=1}c_j^-\\
 &=  \sum_{\twoline{j:\,x^*_j=1,}{\ \,\bar{x}_j=1}}c_j^- + 
\sum_{\twoline{j:\,x^*_j=1,}{\twoline{\ \,\hat{x}_j=0,}{\ \bar{x}_j=0}}}c_j^+ +
\sum_{\threeline{j:\,x^*_j=1,}{\ \,\hat{x}_j=1,}{\ \bar{x}_j=0}}c_j^- +
\sum_{\threeline{j:\,x^*_j=0,}{\ \,\hat{x}_j=0,}{\ \bar{x}_j=1}}\left(c_j^- - c_j^+\right) -
\sum_{j:\,\bar{x}_j=1}c_j^-\\
 &\le  \sum_{\twoline{j:\,x^*_j=1,}{\ \,\bar{x}_j=1}}c_j^- + 
\sum_{\twoline{j:\,x^*_j=1,}{\twoline{\ \,\hat{x}_j=0,}{\ \bar{x}_j=0}}}c_j^+ +
\sum_{\threeline{j:\,x^*_j=1,}{\ \,\hat{x}_j=1,}{\ \bar{x}_j=0}}c_j^+ -
\sum_{j:\,\bar{x}_j=1}c_j^-\\
 &= z^{\sigma(\boldsymbol{\bar{x}})}(\boldsymbol{x^*}) - \sum_{j:\,\bar{x}_j=1}c_j^- \le z_*^{\sigma(\boldsymbol{\bar{x}})} - \sum_{j:\,\bar{x}_j=1}c_j^- = r_{\max}(\boldsymbol{\bar{x}}).
\end{align*}
\end{proof}

It is not difficult to adapt the proof of the previous lemma to show that the property also holds if the original BIP is a minimization problem i.e.,
\begin{lemma}\label{lem:bst2-min}
For any feasible pair of solutions $\boldsymbol{\bar{x}},\boldsymbol{\hat{x}}\in X_0$ of problem~\eqref{std:MinObj}, \eqref{std:InqCon}, \eqref{std:IteCon},
if
\begin{align}\label{Lem:ConMIN}
z^{\sigma(\boldsymbol{\bar{x}})}(\boldsymbol{\hat{x}}) \le z^{\sigma(\boldsymbol{\bar{x}})}(\boldsymbol{\bar{x}})
\end{align}
holds, then we have $r_{\max}(\boldsymbol{\hat{x}})\le r_{\max}(\boldsymbol{\bar{x}})$.
\end{lemma}

\subsection{Iterated Dual Substitution}\label{sec:iDS}

Computational experiments on specific min--max regret optimization problems (see, e.g., \citealp{FIMY15}, \citealp{WIMY18}) have shown that the DS heuristic performs better than the fixed scenario algorithm in most cases, but it can, in some cases, produce unsatsfactory solutions. 
To overcome this, we propose an iterative method, which we call the \emph{iterated dual substitution} (iDS), that improves the performance of the DS heuristic through iterative application, 
excluding from the search space the solutions that were already checked in previous iterations.

We define $\hat{X}$ ($\subseteq X_0$) as the set of already obtained feasible solutions,
and denote by \mbox{D-MMR-BIP($\hat{X}$)} the problem obtained by adding constraints based on $\hat{X}$ to the \mbox{D-MMR-BIP}~\eqref{form:idsObj}--\eqref{form:idsEnd} in the way described below.
The iDS algorithm starts with an empty set $\hat{X}=\emptyset$.
It then solves the \mbox{D-MMR-BIP($\hat{X}$)}, obtaining a solution $\boldsymbol{\hat{x}}$, and evaluates its maximum regret $r_{\max}(\boldsymbol{\hat{x}})$.
Since the \mbox{D-MMR-BIP($\hat{X}$)} contains all constraints of the original BIP,
any feasible solution $\boldsymbol{\hat{x}}$ to the \mbox{D-MMR-BIP($\hat{X}$)} is also feasible to the \mbox{MMR-BIP}.
At each iteration, iDS adds $\boldsymbol{\hat{x}}$ to $\hat{X}$ and solves the updated \mbox{D-MMR-BIP($\hat{X}$)}.
This process is repeated until the search space of \mbox{D-MMR-BIP($\hat{X}$)} becomes empty, or a termination condition (here, a time limit) is satisfied.

To remove already checked feasible regions, we consider two kinds of constraints based on $\hat{X}$: a Hamming-distance constraint and a best-scenario constraint.

\subsubsection*{Hamming-Distance Constraint.}\label{sec:HDC}
We consider the constraint
\begin{align}
\label{DistCon}	&\sum_{j:\,\hat{x}_j=0}x_j+\sum_{j:\,\hat{x}_j=1}\left(1-x_j\right)\ge d	&\forall \boldsymbol{\hat{x}} \in \hat{X},
\end{align}
where $d$\ ($\ge$1) is an integer parameter. 
Any solution $x \in X_0$ that satisfies \eqref{DistCon} for a solution $\boldsymbol{\hat{x}}$ has a Hamming distance larger than or equal to $d$ from $\boldsymbol{\hat{x}}$.
Our aim is to find a good solution to the \mbox{D-MMR-BIP} that has a Hamming distance greater than or equal to $d$ for every solution in $\hat{X}$, and hence  
solutions within a distance less than $d$ from a solution in $\hat{X}$ are removed from the feasible region.


Note that Hamming-distance constraints~\eqref{DistCon} are also valid when the original BIP is a minimization problem.

\subsubsection*{Best-Scenario Constraint.}\label{sec:BSC}
Lemma~\ref{lem:bst2} indicates that a feasible solution $\boldsymbol{\hat{x}}$ dominates all solutions in
\begin{align}\label{dominate}
\left\{ \boldsymbol{x}\in X_0 \ \big|\  z^{\phi(\boldsymbol{x})}(\boldsymbol{\hat{x}}) \ge z^{\phi(\boldsymbol{x})}(\boldsymbol{x}) \right\}.
\end{align}
For every solution $\boldsymbol{\hat{x}} \in \hat{X}$, we consider a constraint for removing those solutions that are dominated by $\boldsymbol{\hat{x}}$, namely,
\begin{align}
\label{LemCon}
z^{\phi(\boldsymbol{x})}(\boldsymbol{\hat{x}}) < z^{\phi(\boldsymbol{x})}(\boldsymbol{x})&&\forall \boldsymbol{\hat{x}} \in \hat{X}.
\end{align}
Note that the constraint~\eqref{LemCon} of each $\boldsymbol{\hat{x}}$ excludes $\boldsymbol{\hat{x}}$ itself from the feasible region.
The linear expression of~\eqref{LemCon} is
\begin{align}
\nonumber\sum_{j:\,\hat{x}_j=1} \left( c^+_j x_j - c^-_j x_j + c^-_j \right) < \sum_{j=1}^n c^+_j x_j &&\forall \boldsymbol{\hat{x}} \in \hat{X}\\
\label{LenConLinear}\sum_{j:\,\hat{x}_j=1}c^-_j x_j + \sum_{j:\,\hat{x}_j=0} c^+_j x_j  > \sum_{j:\,\hat{x}_j=1} c^-_j &&\forall \boldsymbol{\hat{x}} \in \hat{X}.
\end{align}

Given a feasible solution set $\hat{X}$, best-scenario constraints~\eqref{LenConLinear} dominate Hamming-distance constraints when $d = 1$,
because Hamming-distance constraints with $d=1$ can only exclude all solutions in $\hat{X}$ from the search space.

If the original BIP is a minimization problem, the constraints corresponding to~\eqref{LenConLinear} become
\begin{align}
\label{LenConLinearMin}\sum_{j:\,\hat{x}_j=1}c^+_j x_j + \sum_{j:\,\hat{x}_j=0} c^-_j x_j  < \sum_{j:\,\hat{x}_j=1} c^+_j &&\forall \boldsymbol{\hat{x}} \in \hat{X}.
\end{align}
If all the values of $c^-_j$ and $c^+_j$ are integers, we can replace \eqref{LenConLinear} and~\eqref{LenConLinearMin}, respectively, with 
\begin{align*}
\sum_{j:\,\hat{x}_j=1}c^-_j x_j + \sum_{j:\,\hat{x}_j=0} c^+_j x_j  \ge \sum_{j:\,\hat{x}_j=1} c^-_j + 1 &&\forall \boldsymbol{\hat{x}} \in \hat{X},
\end{align*}
and
\begin{align*}
\sum_{j:\,\hat{x}_j=1}c^+_j x_j + \sum_{j:\,\hat{x}_j=0} c^-_j x_j  \le \sum_{j:\,\hat{x}_j=1} c^+_j - 1 &&\forall \boldsymbol{\hat{x}} \in \hat{X}.
\end{align*}

Given sufficient computation time and memory space, the iDS algorithm using best-scenario constraints is theoretically an exact approach for the \mbox{MMR-BIP},
because the newly excluded search space is non-empty in each iteration, and all excluded solutions are guaranteed to be no better than the candidate solution according to Lemma~\ref{lem:bst2}.

\subsection{Local Exact Subroutine}\label{sec:BCR}
Note that if we set $d = 1$, the iDS approach using Hamming-distance constraints at every iteration only excludes from the search space
those $\boldsymbol{\hat{x}}$ solutions that have already been checked during the search.
When $d \ge 2$, the iDS approach also removes from the search space unchecked solutions around $\boldsymbol{\hat{x}}$.
The unchecked space around a solution $\boldsymbol{\hat{x}}$ can be defined as
\begin{align}
\label{EXDistCon}
1\le \sum_{j:\,\hat{x}_j=0}x_j+\sum_{j:\,\hat{x}_j=1}\left(1-x_j\right)\le d-1.
\end{align}

When $d \ge 2$, we consider an option that exactly solves,
whenever a solution $\boldsymbol{\hat{x}}$ is obtained by solving \mbox{D-MMR-BIP($\hat{X}$)},
the \mbox{MMR-BIP} described in~\eqref{form:mmrbipBasic} with the additional constraint~\eqref{EXDistCon} only for the newly obtained $\boldsymbol{\hat{x}}$, which we call the \mbox{MMR-BIP}($\boldsymbol{\hat{x}}$).
Taking this option, the iDS algorithm using Hamming-distance constraints is guaranteed to find an exact optimal solution for the \mbox{MMR-BIP} when computation time and memory space are sufficient to run until the \mbox{D-MMR-BIP($\hat{X}$)} search space becomes empty.

The branch-and-cut approach described in Section~\ref{sec:BC} can be used to solve the \mbox{MMR-BIP}($\boldsymbol{\hat{x}}$).

\section{Test Problems} \label{sec:TP}
To test the performance of the proposed algorithms for binary integer programming problems,
 we selected four representative combinatorial optimization problems: two maximization problems (the knapsack problem and the multidimensional knapsack problem) and two minimization problems (the set covering problem and the generalized assignment problem).
All these problems are known to be NP hard, and the corresponding problems under min--max regret criteria are $\Sigma^p_2$ hard.
We describe the original problems in this section.
The corresponding problems under min--max regret criteria can be defined using the definitions in Sections~\ref{sec:MMRC}--\ref{sec:MIP}.
\subsection{Knapsack Problem} \label{sec:KP}
The \emph{knapsack problem} (KP) is defined as follows.
Given a set $J$ of $n$ items $\{1, 2, \ldots, n\}$,
each item $j$ having weight $a_j$ and value $c_j$,
choose items maximizing the total value such that the total weight is less than or equal to the knapsack capacity $b$.

The KP can be formulated over binary variables $x_j$,
indicating that item $j$ is chosen if and only if $x_j=1$:
\begin{align}
\max\ &\sum^n_{j=1}c_jx_j\\
\label{KP:CapCon}{\text s.\text t.}\ \ &\sum^n_{j=1}a_jx_j\le b\\
\label{KP:IteCon}&x_j\in \{0,1\} &\forall j \in J.
\end{align}

The KP has been thoroughly studied due to its simple structure and important practical applications (see~\citealp{KPP04}, and~\citealp{MT90}).
\subsection{Multidimensional Knapsack Problem} \label{sec:MKP}
The \emph{multidimensional knapsack problem} (MKP) is an extension of the KP.
Given a set $J$ of $n$ items $\{1, 2, \ldots, n\}$ and a set $I$ of $m$ multiple-resource restraints $\{1, 2, \ldots, m\}$,
choose items with maximum total value, subject to satisfying every resource restraint.
Choosing item $j$ provides value $c_j$ and consumes an amount $a_{ij}$ of resource in dimension $i$, 
with total resource availability (capacity) along dimension $i$ being $b_{i}$. 

The MKP can be formulated over binary variables $x_j$,
indicating that item $j$ is chosen if and only if $x_j=1$:
\begin{align}
\max\ &\sum^n_{j=1}c_jx_j\\
\label{MKP:CapCon}{\text s.\text t.}\ \ &\sum^n_{j=1}a_{ij}x_j\le b_{i} &\forall i \in I,\\
\label{MKP:IteCon}&x_j\in \{0,1\} &\forall j \in J.
\end{align}

The MKP is known to be computationally much harder than the KP, even for $m=2$ \citep[see][]{KS2010}.
The problem has been thoroughly studied over many decades due to its theoretical interest and its broad applications in several fields,
including cargo loading, cutting stock, bin-packing, finance, and management issues (see \citealp{LNAA18} and \citealp{PRP10}).

\subsection{Set Covering Problem} \label{sec:SCP}
The \emph{set covering problem} (SCP) is defined as follows.
Given a set $I$ of $n$ items $\{1, 2, \ldots, n\}$ and $m$ subsets of $I$, each subset $j$ having a cost $c_j$,
select subsets with minimum cost such that each item is covered by at least one selected subset. 
Letting $a_{ij}$ be a binary value that equals 1 when subset $j$ covers item $i$ and $0$ otherwise,
the SCP can be formulated over binary variables $x_j$,
indicating that subset $j$ is chosen if and only if $x_j=1$:
\begin{align}
\min\ &\sum^n_{j=1}c_jx_j\\
\label{SCP:CoverCon}{\text s.\text t.}\ \ &\sum^n_{j=1}a_{ij}x_j\ge 1 &\forall i \in I,\\
\label{IteCon}&x_j\in \{0,1\} &\forall j \in J.
\end{align}

The decision version of the SCP is one of Karp's 21 NP-complete problems \citep[see][]{K72}.
There are many real-world applications involving this type of problem, including facility location \citep[see][]{FAHHG12}, airline crew scheduling \citep[see][]{WJJASU16}, and vehicle routing \citep[see][]{CHT14}.
\subsection{Generalized Assignment Problem} \label{sec:GAP}
The \emph{generalized assignment problem} (GAP) is defined as follows.
Given a set $J$ of $n$ jobs $\{1, 2, \ldots, n\}$ and a set $I$ of $m$ agents $\{1, 2, \ldots, m\}$, 
look for a minimum cost assignment, subject to assigning each job to exactly one agent and 
satisfying a resource constraint for each agent. Assigning job $j$ to agent $i$ 
incurs a cost $c_{ij}$ and consumes an amount $a_{ij}$ of resource, 
where the total amount of the resource available for agent $i$ (the agent capacity) is $b_{i}$. 

A natural formulation of the GAP is defined over a two-dimensional binary variable $x_{ij}$,
indicating that job $j$ is assigned to agent $i$ if and only if $x_{ij}=1$:
\begin{align}
\min\ &\sum^m_{i=1}\sum^n_{j=1}c_{ij}x_{ij}\\
\label{GAP:CapCon}
{\text s.\text t.}\ \ &\sum^n_{j=1}a_{ij}x_{ij}\le b_{i} &\forall i \in I,\\
\label{GAP:SemiCon}
&\sum^m_{i=1}x_{ij}=1 &\forall j \in J,\\
\label{GAP:IteCon}
&x_{ij}\in \{0,1\}&\forall i \in I,\ \forall j \in J.
\end{align}

The GAP has many real-world applications,
including production planning, scheduling, timetabling, telecommunication, investment, and transportation (see \citealp{FJ81}, \citealp{R99}, and \citealp{WYI18}).
\section{Computational Experiments} \label{sec:CE}
Let \mbox{MMR-KP}, \mbox{MMR-MKP}, \mbox{MMR-SCP}, and \mbox{MMR-GAP} denote respectively the KP, the MKP, the SCP, and the GAP under the min--max regret criterion.
The following sections present computational results for these four resulting problems.
\subsection{Implementation Details} \label{sec:ID}
All experiments were carried out on a PC with a Xeon E-2144G CPU running at 3.60~GHz, with 64~GB memory.
All computations were performed on a single CPU core.
We used Gurobi version 8.1 to solve the mixed integer linear programming problems.
Detailed results for the four considered problems are presented in four appendices and are available online at \url{http://www.co.mi.i.nagoya-u.ac.jp/~yagiura/ids/}.

Concerning iDS using Hamming-distance constraints,
we used the constraints~\eqref{DistCon} with $d=1$ for all tested problems,
because preliminary experiments indicated that this value gives better results than do larger values of $d$.
All algorithms were given a time limit of 1~CPU hour per instance.
\subsection{Min--Max Regret Multidimensional Knapsack Problem} \label{sec:RMKP}
To the best of our knowledge, this is the first study on the \mbox{MMR-MKP}.
To generate \mbox{MMR-MKP} instances, we used the classical MKP benchmarks of~\citet{CB98} with $(m,n)\in\{(5, 100),(10,100),(5, 250)\}$.
There are 30 instances for each $(m, n)$ pair.
For each MKP instance, we generated the extremes of the value intervals $c^-_j$ and $c^+_j$ as random integers uniformly distributed in $[(1-\delta)c_j, c_j]$ and $[c_j, (1+\delta)c_j]$, respectively, with $\delta \in \{0.1, 0.2, 0.3\}$.
The 270 generated instances are available at \url{http://www.co.mi.i.nagoya-u.ac.jp/~yagiura/mmr-mkp/}.
 
\input{tableMKP_new}

Table~\ref{tbl:mkp} shows the results of the algorithms described in Sections~\ref{sec:CH} and~\ref{sec:iDSH}:
the branch-and-cut algorithm (``B\&C''),  the fixed-scenario algorithm (``Fix''), the DS algorithm (``DS''), the iDS algorithm with Hamming-distance constraints (``\mbox{iDS-H}''), and iDS with the best-scenario constraints (``\mbox{iDS-B}'').
A row ``\mbox{$ww$-$xxx$-$yy$}'' refers to the 30 instances generated  with $m=ww$, $n=xxx$, and $100 \delta = yy$. For the branch-and-cut algorithm,
the entries show
the average CPU time in seconds required to obtain the best solution (``time''),
the average percentage gap between the solution value and the best lower-bound value (``\%gap''),
the number of instances solved to proven optimality (``\#opt''),
and the number of instances for which the solution value obtained was the best among the tested algorithms (``\#best''). 
Since the fixed-scenario approach under the median-value scenario is a \mbox{2-approximation} algorithm (see Lemma~\ref{LemMidPnt}), half of its solution value is a valid lower bound.
The lower-bound value used to compute the average percentage optimality gap was the best between such value and the lower bound produced by the branch-and-cut.
For \mbox{iDS-H} and \mbox{iDS-B}, we provide the average number of iterations (``iter'') required to obtain the best solution.
Values in bold indicate that the corresponding algorithm(s) obtained the best results among the tested algorithms.
The final line of the table reports the overall average and total values over the 270 instances.

The branch-and-cut algorithm exactly solved
63 instances with $m=5$ and $n=100$,
32 with $m=10$ and $n=100$,
and 0 with $m=5$ and $n=250$.
In total, it only solved 95 out of 270 instances to proven optimality, and failed to obtain solutions with low optimality gaps for all instances with $m=5$ and $n=250$.
Focusing on instances with same $m$ and $n$ values,
we observe that it becomes hard to obtain optimal solutions or even good lower bounds for instances with large $\delta$ values,
resulting in large ``gap" column values.
This indicates that an \mbox{MMR-MKP} instance with the tested size should be preferably handled through heuristic algorithms.

The fixed-scenario approach under the median-value scenario obtained solutions within 25~seconds on average even for the most hard instances with $m = 5$ and $n=250$.
It produced solutions with good optimality gaps for instances with $m=5$ and $n=100$.

The DS algorithm obtained all best solutions within 20~minutes on average.
However, for 37 out of 60 instances with $(m,n)=(5,250)$ and $\delta\ge 0.2$, the DS algorithm failed to complete within the time limit (see appendix~\ref{Asec:MKP}).


By comparing the five heuristics (the branch-and-cut algorithm can be also used as a heuristic by imposing a time limit),
we can observe that the fixed-scenario approach under the median-value scenario obtained solutions within 10~seconds on average for all tested instances.
Concerning the solution quality,
the two iDS algorithms obtained the best solution values for 258 out of 270 instances,
while branch-and-cut, fixed-scenario approach, and the DS algorithm obtained the best solution values for 116, 125, and 216 instances, respectively.
The iDS algorithms obtained optimal values for all instances with proven optimality,
and they dominated the other algorithms for all instances with $m=5$ and $n=100$.
For instances where the branch-and-cut and iDS algorithms obtained solutions with the same value,
the iDS algorithms used smaller CPU times.
However, the two iDS algorithms failed to obtain the best solution for 11 instances with $(m,n)=(5,250)$ and $\delta\ge 0.2$.
These failures occurred because the computation of each iteration was very expensive, and even the first iteration was not completed within the time limit in most cases.

We finally observe that there was no clear winner between \mbox{iDS-B} and \mbox{iDS-H}. They always obtained the best solutions at the same iteration with similar CPU times.

We refer to appendix~\ref{Asec:MKP} for more details on the computational results for the \mbox{MMR-MKP}.

\subsection{Min--Max Regret Knapsack Problem} \label{sec:RKP}
For the \mbox{MMR-KP}, 
we used the benchmark instances generated by \citet{FIMY15}
(available at \url{http://www.or.deis.unibo.it/research_pages/ORinstances/MRKP_instances.zip}\,).
They first generated 18 KP instances of nine types using a parameter $\bar{R}\in \{1000, 10000\}$ as follows (u.d. stands for ``uniformly distributed''):
\begin{enumerate}
    \item{\bf Uncorrelated}: $a_j$ and $c_j$ are random integers u.d. in $[1, \bar{R}]$.
    \item{\bf Weakly correlated}: $a_j$ is a random integer u.d. in $[1, \bar{R}]$, $c_j$ is a random integer u.d. in $[a_j-\bar{R}/10, a_j+\bar{R}/10]$ so that $c_j \ge 1$.
    \item{\bf Strongly correlated}: $a_j$ is a random integer u.d. in $[1, \bar{R}]$, $c_j=a_j+\bar{R}/10$.
    \item{\bf Inverse strongly correlated}: $c_j$ is a random integer u.d. in $[1, \bar{R}]$, $a_j=c_j+\bar{R}/10$.
    \item{\bf Almost strongly correlated}: $a_j$ is a random integer u.d. in $[1, \bar{R}]$, $c_j$ is a random integer u.d. in $[a_j+\bar{R}/10-\bar{R}/500, a_j+\bar{R}/10+\bar{R}/500]$.
    \item{\bf Subset sum}: $a_j$ is a random integer u.d. in $[1, \bar{R}]$, $c_j=a_j$.
    \item{\bf Even--odd subset sum}: $a_j$ is an even random integer u.d. in $[1, \bar{R}]$, $c_j=a_j$, $b$ is odd.
    \item{\bf Even--odd strongly correlated}: $a_j$ is an even random integer u.d. in $[1, \bar{R}]$, $c_j=a_j+\bar{R}/10$, $b$ is odd.
    \item{\bf Uncorrelated with similar weights}: $a_j$ is a random integer u.d. in $[100\bar{R},100\bar{R}+\bar{R}/10]$, $c_j$ is a random integer u.d. in $[1, \bar{R}]$.
\end{enumerate}
Then, for each KP instance, 27 \mbox{MMR-KP} instances were generated by considering all combinations of
\begin{itemize}
    \item $n\in \{50,60,70\}$;
    \item $b= \lfloor\gamma\sum_{j=1}^na_j\rfloor$, with tightness ratio $\gamma\in \{0.45, 0.5, 0.55\}$ ($b$ was increased by 1 if its value was even for classes 7 and 8);
    \item $c^-_j$ and $c^+_j$ set as random integers u.d. in $[(1-\delta)c_j, c_j]$ and $[c_j, (1+\delta)c_j]$, respectively, with $\delta \in \{0.1, 0.2, 0.3\}$,
\end{itemize}
obtaining in total 486 \mbox{MMR-KP} instances.

\input{tableKP_new}

Table~\ref{tbl:kp} shows the results obtained by DS, \mbox{iDS-H} and \mbox{iDS-B}.
Each row of the table refers to the 54 instances of the same type.
The ``Best Known'' values (optimality gaps and number of instances solved to proven optimality) are the values obtained from the three heuristic algorithms and the three exact algorithms by \citet{FIMY15}.

The entries ``\%gap,'' ``\#opt,'' ``time,'' and ``iter'' provide the same information as in Table~\ref{tbl:mkp}.
In columns ``\#b(w),'' the value outside (resp. inside) parentheses shows the total number of instances for which the objective function value was better (resp. worse) than the best known solution value.
The final line reports the overall average and total values over the 486 instances.

For all 108 instances of Types~1 and 9, both \mbox{iDS-H} and \mbox{iDS-B} obtained optimal solutions within 0.01~seconds on average.
For the 162 instances of Types~3, 4, and 6,
both \mbox{iDS-H} and \mbox{iDS-B} obtained solutions of value equal to or better than the best known solution value, improving the best known solution values for 4 instances.
For instances of Types~2 and 5, both iDS algorithms very quickly obtained solutions with objective values equal to the best known solution values for most instances (106 out of 108),
with \mbox{iDS-H} updating one best known objective value.
For Type~7 instances, the iDS algorithms updated the best known objective values for 14 out of 54 instances,
although they obtained worse solution values for 4 instances.
For no Type~8 instance the iDS algorithms improved the best known solution value.
We next evaluate the iDS algorithms by comparing them with the results of the DS algorithm.
The DS updated 4 best known objective values, but obtained worse solutions for 103 instances.
Both \mbox{iDS-H} and \mbox{iDS-B} improved most of these worse solutions (86 out of 103),
and \mbox{iDS-H} (resp. iDS-B) further updated the best known objective values for 15 (resp. 14) more instances.

In total, the iDS algorithms updated the best known objective values for 19 instances,
and obtained worse solutions for 16 instances.
They obtained optimal solutions for most instances (235 out of 236, see appendix~\ref{Asec:KP}) whose optimal solutions were known.
These results suggest that the iDS algorithms can be used as a heuristic for quickly obtaining high-quality solutions for the \mbox{MMR-KP}, except for Type~8 instances.

Comparing the two iDS algorithms, on the basis of the detailed results in appendix~\ref{Asec:KP},
it turns out that \mbox{iDS-B} was the winner (better solution or solution of the same value in fewer iterations) in 25 instances, while \mbox{iDS-H} won only once.
For Type~6 and 7 instances in particular, \mbox{iDS-B} won in 23 instances.

\subsection{Min--Max Regret Set Covering Problem} \label{sec:RSCP}
For the \mbox{MMR-SCP}, 
we used the benchmark instances of \citet{PA13},
which were generated from 25 benchmark instances for the classical SCP, namely instance sets~4, 5, and 6 in the OR-Library \citep[see][]{B90}. 
Instance set 4 contains 10 instances with $n=1000$ and $m=200$, instance set 5 contains 10 instances with $n=2000$ and $m=200$, and instance set 6 contains 5 instances with $n=1000$ and $m=200$.
In all cases, the original cost values range between 1 and 100. The total number of classical SCP instances is 25.
The extremes of the cost intervals $c^-_j$ and $c^+_j$ for each instance family of the \mbox{MMR-SCP} were set according to the following three types:
\begin{description}
\item[\quad{\bf B}:] the extremes of cost intervals $c^-_j$ and $c^+_j$ were generated as random integers u.d. in $[(1-\delta)c_j, c_j]$ and $[c_j, (1+\delta)c_j]$, respectively, with $\delta \in \{0.1, 0.3, 0.5\}$,
where $c_j$ are the original costs of the corresponding SCP instances from the OR-Library.
The total number of Type-B instances is thus 75;
\item[\quad{\bf M}:] for every $j\in N$, $c^+_j$ was generated as a random integer u.d. in $[0, 1000]$, and $c^-_j$ as a random integer u.d. in $[0, c_j^+]$.
Three instances were generated from each original SCP instance, resulting in a total of 75 Type-M instances;
\item[\quad{\bf K}:] for every $j\in N$, $c^-_j$ was generated as a random integer u.d. in $[0, 1000]$, and $c^+_j$ as a random integer u.d. in $[c_j^-, c_j^++1000]$.
Three instances were generated from each original SCP instance, producing in total 75 Type-K instances.
\end{description}

\input{tableSCP_new}

Table~\ref{tbl:scp} shows the results obtained by the DS and iDS algorithms for each instance group.
A row ``\mbox{$Tx$}'' refers to all Type-$T$ instances whose corresponding SCP instance is in family $x$ from the \mbox{OR-Library}.
The ``Best Known'' values come from the algorithms by \citet{PA13}.
All notations are the same as in Table~\ref{tbl:kp}.

For all 150 Type~B and M instances, both \mbox{iDS-H} and \mbox{iDS-B} required less than 10~seconds on average.
For all 75 Type~K instances,
both algorithms obtained solutions with objective function value equal to or better than the best known solution value,
and they updated the best known solution values for 11 instances.
Most best solutions obtained by the iDS algorithms were found in early iterations.
For the 159 instances whose optimal values are known, both algorithms provided all optimal solutions.

For 58 instances for which the best solution values obtained by the DS algorithm were worse than the best known solution values,
the iDS algorithms improved all the solutions in or after the second iteration. 
The detailed results in appendix~\ref{Asec:SCP}
show that \mbox{iDS-B} is the winner for 26 out of 64 instances whose best solutions were obtained in or after the second iteration,while both iDS algorithms show the same performance for the other instances.

\subsection{Min--Max Regret Generalized Assignment Problem} \label{sec:RGAP}
For the \mbox{MMR-GAP}, 
we used the benchmark instances by \citet{WIMY18},
available at \url{http://www.co.mi.i.nagoya-u.ac.jp/~yagiura/mmr-gap/}\,.
The instances were generated from the classical GAP benchmark instances (see \citealp{CB97}, and \citealp{LKGG95})
of the following types:
\begin{description}
\item[\quad{\bf A}:] $\forall i,j$, $a_{ij}$ is a random integer u.d. in $[5,25]$, $c_{ij}$ is a random integer u.d. in $[10,50]$, and $b_i=0.6(n/m)15+0.4\gamma$, where $\gamma=\max_{i\in I}\sum_{j\in J,\,\theta_j=i}a_{ij}$ and $\theta_j = \min \{i\mid c_{ij}\le c_{kj}, \forall k\in I\}$.
\item[\quad{\bf B}:] $a_{ij}$ and $c_{ij}$ as for Type A; $b_i$ is set to $70\%$ of the Type A value.
\item[\quad{\bf C}:] $a_{ij}$ and $c_{ij}$ as for Type A; $b_i=0.8\sum^n_{j=1}a_{ij}/m$. 
\item[\quad{\bf E}:] $\forall i,j$, $a_{ij}=1-10\ln e_2$ ($e_2$ is a random number u.d. in $(0,1]$), $c_{ij}=1000/a_{ij}-10e_3$ ($e_3$ is a random number u.d. in $[0,1]$); $b_i=0.8\sum^n_{j=1}a_{ij}/m$.
\end{description}
The GAP instances were generated with a number of agents $m\in\{5,10\}$ and a number of jobs $n\in\{40,80\}$.
The cost interval extremes $c^-_{ij}$ and $c^+_{ij}$ were generated as random integers u.d. in $[(1-\delta)c_{ij}, c_{ij}]$ and $[c_{ij}, (1+\delta)c_{ij}]$, respectively, with $\delta \in \{0.10, 0.25, 0.50\}$.
By generating 5 instances for each $\delta$, we obtained 240 instances in total.

\input{tableGAP_new}

Table~\ref{tbl:gap} shows the results
for the \mbox{MMR-GAP} for each instance group.
A row denoted ``\mbox{$Txx$}'' includes 30 Type-$T$ instances whose number of agents is $xx$.
The ``Best Known'' values are taken from \citet{WIMY18}.
All notations are the same as in Table~\ref{tbl:kp}.

For most instances of Types A, B, and C (179 out of 180), both \mbox{iDS-H} and \mbox{iDS-B} obtained solutions whose objective values were equal to or better than the best known objective values.
Both algorithms updated 17 best known objective values out of 180 instances.
For the 60 instances of Type E, \mbox{iDS-B} updated 9 best known objective values, while \mbox{iDS-H} updated 8,
but both obtained worse solutions for 19 instances.
Note that in most of these 19 instances, the iDS algorithms failed to complete the first iteration (see appendix~\ref{Asec:GAP}, Table~\ref{tbl:gap-e}),
indicating that the performance of iDS algorithms depends on how fast the classical DS can exactly solve the \mbox{D-MMR-BIP}.

The iDS algorithms improved solutions for 38 out of 58 instances for which the best solution values obtained by the DS were worse than the best known solution values,
and \mbox{iDS-H} (resp. iDS-B) further updated the best known objective values for 22 (resp. 23) more instances than the DS.

A similar observation can be made for the \mbox{MMR-GAP}, namely that most best solutions obtained by the iDS algorithms were found in early iterations.
For the 122 instances whose optimal values are known, the iDS algorithms provided all optimal solutions.
Thus, similarly to what happens for the \mbox{MMR-MKP}, \mbox{MMR-KP}, and \mbox{MMR-SCP} cases,
the iDS algorithms provide a good heuristic for quickly obtaining optimal or near-optimal solutions for \mbox{MMR-GAP}, especially for Type A, B, and C instances.

By comparing \mbox{iDS-B} and \mbox{iDS-H} through the results in~appendix~\ref{Asec:GAP}, 
focusing on the 60 instances whose best solutions were obtained in or after the second iteration,
it turns out that
\mbox{iDS-B} outperformed \mbox{iDS-H} for 9 instances, while \mbox{iDS-H} was superior for only 2 instances,
suggesting that \mbox{iDS-B} is a better choice than \mbox{iDS-H} for the \mbox{MMR-GAP}.

\section{Conclusions}
We studied a robust version of the binary integer programming problem (BIP) called the min--max regret BIP.
We proposed a new heuristic framework, the iterated dual substitution (iDS) algorithm.
In the iDS approach, we considered two kinds of constraints to exclude from the search space already checked solutions:
constraints pruning checked solutions based on Hamming distance and constraints using a lemma on dominance between solutionsbased on the best scenario.
We tested the two resulting heuristics (\mbox{iDS-H} and \mbox{iDS-B}) on four representative problems under the min--max regret criterion:
the knapsack problem, the multidimensional knapsack problem, the generalized assignment problem, and the set covering problem.

For the min--max regret multidimensional knapsack problem, we compared the proposed algorithms with two classical approaches:
a branch-and-cut algorithm and a fixed-scenario algorithm.
Computational results showed that the iDS algorithms outperformed in terms of solution quality the other algorithms for most tested instances.
No clear winner emerged between \mbox{iDS-H} and \mbox{iDS-B}.

For the min--max regret knapsack problem,
the min--max regret generalized assignment problem,
and the min--max regret set covering problem,
we compared the iDS algorithms with state-of-the-art results.
For all these problems, 
the iDS algorithms successfully updated records for a number of benchmark instances,
and \mbox{iDS-B} exhibited better performance than \mbox{iDS-H}.

Overall, the computational results for all tested problems suggest that
the iDS approach is an effective heuristic for binary integer programming problems.

\section*{Acknowledgement}
This work was supported by JSPS KAKENHI [Grant No.15K12460, 15H02969, 19K11843, 20H02388],
and by the Air Force Office of Scientific Research under Award number FA8655-20-1-7012.

\bibliographystyle{abbrvnat}  
\bibliography{references}

\newpage
\begin{appendices}
\numberwithin{table}{section}
\section{Detailed results for the Min--Max Regret Multidimensional Knapsack Problem}\label{Asec:MKP}

Tables~\ref{tbl:mkp-05100}--\ref{tbl:mkp-05250} show the results of the algorithms
including
the branch-and-cut algorithm (``B\&C''),
the fixed-scenario algorithm (``Fix''),
the DS algorithm (``DS''),
iDS using Hamming-distance constraints (``\mbox{iDS-H}''),
and iDS using best-scenario constraints (``\mbox{iDS-B}'').
An \mbox{MMR-MKP} instance denoted by ``\mbox{$wwxxxyy$-$zz$}'' is generated from the $zz$th MKP instance with $m=ww$, $n=xxx$, $100\delta=yy$.
Concerning the branch-and-cut algorithm,
each entry shows the best obtained solution value (``obj''), the CPU time in seconds required to obtain that solution (``time''),
the lower-bound value obtained within the time limit (``LB''), and the optimality gap as a percentage (``\%gap''), that is, the gap between the solution value and the best lower-bound values.
Since the fixed-scenario approach under the median-value scenario is a \mbox{2-approximation} algorithm, half of its solution value is a valid lower bound.
The lower-bound value used to compute the percentage optimality gap was the best between such value and the lower bound produced by branch-and-cut.
For \mbox{iDS-H} and \mbox{iDS-B}, we provide the number of iterations (``iter'') required to obtain the best solution.
Values in bold signify that the corresponding algorithm(s) obtained the smallest percentage optimality gap among the tested algorithms.

\input{tableMKP}

\section{Detailed results for the Min--Max Regret Knapsack Problem}\label{Asec:KP}
Tables~\ref{tbl:kp-1}--\ref{tbl:kp-9} show
the results of the DS algorithm (``DS'') and iDS algorithms using Hamming-distance constraints (``\mbox{iDS-H}'') and best-scenario constraints (``\mbox{iDS-B}'') for the \mbox{MMR-SCP} for each instance type.
The \mbox{MMR-KP} instances are denoted by ``\mbox{$v$-$ww$-$xx$-$yy$-$zz$},'' where $\textrm{type}=f$, $n=ww$, $\bar{R}/1000=xx$, $100\gamma=yy$, and $100\delta=zz$.
Best known lower-bound values (``LB'') and solution values (``UB'') are the best results as obtained from three heuristic algorithms and three exact algorithms by \citet{FIMY15}.
Concerning the best solution obtained by DS and iDS for each instance,
the table shows its objective value (``obj''), required CPU time in seconds (``time''), and optimality gap as a percentage (``\%gap''),
that is, the gap between the solution value and the best known lower-bound value (``LB'').
For the best solution obtained by iDS, the table shows iteration index (``iter'') other than ``obj,'' ``time'' and ``gap.''
Values in ``obj'' columns marked by ``$\downarrow$'' (or ``$\uparrow$'') indicate instances
whose objective function obtained by the proposed algorithm were better (or worse) than the best known solution values in ``UB'' columns.
Bold values in ``obj'' columns indicate better objective values obtained within the time limit (3600 seconds) between \mbox{iDS-H} and \mbox{iDS-B}.
Bold values in ``iter'' columns signify that fewer iterations were needed when \mbox{iDS-H} and \mbox{iDS-B} obtained solutions with the same objective value in ``obj'' columns.

\input{tableKP}

\section{Detailed results for the Min--Max Regret Set Covering Problem}\label{Asec:SCP}
Tables~\ref{tbl:scp-b}--\ref{tbl:scp-k} show
the results of the DS algorithm (``DS'') and iDS algorithms using Hamming-distance constraints (``\mbox{iDS-H}'') and best-scenario constraints (``\mbox{iDS-B}'') for \mbox{MMR-SCP} for each instance type.
An \mbox{MMR-SCP} instance denoted by ``\mbox{B$xyyzz$}'' indicates a Type-B instance whose corresponding SCP instance is the $yy$th instance in family $x$ from the \mbox{OR-Library} with $zz=100\delta$,
while ``\mbox{M$xyy$-$z$}'' (or ``\mbox{K$xyy$-$z$}'') stands for the $z$th Type-M (or Type-K) instance whose corresponding SCP instance is the $yy$th instance in family $x$.
The best known lower-bound value (``LB'') and solution value (``UB'') are the results obtained from three heuristic algorithms and three exact algorithms by \citet{PA13}.
The notations ``obj,'' ``time,'' ``ite,'' ``gap,'' ``LB,'' and ``$\downarrow$,''
as well as the bold values in columns ``obj'' and columns ``ite,'' are the same as in Tables~\ref{tbl:kp-1}--\ref{tbl:kp-9} for the \mbox{MMR-KP}.

\input{tableSCP}

\section{Detailed results for the Min--Max Regret Generalized Assignment Problem}\label{Asec:GAP}
Tables~\ref{tbl:gap-a}--\ref{tbl:gap-e} show
the results of the DS algorithm (``DS'') and iDS algorithms using Hamming-distance constraints (``\mbox{iDS-H}'') and best-scenario constraints (``\mbox{iDS-B}'') for the \mbox{MMR-GAP} for each instance type.
An \mbox{MMR-GAP} instance denoted by ``\mbox{$Txxyyzz$-$i$}'' indicates the $i$th instance of the $(T, xx, yy, zz)$ combination, where $\textrm{type}=T$, $m=xx$, $n=yy$, and $100\delta=zz$.
The best known lower-bound value (``LB'') and solution value (``UB'') are the results obtained from two heuristic algorithms and two exact algorithms by \citet{WIMY18}.
The notations ``obj,'' ``time,'' ``ite,'' ``gap,'' ``LB,'' ``$\downarrow$,'' and ``$\uparrow$,''
as well as the bold values in columns ``obj'' and ``ite,'' are the same as those in Tables~\ref{tbl:kp-1}--\ref{tbl:kp-9} for the \mbox{MMR-KP}.

\input{tableGAP}

\end{appendices}
\end{document}

%% file: tableMKP_new.tex
\begin{table}[htbp]
\centering \fontsize{6.5pt}{11pt}\selectfont
\setlength{\tabcolsep}{3pt}
\caption{MMR-MKP results}
\label{tbl:mkp}
\begin{tabular}{ll*{22}{r}}
\hline
                             && \multicolumn{4}{c}{B\&C}                  && \multicolumn{3}{c}{Fix}                  && \multicolumn{3}{c}{DS}                  && \multicolumn{4}{c}{iDS-H}                                                                  && \multicolumn{4}{c}{iDS-B}                                                                      \\ \cline{3-6} \cline{8-10} \cline{12-14} \cline{16-19} \cline{21-24}  
\multicolumn{1}{c}{instances} && \multicolumn{1}{c}{time} & \multicolumn{1}{c}{\%gap} & \multicolumn{1}{c}{\#opt} & \multicolumn{1}{c}{\#best} && \multicolumn{1}{c}{time} & \multicolumn{1}{c}{\%gap} & \multicolumn{1}{c}{\#best} && \multicolumn{1}{c}{time} & \multicolumn{1}{c}{\%gap} & \multicolumn{1}{c}{\#best} &&
\multicolumn{1}{c}{time}  & \multicolumn{1}{c}{iter} & \multicolumn{1}{c}{\%gap} & \multicolumn{1}{c}{\#best} && \multicolumn{1}{c}{time} & \multicolumn{1}{c}{iter} & \multicolumn{1}{c}{\%gap} & \multicolumn{1}{c}{\#best}  \\ \bhline{1.5pt}
05-100-10	&&	109.40 	&	\textbf{0.00} 	&	30	&	\textbf{30}	&&	1.17 	&	3.92 	&	13	&&	1.68 	&	1.55 	&	18	&&	3.71 	&	1.77 	&	\textbf{0.00} 	&	\textbf{30}	&&	3.74 	&	1.77 	&	\textbf{0.00} 	&	\textbf{30}	\\
05-100-20	&&	755.31 	&	2.32 	&	26	&	28	&&	0.92 	&	3.49 	&	18	&&	3.58 	&	2.37 	&	25	&&	5.40 	&	1.27 	&	\textbf{2.18} 	&	\textbf{30}	&&	5.32 	&	1.27 	&	\textbf{2.18} 	&	\textbf{30}	\\
05-100-30	&&	1415.65 	&	11.71 	&	7	&	18	&&	0.86 	&	11.85 	&	15	&&	13.14 	&	10.84 	&	19	&&	24.63 	&	1.53 	&	\textbf{10.62} 	&	\textbf{30}	&&	24.95 	&	1.53 	&	\textbf{10.62} 	&	\textbf{30}	\\
10-100-10	&&	1603.19 	&	9.79 	&	24	&	26	&&	8.88 	&	12.61 	&	14	&&	11.09 	&	12.20 	&	23	&&	33.56 	&	1.77 	&	\textbf{8.53} 	&	\textbf{30}	&&	32.26 	&	1.77 	&	\textbf{8.53} 	&	\textbf{30}	\\
10-100-20	&&	2609.02 	&	34.80 	&	6	&	7	&&	7.70 	&	31.39 	&	9	&&	21.12 	&	31.54 	&	20	&&	63.05 	&	1.57 	&	\textbf{29.83} 	&	\textbf{30}	&&	61.93 	&	1.57 	&	\textbf{29.83} 	&	\textbf{30}	\\
10-100-30	&&	1897.73 	&	37.69 	&	2	&	6	&&	7.67 	&	34.45 	&	13	&&	113.41 	&	33.63 	&	20	&&	343.83 	&	1.70 	&	\textbf{33.42} 	&	\textbf{29}	&&	350.37 	&	1.70 	&	\textbf{33.42} 	&	\textbf{29}	\\
05-250-10	&&	2752.21 	&	58.59 	&	0	&	1	&&	24.16 	&	45.10 	&	14	&&	181.55 	&	46.94 	&	21	&&	238.91 	&	1.47 	&	\textbf{44.40} 	&	\textbf{30}	&&	238.16 	&	1.47 	&	\textbf{44.40} 	&	\textbf{30}	\\
05-250-20	&&	2379.45 	&	59.58 	&	0	&	0	&&	17.98 	&	43.43 	&	13	&&	2197.58 	&	43.33 	&	24	&&	2258.67 	&	1.17 	&	\textbf{43.21} 	&	\textbf{26}	&&	2255.03 	&	1.17 	&	\textbf{43.21} 	&	\textbf{26}	\\
05-250-30	&&	2550.39 	&	53.29 	&	0	&	0	&&	17.66 	&	39.94 	&	16	&&	3353.68 	&	\textbf{39.82} 	&	\textbf{23}	&&	3353.68 	&	1.00 	&	\textbf{39.82} 	&	\textbf{23}	&&	3353.68 	&	1.00 	&	\textbf{39.82} 	&	\textbf{23}	\\
\hline																																					
	&&	1785.82 	&	29.75 	&	95 	&	116 	&&	9.67 	&	25.13 	&	125 	&&	655.20 	&	24.69 	&	216 	&&	702.83 	&	1.47 	&	\textbf{23.56} 	&	\textbf{258} 	&&	702.83 	&	1.47 	&	\textbf{23.56} 	&	\textbf{258} 	\\
\bhline{1pt}
\end{tabular}
\end{table}

%% file: tableKP_new.tex
\begin{table}[htbp]
\centering
\fontsize{9pt}{10.5pt}\selectfont
\setlength{\tabcolsep}{4.5pt}
\caption{MMR-KP results}
\label{tbl:kp}
\begin{tabular}{cl*{16}{r}}
\hline
                             && \multicolumn{2}{c}{Best Known}                  && \multicolumn{3}{c}{DS}                                                                  && \multicolumn{4}{c}{iDS-H}                                                                  && \multicolumn{4}{c}{iDS-B}                                                                      \\ \cline{3-4} \cline{6-8} \cline{10-13} \cline{15-18} 
\multicolumn{1}{c}{instances} && \multicolumn{1}{c}{\%gap} & \multicolumn{1}{c}{\#opt} && \multicolumn{1}{c}{time} & \multicolumn{1}{c}{\%gap} & \multicolumn{1}{c}{\#b(w)} && \multicolumn{1}{c}{time} & \multicolumn{1}{c}{iter} & \multicolumn{1}{c}{\%gap} & \multicolumn{1}{c}{\#b(w)} && \multicolumn{1}{c}{time} & \multicolumn{1}{c}{iter} & \multicolumn{1}{c}{\%gap} & \multicolumn{1}{c}{\#b(w)} \\ \hline
Type 1	&&	0.000 	&	54	&&	0.01 	&	0.262 	&	0 (5)&&	0.01 	&	1.17 	&	0.000 	&	0 (0)&&	0.01 	&	1.15 	&	0.000 	&	0 (0)\\
Type 2	&&	1.307 	&	48	&&	0.05 	&	1.604 	&	0 (10)&&	0.06 	&	1.22 	&	1.423 	&	0 (1)&&	0.06 	&	1.19 	&	1.423 	&	0 (1)\\
Type 3	&&	11.239 	&	18	&&	0.27 	&	11.276 	&	0 (9)&&	0.42 	&	1.28 	&	11.237 	&	1 (0)&&	0.40 	&	1.28 	&	11.237 	&	1 (0)\\
Type 4	&&	9.895 	&	18	&&	0.21 	&	9.964 	&	0 (10)&&	0.27 	&	1.24 	&	9.895 	&	0 (0)&&	0.26 	&	1.24 	&	9.895 	&	0 (0)\\
Type 5	&&	8.151 	&	24	&&	0.19 	&	8.407 	&	0 (16)&&	0.30 	&	1.44 	&	8.201 	&	1 (1)&&	0.30 	&	1.44 	&	8.202 	&	0 (1)\\
Type 6	&&	22.144 	&	0	&&	0.44 	&	22.182 	&	0 (15)&&	5.19 	&	6.43 	&	22.140 	&	3 (0)&&	2.73 	&	3.89 	&	22.140 	&	3 (0)\\
Type 7	&&	22.268 	&	0	&&	0.50 	&	22.293 	&	4 (17)&&	5.40 	&	4.65 	&	22.253 	&	14 (5)&&	3.62 	&	3.17 	&	22.253 	&	14 (5)\\
Type 8	&&	9.627 	&	20	&&	0.25 	&	9.801 	&	0 (20)&&	0.37 	&	1.39 	&	9.723 	&	0 (9)&&	0.36 	&	1.39 	&	9.723 	&	0 (9)\\
Type 9	&&	0.000 	&	54	&&	0.01 	&	0.005 	&	0 (1)&&	0.01 	&	1.02 	&	0.000 	&	0 (0)&&	0.01 	&	1.02 	&	0.000 	&	0 (0)\\\hline
overall	&&	9.4033 	&	236	&&	0.22 	&	9.5328 	&	4 (103)&&	1.34 	&	2.20 	&	9.4303 	&	19 (16)&&	0.86 	&	1.75 	&	9.4304 	&	18 (16)\\\hline\end{tabular}
\end{table}

%% file: tableSCP_new.tex
\begin{table}[htbp]
\centering \fontsize{9pt}{10.5pt}\selectfont
\setlength{\tabcolsep}{4.4pt}
\caption{MMR-SCP results}
\label{tbl:scp}
\begin{tabular}{cl*{16}{r}}
\hline
                             && \multicolumn{2}{c}{Best Known}                  && \multicolumn{3}{c}{DS}                                                                  && \multicolumn{4}{c}{iDS-H}                                                                  && \multicolumn{4}{c}{iDS-B}                                                                      \\ \cline{3-4} \cline{6-8} \cline{10-13} \cline{15-18} 
\multicolumn{1}{c}{instances} && \multicolumn{1}{c}{\%gap} & \multicolumn{1}{c}{\#opt} && \multicolumn{1}{c}{time} & \multicolumn{1}{c}{\%gap} & \multicolumn{1}{c}{\#b(w)} && \multicolumn{1}{c}{time} & \multicolumn{1}{c}{iter} & \multicolumn{1}{c}{\%gap} & \multicolumn{1}{c}{\#b(w)} && \multicolumn{1}{c}{time} & \multicolumn{1}{c}{iter} & \multicolumn{1}{c}{\%gap} & \multicolumn{1}{c}{\#b(w)} \\ \hline		
B4	&&	0.000 	&	30 	&&	0.4 	&	0.834 	&	0 (13)&&	8.2 	&	9.03 	&	0.000 	&	0 (0)&&	4.8 	&	4.83 	&	0.000 	&	0 (0)\\
B5	&&	0.000 	&	30 	&&	0.4 	&	1.752 	&	0 (14)&&	1.7 	&	3.23 	&	0.000 	&	0 (0)&&	1.5 	&	2.33 	&	0.000 	&	0 (0)\\
B6	&&	0.000 	&	15 	&&	0.6 	&	1.499 	&	0 (6)&&	5.8 	&	7.47 	&	0.000 	&	0 (0)&&	2.0 	&	2.67 	&	0.000 	&	0 (0)\\\cline{1-1} \cline{3-4} \cline{6-8} \cline{10-13}  \cline{15-18}
M4	&&	0.000 	&	30 	&&	0.5 	&	0.021 	&	0 (4)&&	1.5 	&	2.20 	&	0.000 	&	0 (0)&&	1.2 	&	1.53 	&	0.000 	&	0 (0)\\
M5	&&	0.000 	&	30 	&&	0.4 	&	0.008 	&	0 (1)&&	0.5 	&	1.13 	&	0.000 	&	0 (0)&&	0.4 	&	1.07 	&	0.000 	&	0 (0)\\
M6	&&	0.000 	&	15 	&&	0.6 	&	0.063 	&	0 (3)&&	1.2 	&	1.60 	&	0.000 	&	0 (0)&&	1.0 	&	1.40 	&	0.000 	&	0 (0)\\\cline{1-1} \cline{3-4} \cline{6-8} \cline{10-13}  \cline{15-18}
K4	&&	11.583 	&	0 	&&	150.1 	&	11.531 	&	6 (6)&&	241.8 	&	1.63 	&	11.490 	&	8 (0)&&	352.1 	&	1.73 	&	11.489 	&	8 (0)\\
K5	&&	6.036 	&	1 	&&	46.9 	&	6.040 	&	3 (8)&&	71.2 	&	1.70 	&	6.004 	&	3 (0)&&	66.5 	&	1.70 	&	6.004 	&	3 (0)\\
K6	&&	1.349 	&	8 	&&	134.1 	&	1.410 	&	0 (3)&&	146.8 	&	1.20 	&	1.349 	&	0 (0)&&	144.6 	&	1.20 	&	1.349 	&	0 (0)\\\hline
overall	&&	2.4392 	&	159 	&&	35.5 	&	2.8896 	&	9	(58)&&	53.6 	&	3.21 	&	2.4225 	&	11	(0)&&	66.7 	&	2.11 	&	2.4223 	&	11	(0)\\\hline
\end{tabular}
\end{table}

%% file: tableGAP_new.tex
\begin{table}[htbp]
\centering \fontsize{9pt}{10.5pt}\selectfont
\setlength{\tabcolsep}{4pt}
\caption{MMR-GAP results}
\label{tbl:gap}
\begin{tabular}{cl*{16}{r}}
\hline
                             && \multicolumn{2}{c}{Best Known}                  && \multicolumn{3}{c}{DS}                                                                  && \multicolumn{4}{c}{iDS-H}                                                                  && \multicolumn{4}{c}{iDS-B}                                                                      \\ \cline{3-4} \cline{6-8} \cline{10-13} \cline{15-18} 
\multicolumn{1}{c}{instances} && \multicolumn{1}{c}{\%gap} & \multicolumn{1}{c}{\#opt} && \multicolumn{1}{c}{time} & \multicolumn{1}{c}{\%gap} & \multicolumn{1}{c}{\#b(w)} && \multicolumn{1}{c}{time} & \multicolumn{1}{c}{iter} & \multicolumn{1}{c}{\%gap} & \multicolumn{1}{c}{\#b(w)} && \multicolumn{1}{c}{time} & \multicolumn{1}{c}{iter} & \multicolumn{1}{c}{\%gap} & \multicolumn{1}{c}{\#b(w)} \\ \hline
A05	&&	5.26 	&	23 	&&	0.2 	&	5.54 	&	0 (3)&&	0.2 	&	1.13 	&	5.26 	&	0 (0)&&	0.2 	&	1.10 	&	5.26 	&	0 (0)\\
A10	&&	9.29 	&	14 	&&	0.5 	&	9.57 	&	0 (2)&&	15.6 	&	11.50 	&	9.24 	&	4 (0)&&	23.1 	&	11.53 	&	9.21 	&	5 (0)\\\cline{1-1} \cline{3-4} \cline{6-8} \cline{10-13}  \cline{15-18}
B05	&&	6.93 	&	20 	&&	11.4 	&	8.21 	&	0 (3)&&	15.0 	&	1.67 	&	6.90 	&	2 (0)&&	14.8 	&	1.53 	&	6.92 	&	1 (0)\\
B10	&&	8.51 	&	15 	&&	514.5 	&	10.54 	&	0 (12)&&	559.2 	&	2.47 	&	8.41 	&	5 (1)&&	552.9 	&	2.47 	&	8.41 	&	5 (1)\\\cline{1-1} \cline{3-4} \cline{6-8} \cline{10-13}  \cline{15-18}
C05	&&	7.31 	&	19 	&&	9.1 	&	7.32 	&	0 (1)&&	12.5 	&	1.13 	&	7.21 	&	2 (0)&&	13.2 	&	1.13 	&	7.21 	&	2 (0)\\
C10	&&	8.02 	&	17 	&&	380.1 	&	9.83 	&	0 (13)&&	481.1 	&	6.10 	&	7.92 	&	4 (0)&&	478.2 	&	5.97 	&	7.92 	&	4 (0)\\\cline{1-1} \cline{3-4} \cline{6-8} \cline{10-13}  \cline{15-18}
E05	&&	18.18 	&	11 	&&	872.3 	&	18.35 	&	1 (7)&&	894.5 	&	1.93 	&	18.06 	&	4 (3)&&	893.0 	&	1.93 	&	18.06 	&	4 (3)\\
E10	&&	24.50 	&	3 	&&	2143.2 	&	25.74 	&	2 (17)&&	2294.8 	&	1.70 	&	25.44 	&	4 (16)&&	2391.0 	&	1.80 	&	25.44 	&	5 (16)\\\hline
overall	&&	11.00 	&	122	&&	491.4 	&	11.89 	&	3 (58)&&	534.1 	&	3.45 	&	11.06 	&	25 (20)&&	545.8 	&	3.43 	&	11.05 	&	26 (20)\\\hline
\end{tabular}
\end{table}

%% file: tableMKP.tex
\begin{table}[htbp]
\centering \fontsize{6pt}{6.8pt}\selectfont
\setlength{\tabcolsep}{4pt}
\caption{MMR-MKP results with $m=5,n=100$}
\label{tbl:mkp-05100}
\begin{tabular}{ll*{22}{r}}
\hline
                             && \multicolumn{4}{c}{B\&C}                  && \multicolumn{3}{c}{Fix}                  && \multicolumn{3}{c}{DS}                  && \multicolumn{4}{c}{iDS-H}                                                                  && \multicolumn{4}{c}{iDS-B}                                                                      \\ \cline{3-6} \cline{8-10} \cline{12-14} \cline{16-19} \cline{21-24} 
\multicolumn{1}{c}{instance} && \multicolumn{1}{c}{obj} & \multicolumn{1}{c}{time} & \multicolumn{1}{c}{LB} & \multicolumn{1}{c}{\%gap} && \multicolumn{1}{c}{obj} & \multicolumn{1}{c}{time} & \multicolumn{1}{c}{\%gap}  && \multicolumn{1}{c}{obj} & \multicolumn{1}{c}{time} & \multicolumn{1}{c}{\%gap}  && \multicolumn{1}{c}{obj} & \multicolumn{1}{c}{time} & \multicolumn{1}{c}{iter} & \multicolumn{1}{c}{\%gap} && \multicolumn{1}{c}{obj} & \multicolumn{1}{c}{time} & \multicolumn{1}{c}{iter} & \multicolumn{1}{c}{\%gap} \\ \bhline{1.5pt}
 0510010-01	&&	\textbf{662}	&	73.7 	&	662 	&	0.0 	&&	\textbf{662}	&	0.5 	&	0.0 	&&	\textbf{662}	&	1.2 	&	0.0 	&&	\textbf{662}	&	1.2 	&	1	&	0.0 	&&	\textbf{662}	&	1.2 	&	1	&	0.0 	\\
 0510010-02	&&	\textbf{663}	&	94.9 	&	663 	&	0.0 	&&	\textbf{663}	&	1.8 	&	0.0 	&&		676 		&	0.7 	&	1.9 	&&	\textbf{663}	&	2.1 	&	2	&	0.0 	&&	\textbf{663}	&	2.6 	&	2	&	0.0 	\\
 0510010-03	&&	\textbf{426}	&	51.7 	&	426 	&	0.0 	&&		457 		&	2.1 	&	6.8 	&&		441 		&	1.3 	&	3.4 	&&	\textbf{426}	&	6.7 	&	5	&	0.0 	&&	\textbf{426}	&	6.3 	&	5	&	0.0 	\\
 0510010-04	&&	\textbf{809}	&	783.7 	&	809 	&	0.0 	&&	\textbf{809}	&	1.9 	&	0.0 	&&	\textbf{809}	&	11.3 	&	0.0 	&&	\textbf{809}	&	11.3 	&	1	&	0.0 	&&	\textbf{809}	&	11.3 	&	1	&	0.0 	\\
 0510010-05	&&	\textbf{737}	&	85.9 	&	737 	&	0.0 	&&		755 		&	3.5 	&	2.4 	&&	\textbf{737}	&	1.7 	&	0.0 	&&	\textbf{737}	&	1.7 	&	1	&	0.0 	&&	\textbf{737}	&	1.7 	&	1	&	0.0 	\\
 0510010-06	&&	\textbf{482}	&	121.8 	&	482 	&	0.0 	&&	\textbf{482}	&	0.4 	&	0.0 	&&		541 		&	1.7 	&	10.9 	&&	\textbf{482}	&	3.7 	&	2	&	0.0 	&&	\textbf{482}	&	3.3 	&	2	&	0.0 	\\
 0510010-07	&&	\textbf{603}	&	34.0 	&	603 	&	0.0 	&&	\textbf{603}	&	0.5 	&	0.0 	&&	\textbf{603}	&	0.8 	&	0.0 	&&	\textbf{603}	&	0.8 	&	1	&	0.0 	&&	\textbf{603}	&	0.8 	&	1	&	0.0 	\\
 0510010-08	&&	\textbf{406}	&	341.7 	&	406 	&	0.0 	&&		482 		&	1.2 	&	15.8 	&&	\textbf{406}	&	0.8 	&	0.0 	&&	\textbf{406}	&	0.8 	&	1	&	0.0 	&&	\textbf{406}	&	0.8 	&	1	&	0.0 	\\
 0510010-09	&&	\textbf{469}	&	110.4 	&	469 	&	0.0 	&&		499 		&	1.9 	&	6.0 	&&		499 		&	1.3 	&	6.0 	&&	\textbf{469}	&	3.1 	&	2	&	0.0 	&&	\textbf{469}	&	3.2 	&	2	&	0.0 	\\
 0510010-10	&&	\textbf{587}	&	362.3 	&	587 	&	0.0 	&&		644 		&	2.0 	&	8.9 	&&		604 		&	3.0 	&	2.8 	&&	\textbf{587}	&	8.6 	&	2	&	0.0 	&&	\textbf{587}	&	8.2 	&	2	&	0.0 	\\
 0510010-11	&&	\textbf{422}	&	62.2 	&	422 	&	0.0 	&&	\textbf{422}	&	0.1 	&	0.0 	&&	\textbf{422}	&	0.4 	&	0.0 	&&	\textbf{422}	&	0.4 	&	1	&	0.0 	&&	\textbf{422}	&	0.4 	&	1	&	0.0 	\\
 0510010-12	&&	\textbf{667}	&	47.3 	&	667 	&	0.0 	&&		696 		&	0.8 	&	4.2 	&&	\textbf{667}	&	2.5 	&	0.0 	&&	\textbf{667}	&	2.5 	&	1	&	0.0 	&&	\textbf{667}	&	2.5 	&	1	&	0.0 	\\
 0510010-13	&&	\textbf{618}	&	204.2 	&	618 	&	0.0 	&&		633 		&	1.9 	&	2.4 	&&		633 		&	6.9 	&	2.4 	&&	\textbf{618}	&	22.3 	&	3	&	0.0 	&&	\textbf{618}	&	22.4 	&	3	&	0.0 	\\
 0510010-14	&&	\textbf{442}	&	78.7 	&	442 	&	0.0 	&&		466 		&	1.9 	&	5.2 	&&		476 		&	1.8 	&	7.1 	&&	\textbf{442}	&	10.8 	&	4	&	0.0 	&&	\textbf{442}	&	11.3 	&	4	&	0.0 	\\
 0510010-15	&&	\textbf{409}	&	62.4 	&	409 	&	0.0 	&&	\textbf{409}	&	0.8 	&	0.0 	&&	\textbf{409}	&	1.0 	&	0.0 	&&	\textbf{409}	&	1.0 	&	1	&	0.0 	&&	\textbf{409}	&	1.0 	&	1	&	0.0 	\\
 0510010-16	&&	\textbf{507}	&	53.0 	&	507 	&	0.0 	&&		603 		&	0.6 	&	15.9 	&&	\textbf{507}	&	1.0 	&	0.0 	&&	\textbf{507}	&	1.0 	&	1	&	0.0 	&&	\textbf{507}	&	1.0 	&	1	&	0.0 	\\
 0510010-17	&&	\textbf{413}	&	38.1 	&	413 	&	0.0 	&&	\textbf{413}	&	0.3 	&	0.0 	&&		417 		&	0.6 	&	1.0 	&&	\textbf{413}	&	1.6 	&	2	&	0.0 	&&	\textbf{413}	&	1.5 	&	2	&	0.0 	\\
 0510010-18	&&	\textbf{579}	&	85.3 	&	579 	&	0.0 	&&		638 		&	3.0 	&	9.2 	&&	\textbf{579}	&	1.6 	&	0.0 	&&	\textbf{579}	&	1.6 	&	1	&	0.0 	&&	\textbf{579}	&	1.6 	&	1	&	0.0 	\\
 0510010-19	&&	\textbf{646}	&	187.7 	&	646 	&	0.0 	&&		647 		&	0.7 	&	0.2 	&&		647 		&	0.9 	&	0.2 	&&	\textbf{646}	&	5.2 	&	3	&	0.0 	&&	\textbf{646}	&	4.4 	&	3	&	0.0 	\\
 0510010-20	&&	\textbf{525}	&	46.8 	&	525 	&	0.0 	&&		554 		&	0.5 	&	5.2 	&&		554 		&	0.9 	&	5.2 	&&	\textbf{525}	&	3.2 	&	3	&	0.0 	&&	\textbf{525}	&	3.3 	&	3	&	0.0 	\\
 0510010-21	&&	\textbf{326}	&	19.6 	&	326 	&	0.0 	&&		377 		&	0.4 	&	13.5 	&&	\textbf{326}	&	0.2 	&	0.0 	&&	\textbf{326}	&	0.2 	&	1	&	0.0 	&&	\textbf{326}	&	0.2 	&	1	&	0.0 	\\
 0510010-22	&&	\textbf{413}	&	33.0 	&	413 	&	0.0 	&&	\textbf{413}	&	0.2 	&	0.0 	&&	\textbf{413}	&	1.4 	&	0.0 	&&	\textbf{413}	&	1.4 	&	1	&	0.0 	&&	\textbf{413}	&	1.4 	&	1	&	0.0 	\\
 0510010-23	&&	\textbf{488}	&	71.6 	&	488 	&	0.0 	&&		494 		&	1.1 	&	1.2 	&&		494 		&	1.9 	&	1.2 	&&	\textbf{488}	&	12.2 	&	4	&	0.0 	&&	\textbf{488}	&	13.9 	&	4	&	0.0 	\\
 0510010-24	&&	\textbf{417}	&	54.8 	&	417 	&	0.0 	&&	\textbf{417}	&	0.6 	&	0.0 	&&	\textbf{417}	&	1.2 	&	0.0 	&&	\textbf{417}	&	1.2 	&	1	&	0.0 	&&	\textbf{417}	&	1.2 	&	1	&	0.0 	\\
 0510010-25	&&	\textbf{363}	&	22.2 	&	363 	&	0.0 	&&		366 		&	3.0 	&	0.8 	&&		380 		&	0.6 	&	4.5 	&&	\textbf{363}	&	3.2 	&	3	&	0.0 	&&	\textbf{363}	&	3.1 	&	3	&	0.0 	\\
 0510010-26	&&	\textbf{497}	&	21.9 	&	497 	&	0.0 	&&		544 		&	1.0 	&	8.6 	&&	\textbf{497}	&	1.4 	&	0.0 	&&	\textbf{497}	&	1.4 	&	1	&	0.0 	&&	\textbf{497}	&	1.4 	&	1	&	0.0 	\\
 0510010-27	&&	\textbf{350}	&	42.3 	&	350 	&	0.0 	&&	\textbf{350}	&	0.3 	&	0.0 	&&	\textbf{350}	&	0.4 	&	0.0 	&&	\textbf{350}	&	0.4 	&	1	&	0.0 	&&	\textbf{350}	&	0.4 	&	1	&	0.0 	\\
 0510010-28	&&	\textbf{363}	&	9.5 	&	363 	&	0.0 	&&	\textbf{363}	&	0.3 	&	0.0 	&&	\textbf{363}	&	0.4 	&	0.0 	&&	\textbf{363}	&	0.4 	&	1	&	0.0 	&&	\textbf{363}	&	0.4 	&	1	&	0.0 	\\
 0510010-29	&&	\textbf{381}	&	32.0 	&	381 	&	0.0 	&&	\textbf{381}	&	0.2 	&	0.0 	&&	\textbf{381}	&	0.4 	&	0.0 	&&	\textbf{381}	&	0.4 	&	1	&	0.0 	&&	\textbf{381}	&	0.4 	&	1	&	0.0 	\\
 0510010-30	&&	\textbf{340}	&	49.7 	&	340 	&	0.0 	&&		383 		&	1.6 	&	11.2 	&&	\textbf{340}	&	1.0 	&	0.0 	&&	\textbf{340}	&	1.0 	&	1	&	0.0 	&&	\textbf{340}	&	1.0 	&	1	&	0.0 	\\
\hline																																															
 0510020-01	&&	\textbf{1963}	&	764.6 	&	1963 	&	0.0 	&&		1968 		&	1.7 	&	0.3 	&&		1984 		&	7.5 	&	1.1 	&&	\textbf{1963}	&	45.3 	&	4	&	0.0 	&&	\textbf{1963}	&	44.0 	&	4	&	0.0 	\\
 0510020-02	&&	\textbf{1434}	&	465.5 	&	1434 	&	0.0 	&&	\textbf{1434}	&	1.0 	&	0.0 	&&		1457 		&	2.3 	&	1.6 	&&	\textbf{1434}	&	8.6 	&	3	&	0.0 	&&	\textbf{1434}	&	8.7 	&	3	&	0.0 	\\
 0510020-03	&&	\textbf{1363}	&	81.7 	&	1363 	&	0.0 	&&		1474 		&	2.4 	&	7.5 	&&	\textbf{1363}	&	1.5 	&	0.0 	&&	\textbf{1363}	&	1.5 	&	1	&	0.0 	&&	\textbf{1363}	&	1.5 	&	1	&	0.0 	\\
 0510020-04	&&	\textbf{2190}	&	3418.2 	&	1987 	&	9.3 	&&		2240 		&	0.4 	&	11.3 	&&	\textbf{2190}	&	11.1 	&	9.3 	&&	\textbf{2190}	&	11.1 	&	1	&	9.3 	&&	\textbf{2190}	&	11.1 	&	1	&	9.3 	\\
 0510020-05	&&	\textbf{1732}	&	1906.1 	&	1732 	&	0.0 	&&	\textbf{1732}	&	2.6 	&	0.0 	&&	\textbf{1732}	&	3.0 	&	0.0 	&&	\textbf{1732}	&	3.0 	&	1	&	0.0 	&&	\textbf{1732}	&	3.0 	&	1	&	0.0 	\\
 0510020-06	&&	\textbf{1364}	&	178.8 	&	1364 	&	0.0 	&&	\textbf{1364}	&	0.5 	&	0.0 	&&	\textbf{1364}	&	0.4 	&	0.0 	&&	\textbf{1364}	&	0.4 	&	1	&	0.0 	&&	\textbf{1364}	&	0.4 	&	1	&	0.0 	\\
 0510020-07	&&	\textbf{1679}	&	882.1 	&	1679 	&	0.0 	&&		1742 		&	1.8 	&	3.6 	&&	\textbf{1679}	&	5.6 	&	0.0 	&&	\textbf{1679}	&	5.6 	&	1	&	0.0 	&&	\textbf{1679}	&	5.6 	&	1	&	0.0 	\\
 0510020-08	&&	\textbf{1577}	&	274.9 	&	1577 	&	0.0 	&&	\textbf{1577}	&	0.5 	&	0.0 	&&	\textbf{1577}	&	2.2 	&	0.0 	&&	\textbf{1577}	&	2.2 	&	1	&	0.0 	&&	\textbf{1577}	&	2.2 	&	1	&	0.0 	\\
 0510020-09	&&	\textbf{1393}	&	457.5 	&	1393 	&	0.0 	&&		1424 		&	0.5 	&	2.2 	&&		1424 		&	1.3 	&	2.2 	&&	\textbf{1393}	&	2.6 	&	2	&	0.0 	&&	\textbf{1393}	&	2.7 	&	2	&	0.0 	\\
 0510020-10	&&	\textbf{1707}	&	940.0 	&	1707 	&	0.0 	&&	\textbf{1707}	&	0.5 	&	0.0 	&&		1717 		&	2.9 	&	0.6 	&&	\textbf{1707}	&	5.8 	&	2	&	0.0 	&&	\textbf{1707}	&	5.7 	&	2	&	0.0 	\\
 0510020-11	&&		1615 		&	2785.1 	&	1301 	&	19.4 	&&		1602 		&	2.2 	&	18.8 	&&	\textbf{1574}	&	8.1 	&	17.3 	&&	\textbf{1574}	&	8.1 	&	1	&	17.3 	&&	\textbf{1574}	&	8.1 	&	1	&	17.3 	\\
 0510020-12	&&	\textbf{1224}	&	214.7 	&	1224 	&	0.0 	&&	\textbf{1224}	&	0.3 	&	0.0 	&&	\textbf{1224}	&	1.8 	&	0.0 	&&	\textbf{1224}	&	1.8 	&	1	&	0.0 	&&	\textbf{1224}	&	1.8 	&	1	&	0.0 	\\
 0510020-13	&&	\textbf{1767}	&	3286.1 	&	1493 	&	15.5 	&&	\textbf{1767}	&	0.5 	&	15.5 	&&	\textbf{1767}	&	7.8 	&	15.5 	&&	\textbf{1767}	&	7.8 	&	1	&	15.5 	&&	\textbf{1767}	&	7.8 	&	1	&	15.5 	\\
 0510020-14	&&	\textbf{1109}	&	484.6 	&	1109 	&	0.0 	&&		1123 		&	1.2 	&	1.2 	&&	\textbf{1109}	&	1.5 	&	0.0 	&&	\textbf{1109}	&	1.5 	&	1	&	0.0 	&&	\textbf{1109}	&	1.5 	&	1	&	0.0 	\\
 0510020-15	&&	\textbf{1272}	&	990.9 	&	1272 	&	0.0 	&&	\textbf{1272}	&	0.8 	&	0.0 	&&	\textbf{1272}	&	2.2 	&	0.0 	&&	\textbf{1272}	&	2.2 	&	1	&	0.0 	&&	\textbf{1272}	&	2.2 	&	1	&	0.0 	\\
 0510020-16	&&	\textbf{1229}	&	191.3 	&	1229 	&	0.0 	&&		1346 		&	0.7 	&	8.7 	&&	\textbf{1229}	&	1.4 	&	0.0 	&&	\textbf{1229}	&	1.4 	&	1	&	0.0 	&&	\textbf{1229}	&	1.4 	&	1	&	0.0 	\\
 0510020-17	&&	\textbf{1420}	&	100.8 	&	1420 	&	0.0 	&&		1446 		&	0.2 	&	1.8 	&&	\textbf{1420}	&	0.8 	&	0.0 	&&	\textbf{1420}	&	0.8 	&	1	&	0.0 	&&	\textbf{1420}	&	0.8 	&	1	&	0.0 	\\
 0510020-18	&&	\textbf{1743}	&	791.5 	&	1743 	&	0.0 	&&		1761 		&	1.1 	&	1.0 	&&	\textbf{1743}	&	5.7 	&	0.0 	&&	\textbf{1743}	&	5.7 	&	1	&	0.0 	&&	\textbf{1743}	&	5.7 	&	1	&	0.0 	\\
 0510020-19	&&		2175 		&	299.1 	&	1626 	&	25.2 	&&		2135 		&	3.6 	&	23.8 	&&	\textbf{2122}	&	21.7 	&	23.4 	&&	\textbf{2122}	&	21.7 	&	1	&	23.4 	&&	\textbf{2122}	&	21.7 	&	1	&	23.4 	\\
 0510020-20	&&	\textbf{1592}	&	746.0 	&	1592 	&	0.0 	&&	\textbf{1592}	&	0.4 	&	0.0 	&&	\textbf{1592}	&	1.9 	&	0.0 	&&	\textbf{1592}	&	1.9 	&	1	&	0.0 	&&	\textbf{1592}	&	1.9 	&	1	&	0.0 	\\
 0510020-21	&&	\textbf{1055}	&	137.0 	&	1055 	&	0.0 	&&	\textbf{1055}	&	0.3 	&	0.0 	&&	\textbf{1055}	&	0.8 	&	0.0 	&&	\textbf{1055}	&	0.8 	&	1	&	0.0 	&&	\textbf{1055}	&	0.8 	&	1	&	0.0 	\\
 0510020-22	&&	\textbf{809}	&	235.5 	&	809 	&	0.0 	&&	\textbf{809}	&	0.4 	&	0.0 	&&	\textbf{809}	&	0.5 	&	0.0 	&&	\textbf{809}	&	0.5 	&	1	&	0.0 	&&	\textbf{809}	&	0.5 	&	1	&	0.0 	\\
 0510020-23	&&	\textbf{1143}	&	245.5 	&	1143 	&	0.0 	&&	\textbf{1143}	&	0.4 	&	0.0 	&&	\textbf{1143}	&	2.3 	&	0.0 	&&	\textbf{1143}	&	2.3 	&	1	&	0.0 	&&	\textbf{1143}	&	2.3 	&	1	&	0.0 	\\
 0510020-24	&&	\textbf{1011}	&	287.9 	&	1011 	&	0.0 	&&	\textbf{1011}	&	0.1 	&	0.0 	&&	\textbf{1011}	&	0.2 	&	0.0 	&&	\textbf{1011}	&	0.2 	&	1	&	0.0 	&&	\textbf{1011}	&	0.2 	&	1	&	0.0 	\\
 0510020-25	&&	\textbf{1232}	&	79.2 	&	1232 	&	0.0 	&&		1354 		&	0.8 	&	9.0 	&&	\textbf{1232}	&	1.0 	&	0.0 	&&	\textbf{1232}	&	1.0 	&	1	&	0.0 	&&	\textbf{1232}	&	1.0 	&	1	&	0.0 	\\
 0510020-26	&&	\textbf{1428}	&	343.6 	&	1428 	&	0.0 	&&	\textbf{1428}	&	0.4 	&	0.0 	&&		1433 		&	3.6 	&	0.3 	&&	\textbf{1428}	&	10.2 	&	2	&	0.0 	&&	\textbf{1428}	&	8.7 	&	2	&	0.0 	\\
 0510020-27	&&	\textbf{1122}	&	1232.4 	&	1122 	&	0.0 	&&	\textbf{1122}	&	0.6 	&	0.0 	&&	\textbf{1122}	&	4.8 	&	0.0 	&&	\textbf{1122}	&	4.8 	&	1	&	0.0 	&&	\textbf{1122}	&	4.8 	&	1	&	0.0 	\\
 0510020-28	&&	\textbf{1003}	&	132.6 	&	1003 	&	0.0 	&&	\textbf{1003}	&	0.3 	&	0.0 	&&	\textbf{1003}	&	1.0 	&	0.0 	&&	\textbf{1003}	&	1.0 	&	1	&	0.0 	&&	\textbf{1003}	&	1.0 	&	1	&	0.0 	\\
 0510020-29	&&	\textbf{1214}	&	479.6 	&	1214 	&	0.0 	&&	\textbf{1214}	&	0.9 	&	0.0 	&&	\textbf{1214}	&	1.1 	&	0.0 	&&	\textbf{1214}	&	1.1 	&	1	&	0.0 	&&	\textbf{1214}	&	1.1 	&	1	&	0.0 	\\
 0510020-30	&&	\textbf{1101}	&	226.3 	&	1101 	&	0.0 	&&	\textbf{1101}	&	0.4 	&	0.0 	&&	\textbf{1101}	&	1.0 	&	0.0 	&&	\textbf{1101}	&	1.0 	&	1	&	0.0 	&&	\textbf{1101}	&	1.0 	&	1	&	0.0 	\\
\hline																																															
 0510030-01	&&		2774 		&	2476.9 	&	2134 	&	23.1 	&&		2672 		&	0.6 	&	20.1 	&&		2672 		&	9.8 	&	20.1 	&&	\textbf{2666}	&	21.9 	&	2	&	20.0 	&&	\textbf{2666}	&	22.2 	&	2	&	20.0 	\\
 0510030-02	&&	\textbf{3187}	&	784.4 	&	2752 	&	13.6 	&&		3327 		&	2.4 	&	17.3 	&&	\textbf{3187}	&	25.1 	&	13.6 	&&	\textbf{3187}	&	25.1 	&	1	&	13.6 	&&	\textbf{3187}	&	25.1 	&	1	&	13.6 	\\
 0510030-03	&&	\textbf{2912}	&	1075.5 	&	2912 	&	0.0 	&&	\textbf{2912}	&	0.4 	&	0.0 	&&		2915 		&	2.8 	&	0.1 	&&	\textbf{2912}	&	6.2 	&	2	&	0.0 	&&	\textbf{2912}	&	5.8 	&	2	&	0.0 	\\
 0510030-04	&&		4122 		&	2400.7 	&	3111 	&	24.5 	&&	\textbf{3927}	&	0.9 	&	20.8 	&&		3979 		&	41.7 	&	21.8 	&&	\textbf{3927}	&	89.1 	&	2	&	20.8 	&&	\textbf{3927}	&	85.1 	&	2	&	20.8 	\\
 0510030-05	&&	\textbf{3046}	&	2193.2 	&	2659 	&	12.7 	&&		3109 		&	2.0 	&	14.5 	&&	\textbf{3046}	&	10.1 	&	12.7 	&&	\textbf{3046}	&	10.1 	&	1	&	12.7 	&&	\textbf{3046}	&	10.1 	&	1	&	12.7 	\\
 0510030-06	&&	\textbf{3122}	&	2503.7 	&	2915 	&	6.6 	&&	\textbf{3122}	&	0.3 	&	6.6 	&&	\textbf{3122}	&	10.6 	&	6.6 	&&	\textbf{3122}	&	10.6 	&	1	&	6.6 	&&	\textbf{3122}	&	10.6 	&	1	&	6.6 	\\
 0510030-07	&&	\textbf{3445}	&	1112.6 	&	3244 	&	5.8 	&&		3500 		&	0.6 	&	7.3 	&&		3500 		&	5.0 	&	7.3 	&&	\textbf{3445}	&	13.4 	&	2	&	5.8 	&&	\textbf{3445}	&	13.0 	&	2	&	5.8 	\\
 0510030-08	&&	\textbf{2556}	&	3161.6 	&	2375 	&	7.1 	&&		2644 		&	1.5 	&	10.2 	&&	\textbf{2556}	&	2.6 	&	7.1 	&&	\textbf{2556}	&	2.6 	&	1	&	7.1 	&&	\textbf{2556}	&	2.6 	&	1	&	7.1 	\\
 0510030-09	&&	\textbf{2350}	&	1641.1 	&	2350 	&	0.0 	&&		2423 		&	0.8 	&	3.0 	&&	\textbf{2350}	&	1.6 	&	0.0 	&&	\textbf{2350}	&	1.6 	&	1	&	0.0 	&&	\textbf{2350}	&	1.6 	&	1	&	0.0 	\\
 0510030-10	&&		3172 		&	1849.5 	&	2467 	&	22.2 	&&	\textbf{3022}	&	1.3 	&	18.4 	&&	\textbf{3022}	&	14.9 	&	18.4 	&&	\textbf{3022}	&	14.9 	&	1	&	18.4 	&&	\textbf{3022}	&	14.9 	&	1	&	18.4 	\\
 0510030-11	&&		3591 		&	2424.5 	&	2778 	&	22.6 	&&		3508 		&	0.6 	&	20.8 	&&		3521 		&	15.6 	&	21.1 	&&	\textbf{3488}	&	31.5 	&	2	&	20.4 	&&	\textbf{3488}	&	31.6 	&	2	&	20.4 	\\
 0510030-12	&&		3226 		&	2388.1 	&	2461 	&	23.7 	&&	\textbf{3055}	&	0.3 	&	19.4 	&&	\textbf{3055}	&	14.0 	&	19.4 	&&	\textbf{3055}	&	14.0 	&	1	&	19.4 	&&	\textbf{3055}	&	14.0 	&	1	&	19.4 	\\
 0510030-13	&&		3814 		&	112.2 	&	2767 	&	27.5 	&&		3697 		&	1.6 	&	25.2 	&&		3677 		&	94.8 	&	24.7 	&&	\textbf{3660}	&	220.8 	&	2	&	24.4 	&&	\textbf{3660}	&	217.1 	&	2	&	24.4 	\\
 0510030-14	&&		3873 		&	775.4 	&	3054 	&	21.1 	&&	\textbf{3771}	&	0.7 	&	19.0 	&&	\textbf{3771}	&	13.1 	&	19.0 	&&	\textbf{3771}	&	13.1 	&	1	&	19.0 	&&	\textbf{3771}	&	13.1 	&	1	&	19.0 	\\
 0510030-15	&&		3096 		&	199.2 	&	2373 	&	23.4 	&&	\textbf{3001}	&	1.1 	&	20.9 	&&	\textbf{3001}	&	7.6 	&	20.9 	&&	\textbf{3001}	&	7.6 	&	1	&	20.9 	&&	\textbf{3001}	&	7.6 	&	1	&	20.9 	\\
 0510030-16	&&	\textbf{2996}	&	283.9 	&	2523 	&	15.8 	&&		3058 		&	3.0 	&	17.5 	&&		2997 		&	14.5 	&	15.8 	&&	\textbf{2996}	&	31.0 	&	2	&	15.8 	&&	\textbf{2996}	&	27.8 	&	2	&	15.8 	\\
 0510030-17	&&		2801 		&	137.7 	&	2347 	&	16.2 	&&		2878 		&	1.4 	&	18.5 	&&		2801 		&	6.8 	&	16.2 	&&	\textbf{2795}	&	17.8 	&	2	&	16.0 	&&	\textbf{2795}	&	18.0 	&	2	&	16.0 	\\
 0510030-18	&&		3294 		&	3476.8 	&	2688 	&	18.4 	&&		3320 		&	0.5 	&	19.0 	&&	\textbf{3249}	&	10.8 	&	17.3 	&&	\textbf{3249}	&	10.8 	&	1	&	17.3 	&&	\textbf{3249}	&	10.8 	&	1	&	17.3 	\\
 0510030-19	&&		3170 		&	1328.2 	&	2406 	&	24.1 	&&	\textbf{3018}	&	0.2 	&	20.3 	&&	\textbf{3018}	&	11.6 	&	20.3 	&&	\textbf{3018}	&	11.6 	&	1	&	20.3 	&&	\textbf{3018}	&	11.6 	&	1	&	20.3 	\\
 0510030-20	&&	\textbf{2933}	&	112.3 	&	2821 	&	3.8 	&&	\textbf{2933}	&	0.2 	&	3.8 	&&	\textbf{2933}	&	2.0 	&	3.8 	&&	\textbf{2933}	&	2.0 	&	1	&	3.8 	&&	\textbf{2933}	&	2.0 	&	1	&	3.8 	\\
 0510030-21	&&	\textbf{2450}	&	1193.4 	&	2334 	&	4.7 	&&	\textbf{2450}	&	0.5 	&	4.7 	&&	\textbf{2450}	&	9.4 	&	4.7 	&&	\textbf{2450}	&	9.4 	&	1	&	4.7 	&&	\textbf{2450}	&	9.4 	&	1	&	4.7 	\\
 0510030-22	&&	\textbf{2453}	&	1543.9 	&	2453 	&	0.0 	&&	\textbf{2453}	&	0.3 	&	0.0 	&&	\textbf{2453}	&	6.7 	&	0.0 	&&	\textbf{2453}	&	6.7 	&	1	&	0.0 	&&	\textbf{2453}	&	6.7 	&	1	&	0.0 	\\
 0510030-23	&&	\textbf{2257}	&	1082.6 	&	2140 	&	5.2 	&&	\textbf{2257}	&	0.3 	&	5.2 	&&	\textbf{2257}	&	6.9 	&	5.2 	&&	\textbf{2257}	&	6.9 	&	1	&	5.2 	&&	\textbf{2257}	&	6.9 	&	1	&	5.2 	\\
 0510030-24	&&	\textbf{2206}	&	2703.0 	&	2080 	&	5.7 	&&		2265 		&	0.5 	&	8.2 	&&		2208 		&	4.4 	&	5.8 	&&	\textbf{2206}	&	10.5 	&	2	&	5.7 	&&	\textbf{2206}	&	10.3 	&	2	&	5.7 	\\
 0510030-25	&&	\textbf{2357}	&	70.8 	&	2357 	&	0.0 	&&	\textbf{2357}	&	0.6 	&	0.0 	&&	\textbf{2357}	&	5.8 	&	0.0 	&&	\textbf{2357}	&	5.8 	&	1	&	0.0 	&&	\textbf{2357}	&	5.8 	&	1	&	0.0 	\\
 0510030-26	&&	\textbf{2168}	&	1205.5 	&	2168 	&	0.0 	&&	\textbf{2168}	&	0.3 	&	0.0 	&&	\textbf{2168}	&	3.7 	&	0.0 	&&	\textbf{2168}	&	3.7 	&	1	&	0.0 	&&	\textbf{2168}	&	3.7 	&	1	&	0.0 	\\
 0510030-27	&&		3202 		&	1013.5 	&	2685 	&	16.1 	&&	\textbf{3103}	&	0.4 	&	13.5 	&&		3142 		&	26.6 	&	14.5 	&&	\textbf{3103}	&	52.5 	&	2	&	13.5 	&&	\textbf{3103}	&	71.2 	&	2	&	13.5 	\\
 0510030-28	&&	\textbf{2902}	&	1462.3 	&	2691 	&	7.3 	&&		3006 		&	0.3 	&	10.5 	&&		2943 		&	6.7 	&	8.6 	&&	\textbf{2902}	&	78.7 	&	7	&	7.3 	&&	\textbf{2902}	&	81.1 	&	7	&	7.3 	\\
 0510030-29	&&	\textbf{2200}	&	1290.5 	&	2200 	&	0.0 	&&		2260 		&	0.7 	&	2.7 	&&	\textbf{2200}	&	6.4 	&	0.0 	&&	\textbf{2200}	&	6.4 	&	1	&	0.0 	&&	\textbf{2200}	&	6.4 	&	1	&	0.0 	\\
 0510030-30	&&	\textbf{2194}	&	466.5 	&	2194 	&	0.0 	&&		2394 		&	1.5 	&	8.4 	&&	\textbf{2194}	&	2.4 	&	0.0 	&&	\textbf{2194}	&	2.4 	&	1	&	0.0 	&&	\textbf{2194}	&	2.4 	&	1	&	0.0 	\\
\bhline{1pt}
\end{tabular}
\end{table}

\begin{table}[htbp]
\centering \fontsize{6pt}{6.8pt}\selectfont
\setlength{\tabcolsep}{3.8pt}
\caption{MMR-MKP results with $m=10,n=100$}
\label{tbl:mkp-10100}
\begin{tabular}{ll*{22}{r}}
\hline
                             && \multicolumn{4}{c}{B\&C}                  && \multicolumn{3}{c}{Fix}                  && \multicolumn{3}{c}{DS}                  && \multicolumn{4}{c}{iDS-H}                                                                  && \multicolumn{4}{c}{iDS-B}                                                                      \\ \cline{3-6} \cline{8-10} \cline{12-14} \cline{16-19} \cline{21-24} 
\multicolumn{1}{c}{instance} && \multicolumn{1}{c}{obj} & \multicolumn{1}{c}{time} & \multicolumn{1}{c}{LB} & \multicolumn{1}{c}{\%gap} && \multicolumn{1}{c}{obj} & \multicolumn{1}{c}{time} & \multicolumn{1}{c}{\%gap}  && \multicolumn{1}{c}{obj} & \multicolumn{1}{c}{time} & \multicolumn{1}{c}{\%gap}  && \multicolumn{1}{c}{obj} & \multicolumn{1}{c}{time} & \multicolumn{1}{c}{iter} & \multicolumn{1}{c}{\%gap} && \multicolumn{1}{c}{obj} & \multicolumn{1}{c}{time} & \multicolumn{1}{c}{iter} & \multicolumn{1}{c}{\%gap} \\ \bhline{1.5pt}
 1010010-01	&&	\textbf{798}	&	872.3 	&	798	&	0.0 	&&		847		&	15.1 	&	5.8 	&&	\textbf{798}	&	21.1 	&	0.0 	&&	\textbf{798}	&	21.1 	&	1 	&	0.0 	&&	\textbf{798}	&	21.1 	&	1 	&	0.0 	\\
 1010010-02	&&	\textbf{474}	&	1170.3 	&	474	&	0.0 	&&	\textbf{474}	&	3.7 	&	0.0 	&&	\textbf{474}	&	3.5 	&	0.0 	&&	\textbf{474}	&	3.5 	&	1 	&	0.0 	&&	\textbf{474}	&	3.5 	&	1 	&	0.0 	\\
 1010010-03	&&	\textbf{759}	&	3600.0 	&	492	&	35.2 	&&		789		&	9.4 	&	37.6 	&&	\textbf{759}	&	15.3 	&	35.2 	&&	\textbf{759}	&	15.3 	&	1 	&	35.2 	&&	\textbf{759}	&	15.3 	&	1 	&	35.2 	\\
 1010010-04	&&		1195		&	2648.4 	&	650	&	45.6 	&&		1148		&	47.4 	&	43.4 	&&	\textbf{1093}	&	69.6 	&	40.5 	&&	\textbf{1093}	&	69.6 	&	1 	&	40.5 	&&	\textbf{1093}	&	69.6 	&	1 	&	40.5 	\\
 1010010-05	&&	\textbf{753}	&	3217.5 	&	753	&	0.0 	&&		881		&	11.7 	&	14.5 	&&		805 		&	20.9 	&	6.5 	&&	\textbf{753}	&	44.7 	&	2 	&	0.0 	&&	\textbf{753}	&	47.9 	&	2 	&	0.0 	\\
 1010010-06	&&	\textbf{807}	&	1633.0 	&	451	&	44.1 	&&	\textbf{807}	&	10.9 	&	44.1 	&&	\textbf{807}	&	33.7 	&	44.1 	&&	\textbf{807}	&	33.7 	&	1 	&	44.1 	&&	\textbf{807}	&	33.7 	&	1 	&	44.1 	\\
 1010010-07	&&	\textbf{640}	&	515.2 	&	640	&	0.0 	&&		699		&	7.5 	&	8.4 	&&		666 		&	10.5 	&	3.9 	&&	\textbf{640}	&	117.0 	&	8 	&	0.0 	&&	\textbf{640}	&	113.4 	&	8 	&	0.0 	\\
 1010010-08	&&	\textbf{684}	&	1085.3 	&	684	&	0.0 	&&		731		&	6.7 	&	6.4 	&&	\textbf{684}	&	6.8 	&	0.0 	&&	\textbf{684}	&	6.8 	&	1 	&	0.0 	&&	\textbf{684}	&	6.8 	&	1 	&	0.0 	\\
 1010010-09	&&	\textbf{611}	&	570.1 	&	611	&	0.0 	&&		613		&	3.2 	&	0.3 	&&		613 		&	6.2 	&	0.3 	&&	\textbf{611}	&	15.6 	&	2 	&	0.0 	&&	\textbf{611}	&	13.5 	&	2 	&	0.0 	\\
 1010010-10	&&	\textbf{593}	&	1057.0 	&	593	&	0.0 	&&	\textbf{593}	&	3.5 	&	0.0 	&&	\textbf{593}	&	5.9 	&	0.0 	&&	\textbf{593}	&	5.9 	&	1 	&	0.0 	&&	\textbf{593}	&	5.9 	&	1 	&	0.0 	\\
 1010010-11	&&	\textbf{552}	&	1912.1 	&	552	&	0.0 	&&	\textbf{552}	&	3.3 	&	0.0 	&&	\textbf{552}	&	4.1 	&	0.0 	&&	\textbf{552}	&	4.1 	&	1 	&	0.0 	&&	\textbf{552}	&	4.1 	&	1 	&	0.0 	\\
 1010010-12	&&	\textbf{661}	&	3115.1 	&	661	&	0.0 	&&	\textbf{661}	&	1.6 	&	0.0 	&&	\textbf{661}	&	3.0 	&	0.0 	&&	\textbf{661}	&	3.0 	&	1 	&	0.0 	&&	\textbf{661}	&	3.0 	&	1 	&	0.0 	\\
 1010010-13	&&	\textbf{523}	&	1485.7 	&	523	&	0.0 	&&	\textbf{523}	&	6.0 	&	0.0 	&&	\textbf{523}	&	2.7 	&	0.0 	&&	\textbf{523}	&	2.7 	&	1 	&	0.0 	&&	\textbf{523}	&	2.7 	&	1 	&	0.0 	\\
 1010010-14	&&		936		&	3445.7 	&	305	&	61.0 	&&	\textbf{729}	&	9.8 	&	49.9 	&&	\textbf{729}	&	10.6 	&	58.2 	&&	\textbf{729}	&	10.6 	&	1 	&	49.9 	&&	\textbf{729}	&	10.6 	&	1 	&	49.9 	\\
 1010010-15	&&	\textbf{724}	&	3071.9 	&	724	&	0.0 	&&	\textbf{724}	&	7.9 	&	0.0 	&&	\textbf{724}	&	9.7 	&	0.0 	&&	\textbf{724}	&	9.7 	&	1 	&	0.0 	&&	\textbf{724}	&	9.7 	&	1 	&	0.0 	\\
 1010010-16	&&	\textbf{721}	&	2072.1 	&	721	&	0.0 	&&		774		&	10.5 	&	6.8 	&&		722 		&	11.2 	&	0.1 	&&	\textbf{721}	&	23.4 	&	2 	&	0.0 	&&	\textbf{721}	&	22.5 	&	2 	&	0.0 	\\
 1010010-17	&&		694		&	3546.0 	&	112	&	47.6 	&&		728		&	25.8 	&	50.0 	&&	\textbf{608}	&	17.6 	&	81.6 	&&	\textbf{608}	&	17.6 	&	1 	&	40.1 	&&	\textbf{608}	&	17.6 	&	1 	&	40.1 	\\
 1010010-18	&&	\textbf{447}	&	2487.3 	&	447	&	0.0 	&&	\textbf{447}	&	9.1 	&	0.0 	&&	\textbf{447}	&	5.6 	&	0.0 	&&	\textbf{447}	&	5.6 	&	1 	&	0.0 	&&	\textbf{447}	&	5.6 	&	1 	&	0.0 	\\
 1010010-19	&&	\textbf{553}	&	531.1 	&	553	&	0.0 	&&		651		&	8.8 	&	15.1 	&&	\textbf{553}	&	10.0 	&	0.0 	&&	\textbf{553}	&	10.0 	&	1 	&	0.0 	&&	\textbf{553}	&	10.0 	&	1 	&	0.0 	\\
 1010010-20	&&		768		&	3600.0 	&	59	&	60.2 	&&		611		&	30.5 	&	49.9 	&&		577 		&	24.4 	&	89.8 	&&	\textbf{568}	&	477.7 	&	10 	&	46.1 	&&	\textbf{568}	&	443.2 	&	10 	&	46.1 	\\
 1010010-21	&&	\textbf{365}	&	78.9 	&	365	&	0.0 	&&		386		&	1.0 	&	5.4 	&&		386 		&	1.1 	&	5.4 	&&	\textbf{365}	&	3.7 	&	2 	&	0.0 	&&	\textbf{365}	&	4.5 	&	2 	&	0.0 	\\
 1010010-22	&&	\textbf{569}	&	1834.2 	&	569	&	0.0 	&&		572		&	5.4 	&	0.5 	&&		572 		&	13.2 	&	0.5 	&&	\textbf{569}	&	79.5 	&	4 	&	0.0 	&&	\textbf{569}	&	77.5 	&	4 	&	0.0 	\\
 1010010-23	&&	\textbf{371}	&	714.5 	&	371	&	0.0 	&&	\textbf{371}	&	3.7 	&	0.0 	&&	\textbf{371}	&	3.0 	&	0.0 	&&	\textbf{371}	&	3.0 	&	1 	&	0.0 	&&	\textbf{371}	&	3.0 	&	1 	&	0.0 	\\
 1010010-24	&&	\textbf{449}	&	527.1 	&	449	&	0.0 	&&	\textbf{449}	&	2.7 	&	0.0 	&&	\textbf{449}	&	3.9 	&	0.0 	&&	\textbf{449}	&	3.9 	&	1 	&	0.0 	&&	\textbf{449}	&	3.9 	&	1 	&	0.0 	\\
 1010010-25	&&	\textbf{280}	&	183.7 	&	280	&	0.0 	&&		337		&	1.6 	&	16.9 	&&	\textbf{280}	&	1.1 	&	0.0 	&&	\textbf{280}	&	1.1 	&	1 	&	0.0 	&&	\textbf{280}	&	1.1 	&	1 	&	0.0 	\\
 1010010-26	&&	\textbf{466}	&	1497.9 	&	466	&	0.0 	&&		552		&	9.2 	&	15.6 	&&	\textbf{466}	&	6.8 	&	0.0 	&&	\textbf{466}	&	6.8 	&	1 	&	0.0 	&&	\textbf{466}	&	6.8 	&	1 	&	0.0 	\\
 1010010-27	&&	\textbf{512}	&	949.7 	&	512	&	0.0 	&&		554		&	6.9 	&	7.6 	&&	\textbf{512}	&	7.4 	&	0.0 	&&	\textbf{512}	&	7.4 	&	1 	&	0.0 	&&	\textbf{512}	&	7.4 	&	1 	&	0.0 	\\
 1010010-28	&&	\textbf{335}	&	156.9 	&	335	&	0.0 	&&	\textbf{335}	&	0.7 	&	0.0 	&&	\textbf{335}	&	1.0 	&	0.0 	&&	\textbf{335}	&	1.0 	&	1 	&	0.0 	&&	\textbf{335}	&	1.0 	&	1 	&	0.0 	\\
 1010010-29	&&	\textbf{396}	&	394.4 	&	396	&	0.0 	&&	\textbf{396}	&	1.8 	&	0.0 	&&	\textbf{396}	&	0.9 	&	0.0 	&&	\textbf{396}	&	0.9 	&	1 	&	0.0 	&&	\textbf{396}	&	0.9 	&	1 	&	0.0 	\\
 1010010-30	&&	\textbf{310}	&	122.5 	&	310	&	0.0 	&&	\textbf{310}	&	1.2 	&	0.0 	&&	\textbf{310}	&	2.0 	&	0.0 	&&	\textbf{310}	&	2.0 	&	1 	&	0.0 	&&	\textbf{310}	&	2.0 	&	1 	&	0.0 	\\
\hline																																															
 1010020-01	&&		1990		&	2696.6 	&	1230	&	38.2 	&&		1877		&	7.0 	&	34.5 	&&		1883 		&	36.4 	&	34.7 	&&	\textbf{1827}	&	76.2 	&	2 	&	32.7 	&&	\textbf{1827}	&	76.3 	&	2 	&	32.7 	\\
 1010020-02	&&		1798		&	3240.1 	&	1126	&	37.4 	&&		1746		&	4.9 	&	35.5 	&&	\textbf{1666}	&	14.3 	&	32.4 	&&	\textbf{1666}	&	14.3 	&	1 	&	32.4 	&&	\textbf{1666}	&	14.3 	&	1 	&	32.4 	\\
 1010020-03	&&	\textbf{1671}	&	1065.1 	&	1671	&	0.0 	&&	\textbf{1671}	&	7.9 	&	0.0 	&&	\textbf{1671}	&	5.6 	&	0.0 	&&	\textbf{1671}	&	5.6 	&	1 	&	0.0 	&&	\textbf{1671}	&	5.6 	&	1 	&	0.0 	\\
 1010020-04	&&		2052		&	1063.6 	&	1198	&	41.6 	&&		2074		&	31.3 	&	42.2 	&&		2082 		&	92.8 	&	42.5 	&&	\textbf{2023}	&	461.8 	&	4 	&	40.8 	&&	\textbf{2023}	&	465.2 	&	4 	&	40.8 	\\
 1010020-05	&&		2213		&	1174.4 	&	1209	&	45.4 	&&	\textbf{1870}	&	0.7 	&	35.3 	&&	\textbf{1870}	&	4.9 	&	35.3 	&&	\textbf{1870}	&	4.9 	&	1 	&	35.3 	&&	\textbf{1870}	&	4.9 	&	1 	&	35.3 	\\
 1010020-06	&&		2514		&	3600.0 	&	1361	&	45.9 	&&		2282		&	7.5 	&	40.4 	&&	\textbf{2232}	&	44.6 	&	39.0 	&&	\textbf{2232}	&	44.6 	&	1 	&	39.0 	&&	\textbf{2232}	&	44.6 	&	1 	&	39.0 	\\
 1010020-07	&&		2019		&	2043.1 	&	1481	&	26.6 	&&		2019		&	3.8 	&	26.6 	&&		2019 		&	7.9 	&	26.6 	&&	\textbf{1969}	&	22.9 	&	2 	&	24.8 	&&	\textbf{1969}	&	21.9 	&	2 	&	24.8 	\\
 1010020-08	&&		2243		&	2544.7 	&	1490	&	33.6 	&&	\textbf{1997}	&	5.0 	&	25.4 	&&	\textbf{1997}	&	18.4 	&	25.4 	&&	\textbf{1997}	&	18.4 	&	1 	&	25.4 	&&	\textbf{1997}	&	18.4 	&	1 	&	25.4 	\\
 1010020-09	&&	\textbf{1652}	&	2510.8 	&	1652	&	0.0 	&&		1687		&	1.2 	&	2.1 	&&		1687 		&	2.8 	&	2.1 	&&	\textbf{1652}	&	10.4 	&	2 	&	0.0 	&&	\textbf{1652}	&	10.1 	&	2 	&	0.0 	\\
 1010020-10	&&		2424		&	3276.4 	&	1279	&	47.2 	&&		2318		&	8.8 	&	44.8 	&&	\textbf{2230}	&	27.8 	&	42.6 	&&	\textbf{2230}	&	27.8 	&	1 	&	42.6 	&&	\textbf{2230}	&	27.8 	&	1 	&	42.6 	\\
 1010020-11	&&		2042		&	2556.9 	&	810	&	54.3 	&&		1866		&	23.3 	&	50.0 	&&		1795 		&	74.9 	&	54.9 	&&	\textbf{1753}	&	450.4 	&	5 	&	46.8 	&&	\textbf{1753}	&	437.2 	&	5 	&	46.8 	\\
 1010020-12	&&		1916		&	3088.7 	&	804	&	56.3 	&&		1675		&	7.8 	&	50.0 	&&		1643 		&	63.3 	&	51.1 	&&	\textbf{1636}	&	198.4 	&	3 	&	48.8 	&&	\textbf{1636}	&	203.0 	&	3 	&	48.8 	\\
 1010020-13	&&		1519		&	3171.3 	&	565	&	60.2 	&&	\textbf{1210}	&	2.3 	&	50.0 	&&	\textbf{1210}	&	2.9 	&	53.3 	&&	\textbf{1210}	&	2.9 	&	1 	&	50.0 	&&	\textbf{1210}	&	2.9 	&	1 	&	50.0 	\\
 1010020-14	&&		2659		&	3098.9 	&	976	&	59.6 	&&	\textbf{2148}	&	21.6 	&	50.0 	&&		2152 		&	71.8 	&	54.6 	&&	\textbf{2148}	&	287.7 	&	3 	&	50.0 	&&	\textbf{2148}	&	261.4 	&	3 	&	50.0 	\\
 1010020-15	&&		1902		&	1672.3 	&	1041	&	45.3 	&&		1793		&	7.9 	&	41.9 	&&	\textbf{1656}	&	11.0 	&	37.1 	&&	\textbf{1656}	&	11.0 	&	1 	&	37.1 	&&	\textbf{1656}	&	11.0 	&	1 	&	37.1 	\\
 1010020-16	&&		1892		&	3510.8 	&	759	&	54.4 	&&		1725		&	9.9 	&	50.0 	&&	\textbf{1595}	&	4.8 	&	52.4 	&&	\textbf{1595}	&	4.8 	&	1 	&	45.9 	&&	\textbf{1595}	&	4.8 	&	1 	&	45.9 	\\
 1010020-17	&&		2583		&	5901.5 	&	1063	&	58.8 	&&		2005		&	2.2 	&	47.0 	&&	\textbf{1954}	&	6.7 	&	45.6 	&&	\textbf{1954}	&	6.7 	&	1 	&	45.6 	&&	\textbf{1954}	&	6.7 	&	1 	&	45.6 	\\
 1010020-18	&&		1866		&	1591.0 	&	1133	&	39.3 	&&		1826		&	6.5 	&	38.0 	&&	\textbf{1718}	&	6.6 	&	34.1 	&&	\textbf{1718}	&	6.6 	&	1 	&	34.1 	&&	\textbf{1718}	&	6.6 	&	1 	&	34.1 	\\
 1010020-19	&&		1519		&	4972.5 	&	564	&	56.4 	&&		1324		&	6.9 	&	50.0 	&&	\textbf{1293}	&	7.8 	&	56.4 	&&	\textbf{1293}	&	7.8 	&	1 	&	48.8 	&&	\textbf{1293}	&	7.8 	&	1 	&	48.8 	\\
 1010020-20	&&		1628		&	5876.1 	&	562	&	57.0 	&&		1399		&	19.7 	&	50.0 	&&	\textbf{1389}	&	10.2 	&	59.5 	&&	\textbf{1389}	&	10.2 	&	1 	&	49.6 	&&	\textbf{1389}	&	10.2 	&	1 	&	49.6 	\\
 1010020-21	&&	\textbf{942}	&	2277.9 	&	789	&	16.2 	&&		949		&	2.6 	&	16.9 	&&	\textbf{942}	&	3.6 	&	16.2 	&&	\textbf{942}	&	3.6 	&	1 	&	16.2 	&&	\textbf{942}	&	3.6 	&	1 	&	16.2 	\\
 1010020-22	&&		1625		&	2786.9 	&	861	&	47.0 	&&		1609		&	23.6 	&	46.5 	&&		1540 		&	77.6 	&	44.1 	&&	\textbf{1528}	&	169.8 	&	2 	&	43.7 	&&	\textbf{1528}	&	169.6 	&	2 	&	43.7 	\\
 1010020-23	&&		1045		&	2962.4 	&	779	&	25.5 	&&		1042		&	3.0 	&	25.2 	&&	\textbf{988}	&	3.9 	&	21.2 	&&	\textbf{988}	&	3.9 	&	1 	&	21.2 	&&	\textbf{988}	&	3.9 	&	1 	&	21.2 	\\
 1010020-24	&&		1563		&	3515.6 	&	1081	&	30.8 	&&	\textbf{1489}	&	1.0 	&	27.4 	&&	\textbf{1489}	&	12.9 	&	27.4 	&&	\textbf{1489}	&	12.9 	&	1 	&	27.4 	&&	\textbf{1489}	&	12.9 	&	1 	&	27.4 	\\
 1010020-25	&&	\textbf{901}	&	522.9 	&	901	&	0.0 	&&	\textbf{901}	&	0.7 	&	0.0 	&&	\textbf{901}	&	1.1 	&	0.0 	&&	\textbf{901}	&	1.1 	&	1 	&	0.0 	&&	\textbf{901}	&	1.1 	&	1 	&	0.0 	\\
 1010020-26	&&		1408		&	3237.9 	&	1039	&	26.2 	&&	\textbf{1352}	&	4.9 	&	23.2 	&&		1354 		&	4.7 	&	23.3 	&&	\textbf{1352}	&	11.3 	&	2 	&	23.2 	&&	\textbf{1352}	&	10.5 	&	2 	&	23.2 	\\
 1010020-27	&&		1376		&	1086.5 	&	813	&	40.9 	&&		1235		&	4.0 	&	34.2 	&&	\textbf{1227}	&	8.6 	&	33.7 	&&	\textbf{1227}	&	8.6 	&	1 	&	33.7 	&&	\textbf{1227}	&	8.6 	&	1 	&	33.7 	\\
 1010020-28	&&	\textbf{924}	&	1032.3 	&	924	&	0.0 	&&		963		&	1.8 	&	4.0 	&&	\textbf{924}	&	2.0 	&	0.0 	&&	\textbf{924}	&	2.0 	&	1 	&	0.0 	&&	\textbf{924}	&	2.0 	&	1 	&	0.0 	\\
 1010020-29	&&	\textbf{1396}	&	1786.9 	&	1396	&	0.0 	&&	\textbf{1396}	&	1.6 	&	0.0 	&&	\textbf{1396}	&	2.9 	&	0.0 	&&	\textbf{1396}	&	2.9 	&	1 	&	0.0 	&&	\textbf{1396}	&	2.9 	&	1 	&	0.0 	\\
 1010020-30	&&	\textbf{981}	&	406.6 	&	981	&	0.0 	&&		988		&	1.6 	&	0.7 	&&		988 		&	0.9 	&	0.7 	&&	\textbf{981}	&	2.2 	&	2 	&	0.0 	&&	\textbf{981}	&	2.1 	&	2 	&	0.0 	\\
\hline																																															
 1010030-01	&&		3218		&	2668.4 	&	2137	&	33.6 	&&	\textbf{3027}	&	0.8 	&	29.4 	&&	\textbf{3027}	&	16.1 	&	29.4 	&&	\textbf{3027}	&	16.1 	&	1 	&	29.4 	&&	\textbf{3027}	&	16.1 	&	1 	&	29.4 	\\
 1010030-02	&&		3437		&	3582.6 	&	2071	&	39.7 	&&		3306		&	7.6 	&	37.4 	&&	\textbf{3219}	&	26.7 	&	35.7 	&&	\textbf{3219}	&	26.7 	&	1 	&	35.7 	&&	\textbf{3219}	&	26.7 	&	1 	&	35.7 	\\
 1010030-03	&&	\textbf{3200}	&	660.3 	&	2267	&	29.2 	&&	\textbf{3200}	&	5.1 	&	29.2 	&&	\textbf{3200}	&	8.5 	&	29.2 	&&	\textbf{3200}	&	8.5 	&	1 	&	29.2 	&&	\textbf{3200}	&	8.5 	&	1 	&	29.2 	\\
 1010030-04	&&		4224		&	340.3 	&	2623	&	37.9 	&&		4069		&	24.0 	&	35.5 	&&		4068 		&	348.7 	&	35.5 	&&	\textbf{4057}	&	2325.0 	&	5 	&	35.3 	&&	\textbf{4057}	&	2534.7 	&	5 	&	35.3 	\\
 1010030-05	&&		3590		&	1434.4 	&	2107	&	41.3 	&&		3530		&	7.0 	&	40.3 	&&	\textbf{3336}	&	22.7 	&	36.8 	&&	\textbf{3336}	&	22.7 	&	1 	&	36.8 	&&	\textbf{3336}	&	22.7 	&	1 	&	36.8 	\\
 1010030-06	&&		4557		&	408.8 	&	2667	&	41.5 	&&		4415		&	16.8 	&	39.6 	&&		4328 		&	286.0 	&	38.4 	&&	\textbf{4284}	&	883.6 	&	2 	&	37.7 	&&	\textbf{4284}	&	864.1 	&	2 	&	37.7 	\\
 1010030-07	&&		2790		&	2027.7 	&	1660	&	40.5 	&&		2845		&	4.6 	&	41.7 	&&	\textbf{2628}	&	6.5 	&	36.8 	&&	\textbf{2628}	&	6.5 	&	1 	&	36.8 	&&	\textbf{2628}	&	6.5 	&	1 	&	36.8 	\\
 1010030-08	&&	\textbf{3183}	&	2124.8 	&	2263	&	28.9 	&&		3268		&	7.9 	&	30.8 	&&		3229 		&	18.3 	&	29.9 	&&	\textbf{3183}	&	39.0 	&	2 	&	28.9 	&&	\textbf{3183}	&	37.2 	&	2 	&	28.9 	\\
 1010030-09	&&		3255		&	2539.4 	&	2174	&	33.2 	&&	\textbf{2937}	&	0.4 	&	26.0 	&&	\textbf{2937}	&	8.3 	&	26.0 	&&	\textbf{2937}	&	8.3 	&	1 	&	26.0 	&&	\textbf{2937}	&	8.3 	&	1 	&	26.0 	\\
 1010030-10	&&		3773		&	848.0 	&	2253	&	40.3 	&&	\textbf{3601}	&	2.8 	&	37.4 	&&	\textbf{3601}	&	33.5 	&	37.4 	&&	\textbf{3601}	&	33.5 	&	1 	&	37.4 	&&	\textbf{3601}	&	33.5 	&	1 	&	37.4 	\\
 1010030-11	&&		3375		&	1677.0 	&	1861	&	44.9 	&&		3326		&	8.6 	&	44.0 	&&		3278 		&	130.4 	&	43.2 	&&	\textbf{3263}	&	427.1 	&	3 	&	43.0 	&&	\textbf{3263}	&	424.0 	&	3 	&	43.0 	\\
 1010030-12	&&		3965		&	973.2 	&	2152	&	45.7 	&&		3622		&	9.2 	&	40.6 	&&		3618 		&	137.7 	&	40.5 	&&	\textbf{3608}	&	297.6 	&	2 	&	40.4 	&&	\textbf{3608}	&	300.7 	&	2 	&	40.4 	\\
 1010030-13	&&		3584		&	2495.2 	&	1914	&	46.6 	&&		3623		&	5.3 	&	47.2 	&&	\textbf{3479}	&	173.1 	&	45.0 	&&	\textbf{3479}	&	172.9 	&	1 	&	45.0 	&&	\textbf{3479}	&	173.1 	&	1 	&	45.0 	\\
 1010030-14	&&		4587		&	475.4 	&	2077	&	54.7 	&&	\textbf{3937}	&	32.8 	&	47.2 	&&	\textbf{3937}	&	825.0 	&	47.2 	&&	\textbf{3937}	&	823.1 	&	1 	&	47.2 	&&	\textbf{3937}	&	825.0 	&	1 	&	47.2 	\\
 1010030-15	&&		3650		&	1745.7 	&	1987	&	45.6 	&&	\textbf{3337}	&	2.7 	&	40.5 	&&	\textbf{3337}	&	127.5 	&	40.5 	&&	\textbf{3337}	&	127.7 	&	1 	&	40.5 	&&	\textbf{3337}	&	127.5 	&	1 	&	40.5 	\\
 1010030-16	&&		3386		&	1511.6 	&	2054	&	39.3 	&&		3364		&	7.8 	&	38.9 	&&		3294 		&	47.8 	&	37.6 	&&	\textbf{3273}	&	111.7 	&	2 	&	37.2 	&&	\textbf{3273}	&	109.2 	&	2 	&	37.2 	\\
 1010030-17	&&		3956		&	1830.9 	&	1913	&	51.6 	&&	\textbf{3384}	&	18.6 	&	43.5 	&&	\textbf{3384}	&	137.2 	&	43.5 	&&	\textbf{3384}	&	137.7 	&	1 	&	43.5 	&&	\textbf{3384}	&	137.2 	&	1 	&	43.5 	\\
 1010030-18	&&		3600		&	729.0 	&	1228	&	65.0 	&&	\textbf{2522}	&	11.2 	&	50.0 	&&	\textbf{2522}	&	17.3 	&	51.3 	&&	\textbf{2522}	&	17.5 	&	1 	&	50.0 	&&	\textbf{2522}	&	17.3 	&	1 	&	50.0 	\\
 1010030-19	&&		3931		&	3600.0 	&	1746	&	55.6 	&&		3472		&	8.2 	&	49.7 	&&		3471 		&	207.1 	&	49.7 	&&	\textbf{3461}	&	902.6 	&	4 	&	49.6 	&&	\textbf{3461}	&	910.2 	&	4 	&	49.6 	\\
 1010030-20	&&		3422		&	2848.7 	&	1611	&	52.9 	&&		3190		&	11.0 	&	49.5 	&&	\textbf{3142}	&	93.1 	&	48.7 	&&	\textbf{3142}	&	93.8 	&	1 	&	48.7 	&&	\textbf{3142}	&	93.1 	&	1 	&	48.7 	\\
 1010030-21	&&		2327		&	1393.7 	&	1575	&	32.3 	&&		2264		&	1.5 	&	30.4 	&&		2264 		&	10.8 	&	30.4 	&&	\textbf{2251}	&	48.5 	&	3 	&	30.0 	&&	\textbf{2251}	&	48.9 	&	3 	&	30.0 	\\
 1010030-22	&&	\textbf{3200}	&	3600.0 	&	2178	&	31.9 	&&	\textbf{3200}	&	7.0 	&	31.9 	&&		3243 		&	565.3 	&	32.8 	&&		3214		&	3600.0 	&	6 	&	32.2 	&&		3214		&	3600.0 	&	6 	&	32.2 	\\
 1010030-23	&&		2610		&	2990.7 	&	1703	&	34.8 	&&		2473		&	3.3 	&	31.1 	&&	\textbf{2380}	&	10.7 	&	28.4 	&&	\textbf{2380}	&	10.7 	&	1 	&	28.4 	&&	\textbf{2380}	&	10.7 	&	1 	&	28.4 	\\
 1010030-24	&&		2638		&	618.9 	&	1790	&	32.1 	&&	\textbf{2385}	&	8.1 	&	24.9 	&&	\textbf{2385}	&	23.4 	&	24.9 	&&	\textbf{2385}	&	23.4 	&	1 	&	24.9 	&&	\textbf{2385}	&	23.4 	&	1 	&	24.9 	\\
 1010030-25	&&	\textbf{2059}	&	3069.8 	&	2059	&	0.0 	&&	\textbf{2059}	&	0.5 	&	0.0 	&&	\textbf{2059}	&	2.8 	&	0.0 	&&	\textbf{2059}	&	2.8 	&	1 	&	0.0 	&&	\textbf{2059}	&	2.8 	&	1 	&	0.0 	\\
 1010030-26	&&		2871		&	1071.0 	&	1687	&	41.2 	&&		2677		&	8.3 	&	37.0 	&&		2644 		&	26.0 	&	36.2 	&&	\textbf{2589}	&	55.5 	&	2 	&	34.8 	&&	\textbf{2589}	&	57.9 	&	2 	&	34.8 	\\
 1010030-27	&&		2970		&	937.0 	&	1636	&	44.9 	&&		2673		&	5.7 	&	38.8 	&&	\textbf{2659}	&	85.6 	&	38.5 	&&	\textbf{2659}	&	86.2 	&	1 	&	38.5 	&&	\textbf{2659}	&	85.6 	&	1 	&	38.5 	\\
 1010030-28	&&		2420		&	3131.4 	&	1596	&	34.0 	&&		2266		&	2.0 	&	29.6 	&&	\textbf{2099}	&	1.2 	&	24.0 	&&	\textbf{2099}	&	1.2 	&	1 	&	24.0 	&&	\textbf{2099}	&	1.2 	&	1 	&	24.0 	\\
 1010030-29	&&	\textbf{2092}	&	3600.0 	&	1855	&	11.3 	&&	\textbf{2092}	&	1.1 	&	11.3 	&&	\textbf{2092}	&	3.7 	&	11.3 	&&	\textbf{2092}	&	3.7 	&	1 	&	11.3 	&&	\textbf{2092}	&	3.7 	&	1 	&	11.3 	\\
 1010030-30	&&	\textbf{2134}	&	1998.1 	&	2134	&	0.0 	&&	\textbf{2134}	&	0.4 	&	0.0 	&&	\textbf{2134}	&	1.5 	&	0.0 	&&	\textbf{2134}	&	1.5 	&	1 	&	0.0 	&&	\textbf{2134}	&	1.5 	&	1 	&	0.0 	\\
\bhline{1pt}
\end{tabular}
\end{table}

\begin{table}[htbp]
\centering \fontsize{6pt}{6.8pt}\selectfont
\setlength{\tabcolsep}{3.6pt}
\caption{MMR-MKP results with $m=5,n=250$}
\label{tbl:mkp-05250}
\begin{tabular}{ll*{22}{r}}
\hline
                             && \multicolumn{4}{c}{B\&C}                  && \multicolumn{3}{c}{Fix}                  && \multicolumn{3}{c}{DS}                  && \multicolumn{4}{c}{iDS-H}                                                                  && \multicolumn{4}{c}{iDS-B}                                                                      \\ \cline{3-6} \cline{8-10} \cline{12-14} \cline{16-19} \cline{21-24} 
\multicolumn{1}{c}{instance} && \multicolumn{1}{c}{obj} & \multicolumn{1}{c}{time} & \multicolumn{1}{c}{LB} & \multicolumn{1}{c}{\%gap} && \multicolumn{1}{c}{obj} & \multicolumn{1}{c}{time} & \multicolumn{1}{c}{\%gap}  && \multicolumn{1}{c}{obj} & \multicolumn{1}{c}{time} & \multicolumn{1}{c}{\%gap}  && \multicolumn{1}{c}{obj} & \multicolumn{1}{c}{time} & \multicolumn{1}{c}{iter} & \multicolumn{1}{c}{\%gap} && \multicolumn{1}{c}{obj} & \multicolumn{1}{c}{time} & \multicolumn{1}{c}{iter} & \multicolumn{1}{c}{\%gap} \\ \bhline{1.5pt}
 0525010-01	&&		1176 		&	3284.8 	&	481 	&	58.6 	&&		973 		&	9.4 	&	49.9 	&&		976 		&	71.8 	&	50.7 	&&	\textbf{940}	&	145.4 	&	2	&	48.2 	&&	\textbf{940}	&	144.1 	&	2	&	48.2 	\\
 0525010-02	&&		1378 		&	3600.0 	&	501 	&	63.6 	&&		993 		&	25.4 	&	49.5 	&&		993 		&	59.1 	&	49.5 	&&	\textbf{978}	&	123.6 	&	2	&	48.8 	&&	\textbf{978}	&	129.8 	&	2	&	48.8 	\\
 0525010-03	&&		1288 		&	1073.5 	&	353 	&	62.7 	&&		959 		&	8.3 	&	49.9 	&&		959 		&	110.9 	&	63.2 	&&	\textbf{956}	&	530.2 	&	4	&	49.8 	&&	\textbf{956}	&	529.3 	&	4	&	49.8 	\\
 0525010-04	&&		1943 		&	574.9 	&	600 	&	66.0 	&&	\textbf{1321}	&	121.2 	&	50.0 	&&	\textbf{1321}	&	575.0 	&	54.6 	&&	\textbf{1321}	&	575.0 	&	1	&	50.0 	&&	\textbf{1321}	&	575.0 	&	1	&	50.0 	\\
 0525010-05	&&		1203 		&	3133.8 	&	311 	&	61.0 	&&		937 		&	17.0 	&	49.9 	&&	\textbf{919}	&	65.7 	&	66.2 	&&	\textbf{919}	&	65.7 	&	1	&	49.0 	&&	\textbf{919}	&	65.7 	&	1	&	49.0 	\\
 0525010-06	&&		1554 		&	466.3 	&	645 	&	58.5 	&&	\textbf{1209}	&	18.1 	&	46.7 	&&	\textbf{1209}	&	202.1 	&	46.7 	&&	\textbf{1209}	&	202.1 	&	1	&	46.7 	&&	\textbf{1209}	&	202.1 	&	1	&	46.7 	\\
 0525010-07	&&		1472 		&	1959.1 	&	696 	&	52.7 	&&		1263 		&	63.8 	&	44.9 	&&	\textbf{1218}	&	298.1 	&	42.9 	&&	\textbf{1218}	&	298.1 	&	1	&	42.9 	&&	\textbf{1218}	&	298.1 	&	1	&	42.9 	\\
 0525010-08	&&		1580 		&	3221.4 	&	534 	&	66.0 	&&		1073 		&	10.5 	&	50.0 	&&		1073 		&	88.3 	&	50.2 	&&	\textbf{1067}	&	302.9 	&	3	&	49.7 	&&	\textbf{1067}	&	282.4 	&	3	&	49.7 	\\
 0525010-09	&&		1480 		&	2924.7 	&	661 	&	55.3 	&&	\textbf{1154}	&	7.6 	&	42.7 	&&	\textbf{1154}	&	97.8 	&	42.7 	&&	\textbf{1154}	&	97.8 	&	1	&	42.7 	&&	\textbf{1154}	&	97.8 	&	1	&	42.7 	\\
 0525010-10	&&		1342 		&	3600.0 	&	635 	&	52.7 	&&	\textbf{1086}	&	11.2 	&	41.5 	&&	\textbf{1086}	&	84.9 	&	41.5 	&&	\textbf{1086}	&	84.9 	&	1	&	41.5 	&&	\textbf{1086}	&	84.9 	&	1	&	41.5 	\\
 0525010-11	&&		1750 		&	3600.0 	&	644 	&	63.2 	&&		1240 		&	42.7 	&	48.1 	&&	\textbf{1233}	&	347.3 	&	47.8 	&&	\textbf{1233}	&	347.3 	&	1	&	47.8 	&&	\textbf{1233}	&	347.3 	&	1	&	47.8 	\\
 0525010-12	&&		1046 		&	2894.1 	&	573 	&	45.2 	&&	\textbf{947}	&	7.3 	&	39.5 	&&	\textbf{947}	&	79.7 	&	39.5 	&&	\textbf{947}	&	79.7 	&	1	&	39.5 	&&	\textbf{947}	&	79.7 	&	1	&	39.5 	\\
 0525010-13	&&		1854 		&	3600.0 	&	504 	&	72.8 	&&	\textbf{1004}	&	9.2 	&	49.8 	&&	\textbf{1004}	&	110.9 	&	49.8 	&&	\textbf{1004}	&	110.9 	&	1	&	49.8 	&&	\textbf{1004}	&	110.9 	&	1	&	49.8 	\\
 0525010-14	&&		1658 		&	3600.0 	&	535 	&	67.7 	&&	\textbf{1021}	&	7.9 	&	47.6 	&&	\textbf{1021}	&	83.8 	&	47.6 	&&	\textbf{1021}	&	83.8 	&	1	&	47.6 	&&	\textbf{1021}	&	83.8 	&	1	&	47.6 	\\
 0525010-15	&&		2845 		&	1111.0 	&	531 	&	79.6 	&&		1162 		&	19.3 	&	50.0 	&&	\textbf{1135}	&	671.2 	&	53.2 	&&	\textbf{1135}	&	671.2 	&	1	&	48.8 	&&	\textbf{1135}	&	671.2 	&	1	&	48.8 	\\
 0525010-16	&&		1653 		&	1389.0 	&	581 	&	64.9 	&&		1072 		&	50.8 	&	45.8 	&&		1072 		&	481.9 	&	45.8 	&&	\textbf{1061}	&	983.6 	&	2	&	45.2 	&&	\textbf{1061}	&	1003.1 	&	2	&	45.2 	\\
 0525010-17	&&		1397 		&	3285.8 	&	573 	&	59.0 	&&		1103 		&	17.1 	&	48.1 	&&	\textbf{1099}	&	155.0 	&	47.9 	&&	\textbf{1099}	&	155.0 	&	1	&	47.9 	&&	\textbf{1099}	&	155.0 	&	1	&	47.9 	\\
 0525010-18	&&		2622 		&	942.6 	&	648 	&	71.5 	&&		1496 		&	82.3 	&	50.0 	&&	\textbf{1476}	&	1028.9 	&	56.1 	&&	\textbf{1476}	&	1028.9 	&	1	&	49.3 	&&	\textbf{1476}	&	1028.9 	&	1	&	49.3 	\\
 0525010-19	&&		2819 		&	3600.0 	&	311 	&	85.3 	&&	\textbf{828}	&	4.3 	&	50.0 	&&	\textbf{828}	&	33.6 	&	62.4 	&&	\textbf{828}	&	33.6 	&	1	&	50.0 	&&	\textbf{828}	&	33.6 	&	1	&	50.0 	\\
 0525010-20	&&		1355 		&	3600.0 	&	448 	&	66.3 	&&		911 		&	15.7 	&	49.9 	&&		921 		&	84.2 	&	51.4 	&&	\textbf{904}	&	311.0 	&	3	&	49.6 	&&	\textbf{904}	&	302.8 	&	3	&	49.6 	\\
 0525010-21	&&		1407 		&	3600.0 	&	487 	&	64.5 	&&		999 		&	79.1 	&	49.9 	&&	\textbf{968}	&	315.2 	&	49.7 	&&	\textbf{968}	&	315.2 	&	1	&	48.3 	&&	\textbf{968}	&	315.2 	&	1	&	48.3 	\\
 0525010-22	&&		579 		&	3455.3 	&	376 	&	35.1 	&&	\textbf{543}	&	1.2 	&	30.8 	&&	\textbf{543}	&	10.5 	&	30.8 	&&	\textbf{543}	&	10.5 	&	1	&	30.8 	&&	\textbf{543}	&	10.5 	&	1	&	30.8 	\\
 0525010-23	&&		1169 		&	2469.9 	&	511 	&	56.3 	&&	\textbf{919}	&	36.6 	&	44.4 	&&	\textbf{919}	&	226.3 	&	44.4 	&&	\textbf{919}	&	226.3 	&	1	&	44.4 	&&	\textbf{919}	&	226.3 	&	1	&	44.4 	\\
 0525010-24	&&		908 		&	3169.6 	&	345 	&	61.1 	&&	\textbf{706}	&	18.5 	&	50.0 	&&		708 		&	41.6 	&	51.3 	&&	\textbf{706}	&	208.8 	&	3	&	50.0 	&&	\textbf{706}	&	197.2 	&	3	&	50.0 	\\
 0525010-25	&&		1123 		&	5206.2 	&	387 	&	65.5 	&&	\textbf{749}	&	13.3 	&	48.3 	&&	\textbf{749}	&	27.4 	&	48.3 	&&	\textbf{749}	&	27.4 	&	1	&	48.3 	&&	\textbf{749}	&	27.4 	&	1	&	48.3 	\\
 0525010-26	&&		959 		&	2131.1 	&	437 	&	54.4 	&&	\textbf{712}	&	3.3 	&	38.6 	&&	\textbf{712}	&	14.2 	&	38.6 	&&	\textbf{712}	&	14.2 	&	1	&	38.6 	&&	\textbf{712}	&	14.2 	&	1	&	38.6 	\\
 0525010-27	&&		667 		&	3167.1 	&	242 	&	60.0 	&&		533 		&	5.2 	&	49.9 	&&	\textbf{532}	&	34.4 	&	54.5 	&&	\textbf{532}	&	34.4 	&	1	&	49.8 	&&	\textbf{532}	&	34.4 	&	1	&	49.8 	\\
 0525010-28	&&		790 		&	3103.4 	&	550 	&	30.4 	&&	\textbf{766}	&	3.4 	&	28.2 	&&		790 		&	25.4 	&	30.4 	&&	\textbf{766}	&	62.7 	&	2	&	28.2 	&&	\textbf{766}	&	58.0 	&	2	&	28.2 	\\
 0525010-29	&&		771 		&	3050.9 	&	442 	&	42.7 	&&		703 		&	7.4 	&	37.1 	&&		685 		&	9.0 	&	35.5 	&&	\textbf{669}	&	24.7 	&	2	&	33.9 	&&	\textbf{669}	&	23.8 	&	2	&	33.9 	\\
 0525010-30	&&	\textbf{639}	&	1752.0 	&	543 	&	15.0 	&&		696 		&	7.5 	&	22.0 	&&	\textbf{639}	&	12.4 	&	15.0 	&&	\textbf{639}	&	12.4 	&	1	&	15.0 	&&	\textbf{639}	&	12.4 	&	1	&	15.0 	\\
\hline																																															
 0525020-01	&&		4758 		&	2158.3 	&	2057 	&	56.8 	&&	\textbf{3404}	&	3.6 	&	39.6 	&&	\textbf{3404}	&	818.9 	&	39.6 	&&	\textbf{3404}	&	818.9 	&	1	&	39.6 	&&	\textbf{3404}	&	818.9 	&	1	&	39.6 	\\
 0525020-02	&&		4977 		&	1013.2 	&	1728 	&	64.9 	&&	\textbf{3492}	&	7.7 	&	50.0 	&&	\textbf{3492}	&	1509.9 	&	50.5 	&&	\textbf{3492}	&	1509.9 	&	1	&	50.0 	&&	\textbf{3492}	&	1509.9 	&	1	&	50.0 	\\
 0525020-03	&&		5207 		&	1686.4 	&	2445 	&	53.0 	&&		4253 		&	17.6 	&	42.5 	&&	\textbf{4205}	&	2243.5 	&	41.9 	&&	\textbf{4205}	&	2243.5 	&	1	&	41.9 	&&	\textbf{4205}	&	2243.5 	&	1	&	41.9 	\\
 0525020-04	&&		6016 		&	2642.2 	&	3034 	&	49.6 	&&	\textbf{5025}	&	50.5 	&	39.6 	&&		5043 		&	3600.0 	&	39.8 	&&		5043		&	3600.0 	&	1	&	39.8 	&&		5043		&	3600.0 	&	1	&	39.8 	\\
 0525020-05	&&		6697 		&	468.7 	&	1720 	&	74.3 	&&		3317 		&	7.0 	&	48.1 	&&		3317 		&	709.1 	&	48.1 	&&	\textbf{3302}	&	1258.0 	&	2	&	47.9 	&&	\textbf{3302}	&	1181.1 	&	2	&	47.9 	\\
 0525020-06	&&		5263 		&	2611.0 	&	2316 	&	56.0 	&&		4020 		&	83.7 	&	42.4 	&&	\textbf{3962}	&	3600.0 	&	41.5 	&&	\textbf{3962}	&	3600.0 	&	1	&	41.5 	&&	\textbf{3962}	&	3600.0 	&	1	&	41.5 	\\
 0525020-07	&&		4810 		&	842.4 	&	1605 	&	65.5 	&&		3318 		&	12.0 	&	50.0 	&&	\textbf{3289}	&	1888.1 	&	51.2 	&&	\textbf{3289}	&	1888.1 	&	1	&	49.6 	&&	\textbf{3289}	&	1888.1 	&	1	&	49.6 	\\
 0525020-08	&&		4773 		&	2593.0 	&	2237 	&	53.1 	&&		3837 		&	59.4 	&	41.7 	&&	\textbf{3813}	&	3600.0 	&	41.3 	&&	\textbf{3813}	&	3600.0 	&	1	&	41.3 	&&	\textbf{3813}	&	3600.0 	&	1	&	41.3 	\\
 0525020-09	&&		4079 		&	3600.0 	&	1808 	&	55.7 	&&		3324 		&	6.7 	&	45.6 	&&	\textbf{3288}	&	637.7 	&	45.0 	&&	\textbf{3288}	&	637.7 	&	1	&	45.0 	&&	\textbf{3288}	&	637.7 	&	1	&	45.0 	\\
 0525020-10	&&		3805 		&	3600.0 	&	1770 	&	53.5 	&&	\textbf{3104}	&	9.1 	&	43.0 	&&	\textbf{3104}	&	1457.5 	&	43.0 	&&	\textbf{3104}	&	1457.5 	&	1	&	43.0 	&&	\textbf{3104}	&	1457.5 	&	1	&	43.0 	\\
 0525020-11	&&		7673 		&	1684.9 	&	2155 	&	71.3 	&&	\textbf{4408}	&	20.3 	&	50.0 	&&		4467 		&	3600.0 	&	51.8 	&&		4467		&	3600.0 	&	1	&	50.7 	&&		4467		&	3600.0 	&	1	&	50.7 	\\
 0525020-12	&&		4635 		&	387.8 	&	2045 	&	55.9 	&&		3477 		&	14.5 	&	41.2 	&&	\textbf{3427}	&	2635.2 	&	40.3 	&&	\textbf{3427}	&	2635.2 	&	1	&	40.3 	&&	\textbf{3427}	&	2635.2 	&	1	&	40.3 	\\
 0525020-13	&&		4583 		&	3600.0 	&	1937 	&	57.7 	&&	\textbf{3323}	&	5.6 	&	41.7 	&&		3346 		&	3600.0 	&	42.1 	&&		3346		&	3600.0 	&	1	&	42.1 	&&		3346		&	3600.0 	&	1	&	42.1 	\\
 0525020-14	&&		4671 		&	3379.2 	&	2406 	&	48.5 	&&		4134 		&	16.1 	&	41.8 	&&	\textbf{4093}	&	3600.0 	&	41.2 	&&	\textbf{4093}	&	3600.0 	&	1	&	41.2 	&&	\textbf{4093}	&	3600.0 	&	1	&	41.2 	\\
 0525020-15	&&		5615 		&	3600.0 	&	2386 	&	57.5 	&&		4390 		&	42.0 	&	45.6 	&&	\textbf{4365}	&	3600.0 	&	45.3 	&&	\textbf{4365}	&	3600.0 	&	1	&	45.3 	&&	\textbf{4365}	&	3600.0 	&	1	&	45.3 	\\
 0525020-16	&&		4905 		&	3310.5 	&	2220 	&	54.7 	&&		3816 		&	7.5 	&	41.8 	&&	\textbf{3774}	&	3600.0 	&	41.2 	&&	\textbf{3774}	&	3600.0 	&	1	&	41.2 	&&	\textbf{3774}	&	3600.0 	&	1	&	41.2 	\\
 0525020-17	&&		6666 		&	615.2 	&	1651 	&	75.2 	&&	\textbf{3161}	&	16.0 	&	47.8 	&&	\textbf{3161}	&	778.0 	&	47.8 	&&	\textbf{3161}	&	778.0 	&	1	&	47.8 	&&	\textbf{3161}	&	778.0 	&	1	&	47.8 	\\
 0525020-18	&&		5722 		&	3600.0 	&	2398 	&	58.1 	&&		4061 		&	18.5 	&	41.0 	&&	\textbf{4034}	&	3600.0 	&	40.6 	&&	\textbf{4034}	&	3600.0 	&	1	&	40.6 	&&	\textbf{4034}	&	3600.0 	&	1	&	40.6 	\\
 0525020-19	&&		7501 		&	449.4 	&	1770 	&	76.4 	&&		3371 		&	4.2 	&	47.5 	&&	\textbf{3362}	&	2077.0 	&	47.4 	&&	\textbf{3362}	&	2077.0 	&	1	&	47.4 	&&	\textbf{3362}	&	2077.0 	&	1	&	47.4 	\\
 0525020-20	&&		6733 		&	817.7 	&	1661 	&	75.3 	&&	\textbf{3315}	&	2.1 	&	49.9 	&&	\textbf{3315}	&	2196.7 	&	49.9 	&&	\textbf{3315}	&	2196.7 	&	1	&	49.9 	&&	\textbf{3315}	&	2196.7 	&	1	&	49.9 	\\
 0525020-21	&&		4950 		&	2804.3 	&	2088 	&	57.8 	&&	\textbf{3516}	&	21.1 	&	40.6 	&&		3543 		&	3600.0 	&	41.1 	&&		3543		&	3600.0 	&	1	&	41.1 	&&		3543		&	3600.0 	&	1	&	41.1 	\\
 0525020-22	&&		3559 		&	2303.4 	&	1366 	&	61.6 	&&		2435 		&	13.2 	&	43.9 	&&	\textbf{2411}	&	179.0 	&	43.3 	&&	\textbf{2411}	&	179.0 	&	1	&	43.3 	&&	\textbf{2411}	&	179.0 	&	1	&	43.3 	\\
 0525020-23	&&		3593 		&	2388.0 	&	1708 	&	52.5 	&&		2786 		&	7.3 	&	38.7 	&&	\textbf{2776}	&	1204.1 	&	38.5 	&&	\textbf{2776}	&	1204.1 	&	1	&	38.5 	&&	\textbf{2776}	&	1204.1 	&	1	&	38.5 	\\
 0525020-24	&&		3556 		&	3390.4 	&	1420 	&	60.1 	&&	\textbf{2405}	&	7.5 	&	41.0 	&&	\textbf{2405}	&	1607.5 	&	41.0 	&&	\textbf{2405}	&	1607.5 	&	1	&	41.0 	&&	\textbf{2405}	&	1607.5 	&	1	&	41.0 	\\
 0525020-25	&&		4129 		&	3001.2 	&	1795 	&	56.5 	&&	\textbf{3023}	&	17.2 	&	40.6 	&&	\textbf{3023}	&	3600.0 	&	40.6 	&&	\textbf{3023}	&	3600.0 	&	1	&	40.6 	&&	\textbf{3023}	&	3600.0 	&	1	&	40.6 	\\
 0525020-26	&&		4075 		&	3046.0 	&	1578 	&	61.3 	&&		2785 		&	50.4 	&	43.3 	&&	\textbf{2772}	&	3600.0 	&	43.1 	&&	\textbf{2772}	&	3600.0 	&	1	&	43.1 	&&	\textbf{2772}	&	3600.0 	&	1	&	43.1 	\\
 0525020-27	&&		3005 		&	3057.1 	&	1328 	&	55.8 	&&		2254 		&	2.6 	&	41.1 	&&		2252 		&	338.5 	&	41.0 	&&	\textbf{2247}	&	1622.3 	&	5	&	40.9 	&&	\textbf{2247}	&	1589.9 	&	5	&	40.9 	\\
 0525020-28	&&		2999 		&	3541.9 	&	1594 	&	46.8 	&&		2550 		&	11.3 	&	37.5 	&&	\textbf{2514}	&	2314.7 	&	36.6 	&&	\textbf{2514}	&	2314.7 	&	1	&	36.6 	&&	\textbf{2514}	&	2314.7 	&	1	&	36.6 	\\
 0525020-29	&&		2959 		&	1209.4 	&	1230 	&	58.4 	&&	\textbf{2190}	&	1.0 	&	43.8 	&&	\textbf{2190}	&	35.3 	&	43.8 	&&	\textbf{2190}	&	35.3 	&	1	&	43.8 	&&	\textbf{2190}	&	35.3 	&	1	&	43.8 	\\
 0525020-30	&&		3547 		&	3981.8 	&	1292 	&	63.6 	&&	\textbf{2207}	&	3.4 	&	41.5 	&&	\textbf{2207}	&	96.7 	&	41.5 	&&	\textbf{2207}	&	96.7 	&	1	&	41.5 	&&	\textbf{2207}	&	96.7 	&	1	&	41.5 	\\
\hline																																															
 0525030-01	&&		9578 		&	2959.9 	&	4616 	&	51.8 	&&		7806 		&	39.4 	&	40.9 	&&	\textbf{7793}	&	3600.0 	&	40.8 	&&	\textbf{7793}	&	3600.0 	&	1	&	40.8 	&&	\textbf{7793}	&	3600.0 	&	1	&	40.8 	\\
 0525030-02	&&		9933 		&	3469.1 	&	4083 	&	58.9 	&&	\textbf{7011}	&	3.0 	&	41.8 	&&	\textbf{7011}	&	2047.9 	&	41.8 	&&	\textbf{7011}	&	2047.9 	&	1	&	41.8 	&&	\textbf{7011}	&	2047.9 	&	1	&	41.8 	\\
 0525030-03	&&		10528 		&	3003.9 	&	4956 	&	52.9 	&&		8202 		&	19.0 	&	39.6 	&&	\textbf{8176}	&	3600.0 	&	39.4 	&&	\textbf{8176}	&	3600.0 	&	1	&	39.4 	&&	\textbf{8176}	&	3600.0 	&	1	&	39.4 	\\
 0525030-04	&&		8941 		&	2415.6 	&	4486 	&	49.8 	&&	\textbf{7601}	&	13.5 	&	41.0 	&&		7610 		&	3600.0 	&	41.1 	&&		7610		&	3600.0 	&	1	&	41.1 	&&		7610		&	3600.0 	&	1	&	41.1 	\\
 0525030-05	&&		8061 		&	2872.0 	&	4077 	&	49.4 	&&		6678 		&	14.6 	&	38.9 	&&	\textbf{6668}	&	3600.0 	&	38.9 	&&	\textbf{6668}	&	3600.0 	&	1	&	38.9 	&&	\textbf{6668}	&	3600.0 	&	1	&	38.9 	\\
 0525030-06	&&		9247 		&	3874.1 	&	4283 	&	53.7 	&&	\textbf{7055}	&	7.9 	&	39.3 	&&	\textbf{7055}	&	3600.0 	&	39.3 	&&	\textbf{7055}	&	3600.0 	&	1	&	39.3 	&&	\textbf{7055}	&	3600.0 	&	1	&	39.3 	\\
 0525030-07	&&		9285 		&	2391.7 	&	3319 	&	64.3 	&&		6754 		&	16.7 	&	50.9 	&&	\textbf{6706}	&	3179.0 	&	50.5 	&&	\textbf{6706}	&	3179.0 	&	1	&	50.5 	&&	\textbf{6706}	&	3179.0 	&	1	&	50.5 	\\
 0525030-08	&&		10985 		&	2492.1 	&	5174 	&	52.9 	&&	\textbf{8441}	&	6.3 	&	38.7 	&&	\textbf{8441}	&	3600.0 	&	38.7 	&&	\textbf{8441}	&	3600.0 	&	1	&	38.7 	&&	\textbf{8441}	&	3600.0 	&	1	&	38.7 	\\
 0525030-09	&&		8741 		&	847.9 	&	4568 	&	47.7 	&&	\textbf{7172}	&	2.8 	&	36.3 	&&	\textbf{7172}	&	2562.5 	&	36.3 	&&	\textbf{7172}	&	2562.5 	&	1	&	36.3 	&&	\textbf{7172}	&	2562.5 	&	1	&	36.3 	\\
 0525030-10	&&		8110 		&	2533.9 	&	4172 	&	48.6 	&&		6642 		&	11.2 	&	37.2 	&&	\textbf{6641}	&	3600.0 	&	37.2 	&&	\textbf{6641}	&	3600.0 	&	1	&	37.2 	&&	\textbf{6641}	&	3600.0 	&	1	&	37.2 	\\
 0525030-11	&&		12736 		&	2330.6 	&	4754 	&	62.7 	&&	\textbf{8338}	&	46.4 	&	43.0 	&&		8341 		&	3600.0 	&	43.0 	&&		8341		&	3600.0 	&	1	&	43.0 	&&		8341		&	3600.0 	&	1	&	43.0 	\\
 0525030-12	&&		10717 		&	3391.2 	&	4760 	&	55.6 	&&		7947 		&	6.5 	&	40.1 	&&	\textbf{7885}	&	3600.0 	&	39.6 	&&	\textbf{7885}	&	3600.0 	&	1	&	39.6 	&&	\textbf{7885}	&	3600.0 	&	1	&	39.6 	\\
 0525030-13	&&		8479 		&	4933.2 	&	3926 	&	53.7 	&&		6888 		&	32.7 	&	43.0 	&&	\textbf{6820}	&	3600.0 	&	42.4 	&&	\textbf{6820}	&	3600.0 	&	1	&	42.4 	&&	\textbf{6820}	&	3600.0 	&	1	&	42.4 	\\
 0525030-14	&&		10423 		&	3844.6 	&	4590 	&	56.0 	&&		7616 		&	33.9 	&	39.7 	&&	\textbf{7544}	&	3600.0 	&	39.2 	&&	\textbf{7544}	&	3600.0 	&	1	&	39.2 	&&	\textbf{7544}	&	3600.0 	&	1	&	39.2 	\\
 0525030-15	&&		18185 		&	1740.9 	&	4542 	&	75.0 	&&	\textbf{9067}	&	24.7 	&	49.9 	&&		9093 		&	3600.0 	&	50.0 	&&		9093		&	3600.0 	&	1	&	50.0 	&&		9093		&	3600.0 	&	1	&	50.0 	\\
 0525030-16	&&		12987 		&	838.5 	&	4120 	&	68.3 	&&		7148 		&	12.5 	&	42.4 	&&	\textbf{7114}	&	3600.0 	&	42.1 	&&	\textbf{7114}	&	3600.0 	&	1	&	42.1 	&&	\textbf{7114}	&	3600.0 	&	1	&	42.1 	\\
 0525030-17	&&		9429 		&	2795.1 	&	4584 	&	51.4 	&&	\textbf{7471}	&	7.7 	&	38.6 	&&	\textbf{7471}	&	3600.0 	&	38.6 	&&	\textbf{7471}	&	3600.0 	&	1	&	38.6 	&&	\textbf{7471}	&	3600.0 	&	1	&	38.6 	\\
 0525030-18	&&		9985 		&	2914.9 	&	5068 	&	49.2 	&&	\textbf{8644}	&	28.7 	&	41.4 	&&	\textbf{8644}	&	3600.0 	&	41.4 	&&	\textbf{8644}	&	3600.0 	&	1	&	41.4 	&&	\textbf{8644}	&	3600.0 	&	1	&	41.4 	\\
 0525030-19	&&		10166 		&	2852.4 	&	5293 	&	47.9 	&&		8891 		&	6.6 	&	40.5 	&&	\textbf{8851}	&	3600.0 	&	40.2 	&&	\textbf{8851}	&	3600.0 	&	1	&	40.2 	&&	\textbf{8851}	&	3600.0 	&	1	&	40.2 	\\
 0525030-20	&&		11712 		&	3187.5 	&	5262 	&	55.1 	&&	\textbf{8774}	&	79.3 	&	40.0 	&&		8798 		&	3600.0 	&	40.2 	&&		8798		&	3600.0 	&	1	&	40.2 	&&		8798		&	3600.0 	&	1	&	40.2 	\\
 0525030-21	&&		6545 		&	2734.0 	&	3239 	&	50.5 	&&		5418 		&	8.8 	&	40.2 	&&	\textbf{5369}	&	679.3 	&	39.7 	&&	\textbf{5369}	&	679.3 	&	1	&	39.7 	&&	\textbf{5369}	&	679.3 	&	1	&	39.7 	\\
 0525030-22	&&		6471 		&	2002.1 	&	3179 	&	50.9 	&&	\textbf{5204}	&	4.5 	&	38.9 	&&	\textbf{5204}	&	3600.0 	&	38.9 	&&	\textbf{5204}	&	3600.0 	&	1	&	38.9 	&&	\textbf{5204}	&	3600.0 	&	1	&	38.9 	\\
 0525030-23	&&		8702 		&	2665.0 	&	3833 	&	56.0 	&&		6359 		&	22.7 	&	39.7 	&&	\textbf{6312}	&	3600.0 	&	39.3 	&&	\textbf{6312}	&	3600.0 	&	1	&	39.3 	&&	\textbf{6312}	&	3600.0 	&	1	&	39.3 	\\
 0525030-24	&&		6274 		&	3412.4 	&	3330 	&	46.9 	&&	\textbf{5246}	&	10.9 	&	36.5 	&&	\textbf{5246}	&	3600.0 	&	36.5 	&&	\textbf{5246}	&	3600.0 	&	1	&	36.5 	&&	\textbf{5246}	&	3600.0 	&	1	&	36.5 	\\
 0525030-25	&&		8802 		&	1918.8 	&	3815 	&	56.7 	&&	\textbf{6070}	&	15.8 	&	37.1 	&&		6080 		&	3600.0 	&	37.3 	&&		6080		&	3600.0 	&	1	&	37.3 	&&		6080		&	3600.0 	&	1	&	37.3 	\\
 0525030-26	&&		6501 		&	1458.4 	&	3618 	&	44.3 	&&		5728 		&	12.2 	&	36.8 	&&	\textbf{5726}	&	3600.0 	&	36.8 	&&	\textbf{5726}	&	3600.0 	&	1	&	36.8 	&&	\textbf{5726}	&	3600.0 	&	1	&	36.8 	\\
 0525030-27	&&		6429 		&	1461.1 	&	2993 	&	53.4 	&&		4813 		&	15.9 	&	37.8 	&&	\textbf{4776}	&	3600.0 	&	37.3 	&&	\textbf{4776}	&	3600.0 	&	1	&	37.3 	&&	\textbf{4776}	&	3600.0 	&	1	&	37.3 	\\
 0525030-28	&&		5973 		&	2471.4 	&	3322 	&	44.4 	&&	\textbf{5130}	&	12.6 	&	35.2 	&&		5134 		&	3600.0 	&	35.3 	&&		5134		&	3600.0 	&	1	&	35.3 	&&		5134		&	3600.0 	&	1	&	35.3 	\\
 0525030-29	&&		6226 		&	1665.4 	&	3359 	&	46.0 	&&	\textbf{5273}	&	8.7 	&	36.3 	&&	\textbf{5273}	&	2141.8 	&	36.3 	&&	\textbf{5273}	&	2141.8 	&	1	&	36.3 	&&	\textbf{5273}	&	2141.8 	&	1	&	36.3 	\\
 0525030-30	&&		6611 		&	1034.3 	&	3658 	&	44.7 	&&	\textbf{5747}	&	4.4 	&	36.3 	&&		5768 		&	3600.0 	&	36.6 	&&		5768		&	3600.0 	&	1	&	36.6 	&&		5768		&	3600.0 	&	1	&	36.6 	\\
\bhline{1pt}
\end{tabular}
\end{table}

%% file: tableKP.tex
\begin{table}[htbp]
\centering \fontsize{9pt}{11.2pt}\selectfont
\setlength{\tabcolsep}{3.9pt}
\caption{MMR-KP results for type-1 instances}
\label{tbl:kp-1}
\begin{tabular}{ll*{16}{r}}
\hline
                             && \multicolumn{2}{c}{Best Known}                  && \multicolumn{3}{c}{DS}                  && \multicolumn{4}{c}{iDS-H}                                                                  && \multicolumn{4}{c}{iDS-B}                                                                      \\ \cline{3-4} \cline{6-8} \cline{10-13} \cline{15-18} 
\multicolumn{1}{c}{instance} && \multicolumn{1}{c}{LB} & \multicolumn{1}{c}{UB} &&
\multicolumn{1}{c}{obj} & \multicolumn{1}{c}{time} & \multicolumn{1}{c}{\%gap} &&
\multicolumn{1}{c}{obj} & \multicolumn{1}{c}{time} & \multicolumn{1}{c}{iter} & \multicolumn{1}{c}{\%gap} &&
\multicolumn{1}{c}{obj} & \multicolumn{1}{c}{time} & \multicolumn{1}{c}{iter} & \multicolumn{1}{c}{\%gap} \\ \hline
1-50-01-45-10	&&	15 	&	15 	&&	15 	{\,\,\,\,}	&	0.01 	&	0.0 	&&		15 		{\,\,\,\,}	&	0.01 	&		1 		&	0.0 	&&		15 		{\,\,\,\,}	&	0.01 	&		1 		&	0.0 	\\
1-50-01-50-10	&&	0 	&	0 	&&	0 	{\,\,\,\,}	&	0.01 	&	0.0 	&&		0 		{\,\,\,\,}	&	0.01 	&		1 		&	0.0 	&&		0 		{\,\,\,\,}	&	0.01 	&		1 		&	0.0 	\\
1-50-01-55-10	&&	114 	&	114 	&&	114 	{\,\,\,\,}	&	0.01 	&	0.0 	&&		114 		{\,\,\,\,}	&	0.01 	&		1 		&	0.0 	&&		114 		{\,\,\,\,}	&	0.01 	&		1 		&	0.0 	\\
1-60-01-45-10	&&	91 	&	91 	&&	91 	{\,\,\,\,}	&	0.01 	&	0.0 	&&		91 		{\,\,\,\,}	&	0.01 	&		1 		&	0.0 	&&		91 		{\,\,\,\,}	&	0.01 	&		1 		&	0.0 	\\
1-60-01-50-10	&&	111 	&	111 	&&	111 	{\,\,\,\,}	&	0.01 	&	0.0 	&&		111 		{\,\,\,\,}	&	0.01 	&		1 		&	0.0 	&&		111 		{\,\,\,\,}	&	0.01 	&		1 		&	0.0 	\\
1-60-01-55-10	&&	64 	&	64 	&&	64 	{\,\,\,\,}	&	0.01 	&	0.0 	&&		64 		{\,\,\,\,}	&	0.01 	&		1 		&	0.0 	&&		64 		{\,\,\,\,}	&	0.01 	&		1 		&	0.0 	\\
\hline
\end{tabular}
\end{table}

\begin{table}[htbp]
\centering \fontsize{9pt}{11.2pt}\selectfont
\setlength{\tabcolsep}{3.4pt}
\caption{ MMR-KP results for type-2 instances}
\label{tbl:kp-2}
\begin{tabular}{ll*{16}{r}}
\hline
                             && \multicolumn{2}{c}{Best Known}                  && \multicolumn{3}{c}{DS}                  && \multicolumn{4}{c}{iDS-H}                                                                  && \multicolumn{4}{c}{iDS-B}                                                                      \\ \cline{3-4} \cline{6-8} \cline{10-13} \cline{15-18} 
\multicolumn{1}{c}{instance} && \multicolumn{1}{c}{LB} & \multicolumn{1}{c}{UB} &&
\multicolumn{1}{c}{obj} & \multicolumn{1}{c}{time} & \multicolumn{1}{c}{\%gap} &&
\multicolumn{1}{c}{obj} & \multicolumn{1}{c}{time} & \multicolumn{1}{c}{iter} & \multicolumn{1}{c}{\%gap} &&
\multicolumn{1}{c}{obj} & \multicolumn{1}{c}{time} & \multicolumn{1}{c}{iter} & \multicolumn{1}{c}{\%gap} \\ \hline
2-50-01-45-10	&&	231 	&	231 	&&	231	{\,\,\,\,}	&	0.01 	&	0.0 	&&		231 		{\,\,\,\,}	&	0.01 	&		1 		&	0.0 	&&		231 		{\,\,\,\,}	&	0.01 	&		1 		&	0.0 	\\
2-50-01-50-10	&&	164 	&	164 	&&	164	{\,\,\,\,}	&	0.01 	&	0.0 	&&		164 		{\,\,\,\,}	&	0.01 	&		1 		&	0.0 	&&		164 		{\,\,\,\,}	&	0.01 	&		1 		&	0.0 	\\
2-50-01-55-10	&&	194 	&	194 	&&	194	{\,\,\,\,}	&	0.01 	&	0.0 	&&		194 		{\,\,\,\,}	&	0.01 	&		1 		&	0.0 	&&		194 		{\,\,\,\,}	&	0.01 	&		1 		&	0.0 	\\
2-60-01-45-10	&&	272 	&	272 	&&	278	\,$\uparrow$	&	0.03 	&	2.2 	&&		272 		{\,\,\,\,}	&	0.05 	&		2 		&	0.0 	&&		272 		{\,\,\,\,}	&	0.05 	&		2 		&	0.0 	\\
2-60-01-50-10	&&	240 	&	240 	&&	243	\,$\uparrow$	&	0.02 	&	1.2 	&&		240 		{\,\,\,\,}	&	0.05 	&		2 		&	0.0 	&&		240 		{\,\,\,\,}	&	0.04 	&		2 		&	0.0 	\\
2-60-01-55-10	&&	237 	&	237 	&&	237	{\,\,\,\,}	&	0.01 	&	0.0 	&&		237 		{\,\,\,\,}	&	0.01 	&		1 		&	0.0 	&&		237 		{\,\,\,\,}	&	0.01 	&		1 		&	0.0 	\\
\hline
\end{tabular}
\end{table}

\begin{table}[htbp]
\centering \fontsize{9pt}{11.2pt}\selectfont
\setlength{\tabcolsep}{3.4pt}
\caption{MMR-KP results for type-3 instances}
\label{tbl:kp-3}
\begin{tabular}{ll*{16}{r}}
\hline
                             && \multicolumn{2}{c}{Best Known}                  && \multicolumn{3}{c}{DS}                  && \multicolumn{4}{c}{iDS-H}                                                                  && \multicolumn{4}{c}{iDS-B}                                                                      \\ \cline{3-4} \cline{6-8} \cline{10-13} \cline{15-18} 
\multicolumn{1}{c}{instance} && \multicolumn{1}{c}{LB} & \multicolumn{1}{c}{UB} &&
\multicolumn{1}{c}{obj} & \multicolumn{1}{c}{time} & \multicolumn{1}{c}{\%gap} &&
\multicolumn{1}{c}{obj} & \multicolumn{1}{c}{time} & \multicolumn{1}{c}{iter} & \multicolumn{1}{c}{\%gap} &&
\multicolumn{1}{c}{obj} & \multicolumn{1}{c}{time} & \multicolumn{1}{c}{iter} & \multicolumn{1}{c}{\%gap} \\ \hline
3-50-01-45-10	&&	726 	&	726 	&&	726	{\,\,\,\,}	&	0.07 	&	0.0 	&&		726 		{\,\,\,\,}	&	0.07 	&		1 		&	0.0 	&&		726 		{\,\,\,\,}	&	0.07 	&		1 		&	0.0 	\\
3-50-01-50-10	&&	719 	&	719 	&&	719	{\,\,\,\,}	&	0.06 	&	0.0 	&&		719 		{\,\,\,\,}	&	0.06 	&		1 		&	0.0 	&&		719 		{\,\,\,\,}	&	0.06 	&		1 		&	0.0 	\\
3-50-01-55-10	&&	714 	&	714 	&&	714	{\,\,\,\,}	&	0.05 	&	0.0 	&&		714 		{\,\,\,\,}	&	0.05 	&		1 		&	0.0 	&&		714 		{\,\,\,\,}	&	0.05 	&		1 		&	0.0 	\\
3-60-01-45-10	&&	868 	&	868 	&&	868	{\,\,\,\,}	&	0.07 	&	0.0 	&&		868 		{\,\,\,\,}	&	0.07 	&		1 		&	0.0 	&&		868 		{\,\,\,\,}	&	0.07 	&		1 		&	0.0 	\\
3-60-01-50-10	&&	809 	&	909 	&&	909	{\,\,\,\,}	&	0.24 	&	11.0 	&&		909 		{\,\,\,\,}	&	0.24 	&		1 		&	11.0 	&&		909 		{\,\,\,\,}	&	0.24 	&		1 		&	11.0 	\\
3-60-01-55-10	&&	846 	&	909 	&&	909	{\,\,\,\,}	&	0.24 	&	6.9 	&&		909 		{\,\,\,\,}	&	0.24 	&		1 		&	6.9 	&&		909 		{\,\,\,\,}	&	0.24 	&		1 		&	6.9 	\\
\hline
\end{tabular}
\end{table}

\begin{table}[htbp]
\centering \fontsize{9pt}{11.2pt}\selectfont
\setlength{\tabcolsep}{3.4pt}
\caption{MMR-KP results for type-4 instances}
\label{tbl:kp-4}
\begin{tabular}{ll*{16}{r}}
\hline
                             && \multicolumn{2}{c}{Best Known}                  && \multicolumn{3}{c}{DS}                  && \multicolumn{4}{c}{iDS-H}                                                                  && \multicolumn{4}{c}{iDS-B}                                                                      \\ \cline{3-4} \cline{6-8} \cline{10-13} \cline{15-18} 
\multicolumn{1}{c}{instance} && \multicolumn{1}{c}{LB} & \multicolumn{1}{c}{UB} &&
\multicolumn{1}{c}{obj} & \multicolumn{1}{c}{time} & \multicolumn{1}{c}{\%gap} &&
\multicolumn{1}{c}{obj} & \multicolumn{1}{c}{time} & \multicolumn{1}{c}{iter} & \multicolumn{1}{c}{\%gap} &&
\multicolumn{1}{c}{obj} & \multicolumn{1}{c}{time} & \multicolumn{1}{c}{iter} & \multicolumn{1}{c}{\%gap} \\ \hline
4-50-01-45-10	&&	582 	&	582 	&&	582	{\,\,\,\,}	&	0.09 	&	0.0 	&&		582 		{\,\,\,\,}	&	0.09 	&		1 		&	0.0 	&&		582 		{\,\,\,\,}	&	0.09 	&		1 		&	0.0 	\\
4-50-01-50-10	&&	547 	&	547 	&&	547	{\,\,\,\,}	&	0.06 	&	0.0 	&&		547 		{\,\,\,\,}	&	0.06 	&		1 		&	0.0 	&&		547 		{\,\,\,\,}	&	0.06 	&		1 		&	0.0 	\\
4-50-01-55-10	&&	490 	&	490 	&&	490	{\,\,\,\,}	&	0.07 	&	0.0 	&&		490 		{\,\,\,\,}	&	0.07 	&		1 		&	0.0 	&&		490 		{\,\,\,\,}	&	0.07 	&		1 		&	0.0 	\\
4-60-01-45-10	&&	657 	&	710 	&&	710	{\,\,\,\,}	&	0.10 	&	7.5 	&&		710 		{\,\,\,\,}	&	0.10 	&		1 		&	7.5 	&&		710 		{\,\,\,\,}	&	0.10 	&		1 		&	7.5 	\\
4-60-01-50-10	&&	592 	&	668 	&&	668	{\,\,\,\,}	&	0.12 	&	11.4 	&&		668 		{\,\,\,\,}	&	0.12 	&		1 		&	11.4 	&&		668 		{\,\,\,\,}	&	0.12 	&		1 		&	11.4 	\\
4-60-01-55-10	&&	603 	&	603 	&&	603	{\,\,\,\,}	&	0.08 	&	0.0 	&&		603 		{\,\,\,\,}	&	0.08 	&		1 		&	0.0 	&&		603 		{\,\,\,\,}	&	0.08 	&		1 		&	0.0 	\\
\hline
\end{tabular}
\end{table}

\begin{table}[htbp]
\centering \fontsize{9pt}{11.2pt}\selectfont
\setlength{\tabcolsep}{3.4pt}
\caption{MMR-KP results for type-5 instances}
\label{tbl:kp-5}
\begin{tabular}{ll*{16}{r}}
\hline
                             && \multicolumn{2}{c}{Best Known}                  && \multicolumn{3}{c}{DS}                  && \multicolumn{4}{c}{iDS-H}                                                                  && \multicolumn{4}{c}{iDS-B}                                                                      \\ \cline{3-4} \cline{6-8} \cline{10-13} \cline{15-18} 
\multicolumn{1}{c}{instance} && \multicolumn{1}{c}{LB} & \multicolumn{1}{c}{UB} &&
\multicolumn{1}{c}{obj} & \multicolumn{1}{c}{time} & \multicolumn{1}{c}{\%gap} &&
\multicolumn{1}{c}{obj} & \multicolumn{1}{c}{time} & \multicolumn{1}{c}{iter} & \multicolumn{1}{c}{\%gap} &&
\multicolumn{1}{c}{obj} & \multicolumn{1}{c}{time} & \multicolumn{1}{c}{iter} & \multicolumn{1}{c}{\%gap} \\ \hline
5-50-01-45-10	&&	731 	&	731 	&&	736	\,$\uparrow$	&	0.11 	&	0.7 	&&		731 		{\,\,\,\,}	&	0.49 	&		3 		&	0.0 	&&		731 		{\,\,\,\,}	&	0.43 	&		3 		&	0.0 	\\
5-50-01-50-10	&&	770 	&	770 	&&	770	{\,\,\,\,}	&	0.29 	&	0.0 	&&		770 		{\,\,\,\,}	&	0.29 	&		1 		&	0.0 	&&		770 		{\,\,\,\,}	&	0.29 	&		1 		&	0.0 	\\
5-50-01-55-10	&&	777 	&	777 	&&	777	{\,\,\,\,}	&	0.15 	&	0.0 	&&		777 		{\,\,\,\,}	&	0.15 	&		1 		&	0.0 	&&		777 		{\,\,\,\,}	&	0.15 	&		1 		&	0.0 	\\
5-60-01-45-10	&&	786 	&	874 	&&	874	{\,\,\,\,}	&	0.15 	&	10.1 	&&		874 		{\,\,\,\,}	&	0.15 	&		1 		&	10.1 	&&		874 		{\,\,\,\,}	&	0.15 	&		1 		&	10.1 	\\
5-60-01-50-10	&&	827 	&	924 	&&	932	\,$\uparrow$	&	0.59 	&	11.3 	&&		924 		{\,\,\,\,}	&	1.58 	&		3 		&	10.5 	&&		924 		{\,\,\,\,}	&	1.61 	&		3 		&	10.5 	\\
5-60-01-55-10	&&	828 	&	931 	&&	931	{\,\,\,\,}	&	0.25 	&	11.1 	&&		931 		{\,\,\,\,}	&	0.25 	&		1 		&	11.1 	&&		931 		{\,\,\,\,}	&	0.25 	&		1 		&	11.1 	\\
\hline
\end{tabular}
\end{table}

\begin{table}[htbp]
\centering \fontsize{9pt}{11.2pt}\selectfont
\setlength{\tabcolsep}{3.1pt}
\caption{MMR-KP results for type-6 instances}
\label{tbl:kp-6}
\begin{tabular}{ll*{16}{r}}
\hline
                             && \multicolumn{2}{c}{Best Known}                  && \multicolumn{3}{c}{DS}                  && \multicolumn{4}{c}{iDS-H}                                                                  && \multicolumn{4}{c}{iDS-B}                                                                      \\ \cline{3-4} \cline{6-8} \cline{10-13} \cline{15-18} 
\multicolumn{1}{c}{instance} && \multicolumn{1}{c}{LB} & \multicolumn{1}{c}{UB} &&
\multicolumn{1}{c}{obj} & \multicolumn{1}{c}{time} & \multicolumn{1}{c}{\%gap} &&
\multicolumn{1}{c}{obj} & \multicolumn{1}{c}{time} & \multicolumn{1}{c}{iter} & \multicolumn{1}{c}{\%gap} &&
\multicolumn{1}{c}{obj} & \multicolumn{1}{c}{time} & \multicolumn{1}{c}{iter} & \multicolumn{1}{c}{\%gap} \\ \hline
6-50-01-45-10	&&	646 	&	760 	&&	765	\,$\uparrow$	&	0.09 	&	15.6 	&&		760 		{\,\,\,\,}	&	0.16 	&		2 		&	15.0 	&&		760 		{\,\,\,\,}	&	0.17 	&		2 		&	15.0 	\\
6-50-01-50-10	&&	644 	&	740 	&&	740	{\,\,\,\,}	&	0.07 	&	13.0 	&&		740 		{\,\,\,\,}	&	0.07 	&		1 		&	13.0 	&&		740 		{\,\,\,\,}	&	0.07 	&		1 		&	13.0 	\\
6-50-01-55-10	&&	644 	&	743 	&&	743	{\,\,\,\,}	&	0.07 	&	13.3 	&&		743 		{\,\,\,\,}	&	0.07 	&		1 		&	13.3 	&&		743 		{\,\,\,\,}	&	0.07 	&		1 		&	13.3 	\\
6-60-01-45-10	&&	775 	&	991 	&&	991	{\,\,\,\,}	&	0.14 	&	21.8 	&&		991 		{\,\,\,\,}	&	0.14 	&		1 		&	21.8 	&&		991 		{\,\,\,\,}	&	0.14 	&		1 		&	21.8 	\\
6-60-01-50-10	&&	775 	&	998 	&&	999	\,$\uparrow$	&	0.13 	&	22.4 	&&		998 		{\,\,\,\,}	&	0.53 	&		3 		&	22.3 	&&		998 		{\,\,\,\,}	&	0.48 	&		3 		&	22.3 	\\
6-60-01-55-10	&&	780 	&	1005 	&&	1005	{\,\,\,\,}	&	0.16 	&	22.4 	&&		1005 		{\,\,\,\,}	&	0.16 	&		1 		&	22.4 	&&		1005 		{\,\,\,\,}	&	0.16 	&		1 		&	22.4 	\\
\hline
\end{tabular}
\end{table}

\begin{table}[htbp]
\centering \fontsize{9pt}{11.2pt}\selectfont
\setlength{\tabcolsep}{3.0pt}
\caption{MMR-KP results for type-7 instances}
\label{tbl:kp-7}
\begin{tabular}{ll*{16}{r}}
\hline
                             && \multicolumn{2}{c}{Best Known}                  && \multicolumn{3}{c}{DS}                  && \multicolumn{4}{c}{iDS-H}                                                                  && \multicolumn{4}{c}{iDS-B}                                                                      \\ \cline{3-4} \cline{6-8} \cline{10-13} \cline{15-18} 
\multicolumn{1}{c}{instance} && \multicolumn{1}{c}{LB} & \multicolumn{1}{c}{UB} &&
\multicolumn{1}{c}{obj} & \multicolumn{1}{c}{time} & \multicolumn{1}{c}{\%gap} &&
\multicolumn{1}{c}{obj} & \multicolumn{1}{c}{time} & \multicolumn{1}{c}{iter} & \multicolumn{1}{c}{\%gap} &&
\multicolumn{1}{c}{obj} & \multicolumn{1}{c}{time} & \multicolumn{1}{c}{iter} & \multicolumn{1}{c}{\%gap} \\ \hline
7-50-01-45-10	&&	689 	&	833 	&&	833	{\,\,\,\,}	&	0.21 	&	17.3 	&&		833 		{\,\,\,\,}	&	0.21 	&		1 		&	17.3 	&&		833 		{\,\,\,\,}	&	0.21 	&		1 		&	17.3 	\\
7-50-01-50-10	&&	688 	&	810 	&&	810	{\,\,\,\,}	&	0.07 	&	15.1 	&&		810 		{\,\,\,\,}	&	0.07 	&		1 		&	15.1 	&&		810 		{\,\,\,\,}	&	0.07 	&		1 		&	15.1 	\\
7-50-01-55-10	&&	691 	&	803 	&&	803	{\,\,\,\,}	&	0.08 	&	13.9 	&&		803 		{\,\,\,\,}	&	0.08 	&		1 		&	13.9 	&&		803 		{\,\,\,\,}	&	0.08 	&		1 		&	13.9 	\\
7-60-01-45-10	&&	820 	&	1061 	&&	1061	{\,\,\,\,}	&	0.38 	&	22.7 	&&		1061 		{\,\,\,\,}	&	0.38 	&		1 		&	22.7 	&&		1061 		{\,\,\,\,}	&	0.38 	&		1 		&	22.7 	\\
7-60-01-50-10	&&	820 	&	1073 	&&	1069	\,$\downarrow$	&	0.29 	&	23.3 	&&		1069 		\,$\downarrow$	&	0.29 	&		1 		&	23.3 	&&		1069 		\,$\downarrow$	&	0.29 	&		1 		&	23.3 	\\
7-60-01-55-10	&&	816 	&	1072 	&&	1075	\,$\uparrow$	&	0.32 	&	24.1 	&&		1075 		\,$\uparrow$	&	0.32 	&		1 		&	24.1 	&&		1075 		\,$\uparrow$	&	0.32 	&		1 		&	24.1 	\\
\hline
\end{tabular}
\end{table}

\begin{table}[htbp]
\centering \fontsize{9pt}{11.2pt}\selectfont
\setlength{\tabcolsep}{3.4pt}
\caption{MMR-KP results for type-8 instances}
\label{tbl:kp-8}
\begin{tabular}{ll*{16}{r}}
\hline
                             && \multicolumn{2}{c}{Best Known}                  && \multicolumn{3}{c}{DS}                  && \multicolumn{4}{c}{iDS-H}                                                                  && \multicolumn{4}{c}{iDS-B}                                                                      \\ \cline{3-4} \cline{6-8} \cline{10-13} \cline{15-18} 
\multicolumn{1}{c}{instance} && \multicolumn{1}{c}{LB} & \multicolumn{1}{c}{UB} &&
\multicolumn{1}{c}{obj} & \multicolumn{1}{c}{time} & \multicolumn{1}{c}{\%gap} &&
\multicolumn{1}{c}{obj} & \multicolumn{1}{c}{time} & \multicolumn{1}{c}{iter} & \multicolumn{1}{c}{\%gap} &&
\multicolumn{1}{c}{obj} & \multicolumn{1}{c}{time} & \multicolumn{1}{c}{iter} & \multicolumn{1}{c}{\%gap} \\ \hline
8-50-01-45-10	&&	620 	&	620 	&&	620	{\,\,\,\,}	&	0.05 	&	0.0 	&&		620 		{\,\,\,\,}	&	0.05 	&		1 		&	0.0 	&&		620 		{\,\,\,\,}	&	0.05 	&		1 		&	0.0 	\\
8-50-01-50-10	&&	644 	&	644 	&&	646	\,$\uparrow$	&	0.07 	&	0.3 	&&		644 		{\,\,\,\,}	&	0.22 	&		3 		&	0.0 	&&		644 		{\,\,\,\,}	&	0.20 	&		3 		&	0.0 	\\
8-50-01-55-10	&&	632 	&	632 	&&	634	\,$\uparrow$	&	0.03 	&	0.3 	&&		632 		{\,\,\,\,}	&	0.05 	&		2 		&	0.0 	&&		632 		{\,\,\,\,}	&	0.05 	&		2 		&	0.0 	\\
8-60-01-45-10	&&	782 	&	782 	&&	782	{\,\,\,\,}	&	0.05 	&	0.0 	&&		782 		{\,\,\,\,}	&	0.05 	&		1 		&	0.0 	&&		782 		{\,\,\,\,}	&	0.05 	&		1 		&	0.0 	\\
8-60-01-50-10	&&	815 	&	815 	&&	826	\,$\uparrow$	&	0.09 	&	1.3 	&&		826 		\,$\uparrow$	&	0.09 	&		1 		&	1.3 	&&		826 		\,$\uparrow$	&	0.09 	&		1 		&	1.3 	\\
8-60-01-55-10	&&	783 	&	847 	&&	848	\,$\uparrow$	&	0.24 	&	7.7 	&&		848 		\,$\uparrow$	&	0.24 	&		1 		&	7.7 	&&		848 		\,$\uparrow$	&	0.24 	&		1 		&	7.7 	\\
\hline
\end{tabular}
\end{table}

\begin{table}[htbp]
\centering \fontsize{9pt}{11.2pt}\selectfont
\setlength{\tabcolsep}{3.9pt}
\caption{MMR-KP results for type-9 instances}
\label{tbl:kp-9}
\begin{tabular}{ll*{16}{r}}
\hline
                             && \multicolumn{2}{c}{Best Known}                  && \multicolumn{3}{c}{DS}                  && \multicolumn{4}{c}{iDS-H}                                                                  && \multicolumn{4}{c}{iDS-B}                                                                      \\ \cline{3-4} \cline{6-8} \cline{10-13} \cline{15-18} 
\multicolumn{1}{c}{instance} && \multicolumn{1}{c}{LB} & \multicolumn{1}{c}{UB} &&
\multicolumn{1}{c}{obj} & \multicolumn{1}{c}{time} & \multicolumn{1}{c}{\%gap} &&
\multicolumn{1}{c}{obj} & \multicolumn{1}{c}{time} & \multicolumn{1}{c}{iter} & \multicolumn{1}{c}{\%gap} &&
\multicolumn{1}{c}{obj} & \multicolumn{1}{c}{time} & \multicolumn{1}{c}{iter} & \multicolumn{1}{c}{\%gap} \\ \hline
9-50-01-45-10	&&	46 	&	46 	&&	46	{\,\,\,\,}	&	0.01 	&	0.0 	&&		46 		{\,\,\,\,}	&	0.01 	&		1 		&	0.0 	&&		46 		{\,\,\,\,}	&	0.01 	&		1 		&	0.0 	\\
9-50-01-50-10	&&	31 	&	31 	&&	31	{\,\,\,\,}	&	0.01 	&	0.0 	&&		31 		{\,\,\,\,}	&	0.01 	&		1 		&	0.0 	&&		31 		{\,\,\,\,}	&	0.01 	&		1 		&	0.0 	\\
9-50-01-55-10	&&	51 	&	51 	&&	51	{\,\,\,\,}	&	0.01 	&	0.0 	&&		51 		{\,\,\,\,}	&	0.01 	&		1 		&	0.0 	&&		51 		{\,\,\,\,}	&	0.01 	&		1 		&	0.0 	\\
9-60-01-45-10	&&	46 	&	46 	&&	46	{\,\,\,\,}	&	0.01 	&	0.0 	&&		46 		{\,\,\,\,}	&	0.01 	&		1 		&	0.0 	&&		46 		{\,\,\,\,}	&	0.01 	&		1 		&	0.0 	\\
9-60-01-50-10	&&	40 	&	40 	&&	40	{\,\,\,\,}	&	0.01 	&	0.0 	&&		40 		{\,\,\,\,}	&	0.01 	&		1 		&	0.0 	&&		40 		{\,\,\,\,}	&	0.01 	&		1 		&	0.0 	\\
9-60-01-55-10	&&	71 	&	71 	&&	71	{\,\,\,\,}	&	0.01 	&	0.0 	&&		71 		{\,\,\,\,}	&	0.01 	&		1 		&	0.0 	&&		71 		{\,\,\,\,}	&	0.01 	&		1 		&	0.0 	\\
\hline
\end{tabular}
\end{table}

%% file: tableSCP.tex
\begin{table}[htbp]
\centering \fontsize{7.5pt}{8.2pt}\selectfont
\setlength{\tabcolsep}{6.1pt}
\caption{MMR-SCP results for type-B instances}
\label{tbl:scp-b}
\begin{tabular}{ll*{16}{r}}
\hline
                             && \multicolumn{2}{c}{Best Known}                  && \multicolumn{3}{c}{DS}                  && \multicolumn{4}{c}{iDS-H}                                                                  && \multicolumn{4}{c}{iDS-B}                                                                      \\ \cline{3-4} \cline{6-8} \cline{10-13} \cline{15-18} 
\multicolumn{1}{c}{instance} && \multicolumn{1}{c}{LB} & \multicolumn{1}{c}{UB} &&
\multicolumn{1}{c}{obj} & \multicolumn{1}{c}{time} & \multicolumn{1}{c}{\%gap} &&
\multicolumn{1}{c}{obj} & \multicolumn{1}{c}{time} & \multicolumn{1}{c}{iter} & \multicolumn{1}{c}{\%gap} &&
\multicolumn{1}{c}{obj} & \multicolumn{1}{c}{time} & \multicolumn{1}{c}{iter} & \multicolumn{1}{c}{\%gap} \\ \hline
 B40110	&&	21	&	21	&&	21	{\,\,\,\,}	&	0.493 	&	0.0 	&&		21		{\,\,\,\,}	&	0.05 	&		1		&	0.0 	&&		21		{\,\,\,\,}	&	0.05 	&		1		&	0.0 	\\
 B40130	&&	67	&	67	&&	67	{\,\,\,\,}	&	0.516 	&	0.0 	&&		67		{\,\,\,\,}	&	0.26 	&		1		&	0.0 	&&		67		{\,\,\,\,}	&	0.26 	&		1		&	0.0 	\\
 B40150	&&	124	&	124	&&	126	\,$\uparrow$	&	0.124 	&	1.6 	&&		124		{\,\,\,\,}	&	3.78 	&		8		&	0.0 	&&		124		{\,\,\,\,}	&	1.60 	&	\textbf{4}	&	0.0 	\\
 B40210	&&	23	&	23	&&	23	{\,\,\,\,}	&	0.273 	&	0.0 	&&		23		{\,\,\,\,}	&	0.12 	&		1		&	0.0 	&&		23		{\,\,\,\,}	&	0.12 	&		1		&	0.0 	\\
 B40230	&&	93	&	93	&&	94	\,$\uparrow$	&	0.422 	&	1.1 	&&		93		{\,\,\,\,}	&	3.88 	&		9		&	0.0 	&&		93		{\,\,\,\,}	&	2.58 	&	\textbf{5}	&	0.0 	\\
 B40250	&&	163	&	163	&&	164	\,$\uparrow$	&	0.085 	&	0.6 	&&		163		{\,\,\,\,}	&	1.71 	&		3		&	0.0 	&&		163		{\,\,\,\,}	&	2.11 	&		3		&	0.0 	\\
 B40310	&&	12	&	12	&&	12	{\,\,\,\,}	&	0.268 	&	0.0 	&&		12		{\,\,\,\,}	&	0.03 	&		1		&	0.0 	&&		12		{\,\,\,\,}	&	0.03 	&		1		&	0.0 	\\
 B40330	&&	86	&	86	&&	86	{\,\,\,\,}	&	0.538 	&	0.0 	&&		86		{\,\,\,\,}	&	0.31 	&		1		&	0.0 	&&		86		{\,\,\,\,}	&	0.31 	&		1		&	0.0 	\\
 B40350	&&	190	&	190	&&	191	\,$\uparrow$	&	0.212 	&	0.5 	&&		190		{\,\,\,\,}	&	5.22 	&		7		&	0.0 	&&		190		{\,\,\,\,}	&	5.69 	&		7		&	0.0 	\\
 B40410	&&	20	&	20	&&	20	{\,\,\,\,}	&	0.456 	&	0.0 	&&		20		{\,\,\,\,}	&	0.07 	&		1		&	0.0 	&&		20		{\,\,\,\,}	&	0.07 	&		1		&	0.0 	\\
 B40430	&&	85	&	85	&&	85	{\,\,\,\,}	&	0.557 	&	0.0 	&&		85		{\,\,\,\,}	&	0.55 	&		1		&	0.0 	&&		85		{\,\,\,\,}	&	0.55 	&		1		&	0.0 	\\
 B40450	&&	160	&	160	&&	160	{\,\,\,\,}	&	0.334 	&	0.0 	&&		160		{\,\,\,\,}	&	0.94 	&		1		&	0.0 	&&		160		{\,\,\,\,}	&	0.94 	&		1		&	0.0 	\\
 B40510	&&	13	&	13	&&	13	{\,\,\,\,}	&	0.349 	&	0.0 	&&		13		{\,\,\,\,}	&	0.04 	&		1		&	0.0 	&&		13		{\,\,\,\,}	&	0.04 	&		1		&	0.0 	\\
 B40530	&&	73	&	73	&&	74	\,$\uparrow$	&	0.868 	&	1.4 	&&		73		{\,\,\,\,}	&	5.66 	&		15		&	0.0 	&&		73		{\,\,\,\,}	&	2.65 	&	\textbf{7}	&	0.0 	\\
 B40550	&&	165	&	165	&&	165	{\,\,\,\,}	&	0.110 	&	0.0 	&&		165		{\,\,\,\,}	&	0.56 	&		1		&	0.0 	&&		165		{\,\,\,\,}	&	0.56 	&		1		&	0.0 	\\
 B40610	&&	16	&	16	&&	17	\,$\uparrow$	&	0.444 	&	5.9 	&&		16		{\,\,\,\,}	&	0.27 	&		3		&	0.0 	&&		16		{\,\,\,\,}	&	0.36 	&		3		&	0.0 	\\
 B40630	&&	84	&	84	&&	85	\,$\uparrow$	&	0.615 	&	1.2 	&&		84		{\,\,\,\,}	&	3.32 	&		8		&	0.0 	&&		84		{\,\,\,\,}	&	3.97 	&	\textbf{7}	&	0.0 	\\
 B40650	&&	175	&	175	&&	181	\,$\uparrow$	&	0.045 	&	3.3 	&&		175		{\,\,\,\,}	&	199.30 	&		165		&	0.0 	&&		175		{\,\,\,\,}	&	113.01 	&	\textbf{73}	&	0.0 	\\
 B40710	&&	17	&	17	&&	17	{\,\,\,\,}	&	0.181 	&	0.0 	&&		17		{\,\,\,\,}	&	0.09 	&		1		&	0.0 	&&		17		{\,\,\,\,}	&	0.09 	&		1		&	0.0 	\\
 B40730	&&	59	&	59	&&	60	\,$\uparrow$	&	0.519 	&	1.7 	&&		59		{\,\,\,\,}	&	0.82 	&		3		&	0.0 	&&		59		{\,\,\,\,}	&	0.59 	&	\textbf{2}	&	0.0 	\\
 B40750	&&	104	&	104	&&	106	\,$\uparrow$	&	0.415 	&	1.9 	&&		104		{\,\,\,\,}	&	1.13 	&		5		&	0.0 	&&		104		{\,\,\,\,}	&	0.40 	&	\textbf{2}	&	0.0 	\\
 B40810	&&	24	&	24	&&	24	{\,\,\,\,}	&	0.835 	&	0.0 	&&		24		{\,\,\,\,}	&	0.21 	&		1		&	0.0 	&&		24		{\,\,\,\,}	&	0.21 	&		1		&	0.0 	\\
 B40830	&&	54	&	54	&&	56	\,$\uparrow$	&	0.864 	&	3.6 	&&		54		{\,\,\,\,}	&	0.97 	&		4		&	0.0 	&&		54		{\,\,\,\,}	&	1.03 	&	\textbf{3}	&	0.0 	\\
 B40850	&&	155	&	155	&&	155	{\,\,\,\,}	&	0.493 	&	0.0 	&&		155		{\,\,\,\,}	&	0.79 	&		1		&	0.0 	&&		155		{\,\,\,\,}	&	0.79 	&		1		&	0.0 	\\
 B40910	&&	33	&	33	&&	33	{\,\,\,\,}	&	0.546 	&	0.0 	&&		33		{\,\,\,\,}	&	0.36 	&		1		&	0.0 	&&		33		{\,\,\,\,}	&	0.36 	&		1		&	0.0 	\\
 B40930	&&	99	&	99	&&	100	\,$\uparrow$	&	0.894 	&	1.0 	&&		99		{\,\,\,\,}	&	12.93 	&		16		&	0.0 	&&		99		{\,\,\,\,}	&	2.09 	&	\textbf{4}	&	0.0 	\\
 B40950	&&	263	&	263	&&	263	{\,\,\,\,}	&	0.081 	&	0.0 	&&		263		{\,\,\,\,}	&	1.14 	&		1		&	0.0 	&&		263		{\,\,\,\,}	&	1.14 	&		1		&	0.0 	\\
 B41010	&&	17	&	17	&&	17	{\,\,\,\,}	&	0.568 	&	0.0 	&&		17		{\,\,\,\,}	&	0.03 	&		1		&	0.0 	&&		17		{\,\,\,\,}	&	0.03 	&		1		&	0.0 	\\
 B41030	&&	63	&	63	&&	63	{\,\,\,\,}	&	0.592 	&	0.0 	&&		63		{\,\,\,\,}	&	0.21 	&		1		&	0.0 	&&		63		{\,\,\,\,}	&	0.21 	&		1		&	0.0 	\\
 B41050	&&	141	&	141	&&	143	\,$\uparrow$	&	0.157 	&	1.4 	&&		141		{\,\,\,\,}	&	2.01 	&		8		&	0.0 	&&		141		{\,\,\,\,}	&	2.09 	&		8		&	0.0 	\\
 B50110	&&	24	&	24	&&	24	{\,\,\,\,}	&	0.262 	&	0.0 	&&		24		{\,\,\,\,}	&	0.55 	&		1		&	0.0 	&&		24		{\,\,\,\,}	&	0.55 	&		1		&	0.0 	\\
 B50130	&&	53	&	53	&&	54	\,$\uparrow$	&	0.567 	&	1.9 	&&		53		{\,\,\,\,}	&	16.80 	&		20		&	0.0 	&&		53		{\,\,\,\,}	&	17.42 	&	\textbf{16}	&	0.0 	\\
 B50150	&&	93	&	93	&&	95	\,$\uparrow$	&	0.496 	&	2.1 	&&		93		{\,\,\,\,}	&	4.77 	&		5		&	0.0 	&&		93		{\,\,\,\,}	&	4.66 	&		5		&	0.0 	\\
 B50210	&&	19	&	19	&&	19	{\,\,\,\,}	&	0.902 	&	0.0 	&&		19		{\,\,\,\,}	&	0.48 	&		1		&	0.0 	&&		19		{\,\,\,\,}	&	0.48 	&		1		&	0.0 	\\
 B50230	&&	69	&	69	&&	72	\,$\uparrow$	&	1.219 	&	4.2 	&&		69		{\,\,\,\,}	&	2.24 	&		2		&	0.0 	&&		69		{\,\,\,\,}	&	2.33 	&		2		&	0.0 	\\
 B50250	&&	121	&	121	&&	122	\,$\uparrow$	&	0.12 	&	0.8 	&&		121		{\,\,\,\,}	&	2.54 	&		2		&	0.0 	&&		121		{\,\,\,\,}	&	2.48 	&		2		&	0.0 	\\
 B50310	&&	12	&	12	&&	13	\,$\uparrow$	&	0.38 	&	7.7 	&&		12		{\,\,\,\,}	&	0.75 	&		7		&	0.0 	&&		12		{\,\,\,\,}	&	0.41 	&	\textbf{3}	&	0.0 	\\
 B50330	&&	24	&	24	&&	24	{\,\,\,\,}	&	0.42 	&	0.0 	&&		24		{\,\,\,\,}	&	0.10 	&		1		&	0.0 	&&		24		{\,\,\,\,}	&	0.10 	&		1		&	0.0 	\\
 B50350	&&	66	&	66	&&	67	\,$\uparrow$	&	0.07 	&	1.5 	&&		66		{\,\,\,\,}	&	1.43 	&		5		&	0.0 	&&		66		{\,\,\,\,}	&	0.54 	&	\textbf{2}	&	0.0 	\\
 B50410	&&	18	&	18	&&	19	\,$\uparrow$	&	0.67 	&	5.3 	&&		18		{\,\,\,\,}	&	1.12 	&		4		&	0.0 	&&		18		{\,\,\,\,}	&	1.47 	&		4		&	0.0 	\\
 B50430	&&	50	&	50	&&	51	\,$\uparrow$	&	0.10 	&	2.0 	&&		50		{\,\,\,\,}	&	1.49 	&		3		&	0.0 	&&		50		{\,\,\,\,}	&	1.23 	&	\textbf{2}	&	0.0 	\\
 B50450	&&	84	&	84	&&	86	\,$\uparrow$	&	0.17 	&	2.3 	&&		84		{\,\,\,\,}	&	1.15 	&		2		&	0.0 	&&		84		{\,\,\,\,}	&	1.21 	&		2		&	0.0 	\\
 B50510	&&	10	&	10	&&	11	\,$\uparrow$	&	0.57 	&	9.1 	&&		10		{\,\,\,\,}	&	0.25 	&		2		&	0.0 	&&		10		{\,\,\,\,}	&	0.29 	&		2		&	0.0 	\\
 B50530	&&	34	&	34	&&	37	\,$\uparrow$	&	0.48 	&	8.1 	&&		34		{\,\,\,\,}	&	1.22 	&		4		&	0.0 	&&		34		{\,\,\,\,}	&	1.51 	&		4		&	0.0 	\\
 B50550	&&	57	&	57	&&	57	{\,\,\,\,}	&	0.86 	&	0.0 	&&		57		{\,\,\,\,}	&	0.42 	&		1		&	0.0 	&&		57		{\,\,\,\,}	&	0.42 	&		1		&	0.0 	\\
 B50610	&&	13	&	13	&&	13	{\,\,\,\,}	&	0.11 	&	0.0 	&&		13		{\,\,\,\,}	&	0.09 	&		1		&	0.0 	&&		13		{\,\,\,\,}	&	0.09 	&		1		&	0.0 	\\
 B50630	&&	29	&	29	&&	29	{\,\,\,\,}	&	0.33 	&	0.0 	&&		29		{\,\,\,\,}	&	0.27 	&		1		&	0.0 	&&		29		{\,\,\,\,}	&	0.27 	&		1		&	0.0 	\\
 B50650	&&	51	&	51	&&	52	\,$\uparrow$	&	0.44 	&	1.9 	&&		51		{\,\,\,\,}	&	8.77 	&		17		&	0.0 	&&		51		{\,\,\,\,}	&	2.08 	&	\textbf{4}	&	0.0 	\\
 B50710	&&	17	&	17	&&	17	{\,\,\,\,}	&	0.16 	&	0.0 	&&		17		{\,\,\,\,}	&	0.21 	&		1		&	0.0 	&&		17		{\,\,\,\,}	&	0.21 	&		1		&	0.0 	\\
 B50730	&&	54	&	54	&&	54	{\,\,\,\,}	&	0.19 	&	0.0 	&&		54		{\,\,\,\,}	&	0.46 	&		1		&	0.0 	&&		54		{\,\,\,\,}	&	0.46 	&		1		&	0.0 	\\
 B50750	&&	88	&	88	&&	88	{\,\,\,\,}	&	1.04 	&	0.0 	&&		88		{\,\,\,\,}	&	0.56 	&		1		&	0.0 	&&		88		{\,\,\,\,}	&	0.56 	&		1		&	0.0 	\\
 B50810	&&	24	&	24	&&	25	\,$\uparrow$	&	0.32 	&	4.0 	&&		24		{\,\,\,\,}	&	2.24 	&		5		&	0.0 	&&		24		{\,\,\,\,}	&	1.13 	&	\textbf{3}	&	0.0 	\\
 B50830	&&	57	&	57	&&	57	{\,\,\,\,}	&	0.47 	&	0.0 	&&		57		{\,\,\,\,}	&	0.35 	&		1		&	0.0 	&&		57		{\,\,\,\,}	&	0.35 	&		1		&	0.0 	\\
 B50850	&&	105	&	105	&&	105	{\,\,\,\,}	&	0.31 	&	0.0 	&&		105		{\,\,\,\,}	&	0.87 	&		1		&	0.0 	&&		105		{\,\,\,\,}	&	0.87 	&		1		&	0.0 	\\
 B50910	&&	21	&	21	&&	21	{\,\,\,\,}	&	0.29 	&	0.0 	&&		21		{\,\,\,\,}	&	0.11 	&		1		&	0.0 	&&		21		{\,\,\,\,}	&	0.11 	&		1		&	0.0 	\\
 B50930	&&	56	&	56	&&	57	\,$\uparrow$	&	0.53 	&	1.8 	&&		56		{\,\,\,\,}	&	1.38 	&		3		&	0.0 	&&		56		{\,\,\,\,}	&	1.66 	&		3		&	0.0 	\\
 B50950	&&	91	&	91	&&	91	{\,\,\,\,}	&	0.19 	&	0.0 	&&		91		{\,\,\,\,}	&	0.62 	&		1		&	0.0 	&&		91		{\,\,\,\,}	&	0.62 	&		1		&	0.0 	\\
 B51010	&&	10	&	10	&&	10	{\,\,\,\,}	&	0.13 	&	0.0 	&&		10		{\,\,\,\,}	&	0.05 	&		1		&	0.0 	&&		10		{\,\,\,\,}	&	0.05 	&		1		&	0.0 	\\
 B51030	&&	37	&	37	&&	37	{\,\,\,\,}	&	0.49 	&	0.0 	&&		37		{\,\,\,\,}	&	0.18 	&		1		&	0.0 	&&		37		{\,\,\,\,}	&	0.18 	&		1		&	0.0 	\\
 B51050	&&	76	&	76	&&	76	{\,\,\,\,}	&	0.25 	&	0.0 	&&		76		{\,\,\,\,}	&	0.52 	&		1		&	0.0 	&&		76		{\,\,\,\,}	&	0.52 	&		1		&	0.0 	\\
 B60110	&&	17	&	17	&&	17	{\,\,\,\,}	&	0.55 	&	0.0 	&&		17		{\,\,\,\,}	&	0.42 	&		1		&	0.0 	&&		17		{\,\,\,\,}	&	0.42 	&		1		&	0.0 	\\
 B60130	&&	38	&	38	&&	39	\,$\uparrow$	&	0.82 	&	2.6 	&&		38		{\,\,\,\,}	&	71.73 	&		80		&	0.0 	&&		38		{\,\,\,\,}	&	17.99 	&	\textbf{18}	&	0.0 	\\
 B60150	&&	58	&	58	&&	58	{\,\,\,\,}	&	0.90 	&	0.0 	&&		58		{\,\,\,\,}	&	0.86 	&		1		&	0.0 	&&		58		{\,\,\,\,}	&	0.86 	&		1		&	0.0 	\\
 B60210	&&	15	&	15	&&	15	{\,\,\,\,}	&	0.74 	&	0.0 	&&		15		{\,\,\,\,}	&	0.49 	&		1		&	0.0 	&&		15		{\,\,\,\,}	&	0.49 	&		1		&	0.0 	\\
 B60230	&&	40	&	40	&&	41	\,$\uparrow$	&	1.01 	&	2.4 	&&		40		{\,\,\,\,}	&	1.69 	&		3		&	0.0 	&&		40		{\,\,\,\,}	&	1.42 	&	\textbf{2}	&	0.0 	\\
 B60250	&&	56	&	56	&&	59	\,$\uparrow$	&	0.89 	&	5.1 	&&		56		{\,\,\,\,}	&	1.85 	&		2		&	0.0 	&&		56		{\,\,\,\,}	&	1.80 	&		2		&	0.0 	\\
 B60310	&&	6	&	6	&&	6	{\,\,\,\,}	&	0.29 	&	0.0 	&&		6		{\,\,\,\,}	&	0.08 	&		1		&	0.0 	&&		6		{\,\,\,\,}	&	0.08 	&		1		&	0.0 	\\
 B60330	&&	33	&	33	&&	33	{\,\,\,\,}	&	0.42 	&	0.0 	&&		33		{\,\,\,\,}	&	0.57 	&		1		&	0.0 	&&		33		{\,\,\,\,}	&	0.57 	&		1		&	0.0 	\\
 B60350	&&	54	&	54	&&	54	{\,\,\,\,}	&	0.68 	&	0.0 	&&		54		{\,\,\,\,}	&	0.59 	&		1		&	0.0 	&&		54		{\,\,\,\,}	&	0.59 	&		1		&	0.0 	\\
 B60410	&&	13	&	13	&&	14	\,$\uparrow$	&	0.18 	&	7.1 	&&		13		{\,\,\,\,}	&	0.57 	&		3		&	0.0 	&&		13		{\,\,\,\,}	&	0.49 	&	\textbf{2}	&	0.0 	\\
 B60430	&&	31	&	31	&&	32	\,$\uparrow$	&	0.70 	&	3.1 	&&		31		{\,\,\,\,}	&	3.99 	&		12		&	0.0 	&&		31		{\,\,\,\,}	&	1.66 	&	\textbf{4}	&	0.0 	\\
 B60450	&&	46	&	46	&&	47	\,$\uparrow$	&	0.62 	&	2.1 	&&		46		{\,\,\,\,}	&	1.73 	&		3		&	0.0 	&&		46		{\,\,\,\,}	&	1.47 	&		3		&	0.0 	\\
 B60510	&&	13	&	13	&&	13	{\,\,\,\,}	&	0.59 	&	0.0 	&&		13		{\,\,\,\,}	&	0.50 	&		1		&	0.0 	&&		13		{\,\,\,\,}	&	0.50 	&		1		&	0.0 	\\
 B60530	&&	43	&	43	&&	43	{\,\,\,\,}	&	0.54 	&	0.0 	&&		43		{\,\,\,\,}	&	0.90 	&		1		&	0.0 	&&		43		{\,\,\,\,}	&	0.90 	&		1		&	0.0 	\\
 B60550	&&	71	&	71	&&	71	{\,\,\,\,}	&	0.35 	&	0.0 	&&		71		{\,\,\,\,}	&	1.22 	&		1		&	0.0 	&&		71		{\,\,\,\,}	&	1.22 	&		1		&	0.0 	\\
\hline
\end{tabular}
\end{table}

\begin{table}[htbp]
\centering \fontsize{7.5pt}{8.2pt}\selectfont
\setlength{\tabcolsep}{6.1pt}
\caption{MMR-SCP results for type-M instances}
\label{tbl:scp-m}
\begin{tabular}{ll*{16}{r}}
\hline
                             && \multicolumn{2}{c}{Best Known}                  && \multicolumn{3}{c}{DS}                  && \multicolumn{4}{c}{iDS-H}                                                                  && \multicolumn{4}{c}{iDS-B}                                                                      \\ \cline{3-4} \cline{6-8} \cline{10-13} \cline{15-18} 
\multicolumn{1}{c}{instance} && \multicolumn{1}{c}{LB} & \multicolumn{1}{c}{UB} &&
\multicolumn{1}{c}{obj} & \multicolumn{1}{c}{time} & \multicolumn{1}{c}{\%gap} &&
\multicolumn{1}{c}{obj} & \multicolumn{1}{c}{time} & \multicolumn{1}{c}{iter} & \multicolumn{1}{c}{\%gap} &&
\multicolumn{1}{c}{obj} & \multicolumn{1}{c}{time} & \multicolumn{1}{c}{iter} & \multicolumn{1}{c}{\%gap} \\ \hline
 M401-1	&&	3160	&	3160	&&	3160	{\,\,\,\,}	&	0.68 	&	0.0 	&&		3160		{\,\,\,\,}	&	0.68 	&		1		&	0.0 	&&		3160		{\,\,\,\,}	&	0.68 	&		1		&	0.0 	\\
 M401-2	&&	3495	&	3495	&&	3495	{\,\,\,\,}	&	0.32 	&	0.0 	&&		3495		{\,\,\,\,}	&	0.32 	&		1		&	0.0 	&&		3495		{\,\,\,\,}	&	0.32 	&		1		&	0.0 	\\
 M401-3	&&	3382	&	3382	&&	3382	{\,\,\,\,}	&	0.93 	&	0.0 	&&		3382		{\,\,\,\,}	&	0.93 	&		1		&	0.0 	&&		3382		{\,\,\,\,}	&	0.93 	&		1		&	0.0 	\\
 M402-1	&&	3610	&	3610	&&	3610	{\,\,\,\,}	&	1.00 	&	0.0 	&&		3610		{\,\,\,\,}	&	1.00 	&		1		&	0.0 	&&		3610		{\,\,\,\,}	&	1.00 	&		1		&	0.0 	\\
 M402-2	&&	2986	&	2986	&&	2986	{\,\,\,\,}	&	0.36 	&	0.0 	&&		2986		{\,\,\,\,}	&	0.36 	&		1		&	0.0 	&&		2986		{\,\,\,\,}	&	0.36 	&		1		&	0.0 	\\
 M402-3	&&	4294	&	4294	&&	4294	{\,\,\,\,}	&	1.12 	&	0.0 	&&		4294		{\,\,\,\,}	&	1.12 	&		1		&	0.0 	&&		4294		{\,\,\,\,}	&	1.12 	&		1		&	0.0 	\\
 M403-1	&&	3233	&	3233	&&	3233	{\,\,\,\,}	&	0.24 	&	0.0 	&&		3233		{\,\,\,\,}	&	0.24 	&		1		&	0.0 	&&		3233		{\,\,\,\,}	&	0.24 	&		1		&	0.0 	\\
 M403-2	&&	4205	&	4205	&&	4205	{\,\,\,\,}	&	0.60 	&	0.0 	&&		4205		{\,\,\,\,}	&	0.60 	&		1		&	0.0 	&&		4205		{\,\,\,\,}	&	0.60 	&		1		&	0.0 	\\
 M403-3	&&	3589	&	3589	&&	3594	\,$\uparrow$	&	0.79 	&	0.1 	&&		3589		{\,\,\,\,}	&	9.20 	&		11		&	0.0 	&&		3589		{\,\,\,\,}	&	1.94 	&	\textbf{2}	&	0.0 	\\
 M404-1	&&	3628	&	3628	&&	3628	{\,\,\,\,}	&	0.28 	&	0.0 	&&		3628		{\,\,\,\,}	&	0.28 	&		1		&	0.0 	&&		3628		{\,\,\,\,}	&	0.28 	&		1		&	0.0 	\\
 M404-2	&&	3954	&	3954	&&	3954	{\,\,\,\,}	&	0.80 	&	0.0 	&&		3954		{\,\,\,\,}	&	0.80 	&		1		&	0.0 	&&		3954		{\,\,\,\,}	&	0.80 	&		1		&	0.0 	\\
 M404-3	&&	2957	&	2957	&&	2957	{\,\,\,\,}	&	0.19 	&	0.0 	&&		2957		{\,\,\,\,}	&	0.19 	&		1		&	0.0 	&&		2957		{\,\,\,\,}	&	0.19 	&		1		&	0.0 	\\
 M405-1	&&	3546	&	3546	&&	3546	{\,\,\,\,}	&	0.70 	&	0.0 	&&		3546		{\,\,\,\,}	&	0.70 	&		1		&	0.0 	&&		3546		{\,\,\,\,}	&	0.70 	&		1		&	0.0 	\\
 M405-2	&&	3589	&	3589	&&	3589	{\,\,\,\,}	&	0.25 	&	0.0 	&&		3589		{\,\,\,\,}	&	0.25 	&		1		&	0.0 	&&		3589		{\,\,\,\,}	&	0.25 	&		1		&	0.0 	\\
 M405-3	&&	3698	&	3698	&&	3698	{\,\,\,\,}	&	0.66 	&	0.0 	&&		3698		{\,\,\,\,}	&	0.66 	&		1		&	0.0 	&&		3698		{\,\,\,\,}	&	0.66 	&		1		&	0.0 	\\
 M406-1	&&	3549	&	3549	&&	3549	{\,\,\,\,}	&	0.98 	&	0.0 	&&		3549		{\,\,\,\,}	&	0.98 	&		1		&	0.0 	&&		3549		{\,\,\,\,}	&	0.98 	&		1		&	0.0 	\\
 M406-2	&&	2975	&	2975	&&	2979	\,$\uparrow$	&	0.18 	&	0.1 	&&		2975		{\,\,\,\,}	&	1.89 	&		9		&	0.0 	&&		2975		{\,\,\,\,}	&	1.23 	&	\textbf{4}	&	0.0 	\\
 M406-3	&&	3199	&	3199	&&	3199	{\,\,\,\,}	&	0.26 	&	0.0 	&&		3199		{\,\,\,\,}	&	0.26 	&		1		&	0.0 	&&		3199		{\,\,\,\,}	&	0.26 	&		1		&	0.0 	\\
 M407-1	&&	3115	&	3115	&&	3115	{\,\,\,\,}	&	0.61 	&	0.0 	&&		3115		{\,\,\,\,}	&	0.61 	&		1		&	0.0 	&&		3115		{\,\,\,\,}	&	0.61 	&		1		&	0.0 	\\
 M407-2	&&	4674	&	4674	&&	4680	\,$\uparrow$	&	0.98 	&	0.1 	&&		4674		{\,\,\,\,}	&	20.78 	&		18		&	0.0 	&&		4674		{\,\,\,\,}	&	19.19 	&	\textbf{12}	&	0.0 	\\
 M407-3	&&	4136	&	4136	&&	4136	{\,\,\,\,}	&	0.41 	&	0.0 	&&		4136		{\,\,\,\,}	&	0.41 	&		1		&	0.0 	&&		4136		{\,\,\,\,}	&	0.41 	&		1		&	0.0 	\\
 M408-1	&&	3802	&	3802	&&	3802	{\,\,\,\,}	&	0.63 	&	0.0 	&&		3802		{\,\,\,\,}	&	0.63 	&		1		&	0.0 	&&		3802		{\,\,\,\,}	&	0.63 	&		1		&	0.0 	\\
 M408-2	&&	3302	&	3302	&&	3310	\,$\uparrow$	&	0.36 	&	0.2 	&&		3302		{\,\,\,\,}	&	0.61 	&		2		&	0.0 	&&		3302		{\,\,\,\,}	&	0.89 	&		2		&	0.0 	\\
 M408-3	&&	3147	&	3147	&&	3147	{\,\,\,\,}	&	0.21 	&	0.0 	&&		3147		{\,\,\,\,}	&	0.21 	&		1		&	0.0 	&&		3147		{\,\,\,\,}	&	0.21 	&		1		&	0.0 	\\
 M409-1	&&	3602	&	3602	&&	3602	{\,\,\,\,}	&	0.21 	&	0.0 	&&		3602		{\,\,\,\,}	&	0.21 	&		1		&	0.0 	&&		3602		{\,\,\,\,}	&	0.21 	&		1		&	0.0 	\\
 M409-2	&&	4222	&	4222	&&	4222	{\,\,\,\,}	&	0.58 	&	0.0 	&&		4222		{\,\,\,\,}	&	0.58 	&		1		&	0.0 	&&		4222		{\,\,\,\,}	&	0.58 	&		1		&	0.0 	\\
 M409-3	&&	3365	&	3365	&&	3365	{\,\,\,\,}	&	0.25 	&	0.0 	&&		3365		{\,\,\,\,}	&	0.25 	&		1		&	0.0 	&&		3365		{\,\,\,\,}	&	0.25 	&		1		&	0.0 	\\
 M410-1	&&	3000	&	3000	&&	3000	{\,\,\,\,}	&	0.11 	&	0.0 	&&		3000		{\,\,\,\,}	&	0.11 	&		1		&	0.0 	&&		3000		{\,\,\,\,}	&	0.11 	&		1		&	0.0 	\\
 M410-2	&&	4340	&	4340	&&	4340	{\,\,\,\,}	&	0.71 	&	0.0 	&&		4340		{\,\,\,\,}	&	0.71 	&		1		&	0.0 	&&		4340		{\,\,\,\,}	&	0.71 	&		1		&	0.0 	\\
 M410-3	&&	3235	&	3235	&&	3235	{\,\,\,\,}	&	0.10 	&	0.0 	&&		3235		{\,\,\,\,}	&	0.10 	&		1		&	0.0 	&&		3235		{\,\,\,\,}	&	0.10 	&		1		&	0.0 	\\
 M501-1	&&	1864	&	1864	&&	1864	{\,\,\,\,}	&	0.28 	&	0.0 	&&		1864		{\,\,\,\,}	&	0.28 	&		1		&	0.0 	&&		1864		{\,\,\,\,}	&	0.28 	&		1		&	0.0 	\\
 M501-2	&&	1683	&	1683	&&	1683	{\,\,\,\,}	&	0.09 	&	0.0 	&&		1683		{\,\,\,\,}	&	0.09 	&		1		&	0.0 	&&		1683		{\,\,\,\,}	&	0.09 	&		1		&	0.0 	\\
 M501-3	&&	1708	&	1708	&&	1708	{\,\,\,\,}	&	0.53 	&	0.0 	&&		1708		{\,\,\,\,}	&	0.53 	&		1		&	0.0 	&&		1708		{\,\,\,\,}	&	0.53 	&		1		&	0.0 	\\
 M502-1	&&	1972	&	1972	&&	1972	{\,\,\,\,}	&	0.88 	&	0.0 	&&		1972		{\,\,\,\,}	&	0.88 	&		1		&	0.0 	&&		1972		{\,\,\,\,}	&	0.88 	&		1		&	0.0 	\\
 M502-2	&&	1805	&	1805	&&	1805	{\,\,\,\,}	&	0.07 	&	0.0 	&&		1805		{\,\,\,\,}	&	0.07 	&		1		&	0.0 	&&		1805		{\,\,\,\,}	&	0.07 	&		1		&	0.0 	\\
 M502-3	&&	1930	&	1930	&&	1930	{\,\,\,\,}	&	0.12 	&	0.0 	&&		1930		{\,\,\,\,}	&	0.12 	&		1		&	0.0 	&&		1930		{\,\,\,\,}	&	0.12 	&		1		&	0.0 	\\
 M503-1	&&	1889	&	1889	&&	1889	{\,\,\,\,}	&	0.38 	&	0.0 	&&		1889		{\,\,\,\,}	&	0.38 	&		1		&	0.0 	&&		1889		{\,\,\,\,}	&	0.38 	&		1		&	0.0 	\\
 M503-2	&&	2220	&	2220	&&	2220	{\,\,\,\,}	&	0.42 	&	0.0 	&&		2220		{\,\,\,\,}	&	0.42 	&		1		&	0.0 	&&		2220		{\,\,\,\,}	&	0.42 	&		1		&	0.0 	\\
 M503-3	&&	1571	&	1571	&&	1571	{\,\,\,\,}	&	0.07 	&	0.0 	&&		1571		{\,\,\,\,}	&	0.07 	&		1		&	0.0 	&&		1571		{\,\,\,\,}	&	0.07 	&		1		&	0.0 	\\
 M504-1	&&	2179	&	2179	&&	2184	\,$\uparrow$	&	0.67 	&	0.2 	&&		2179		{\,\,\,\,}	&	3.32 	&		5		&	0.0 	&&		2179		{\,\,\,\,}	&	2.16 	&	\textbf{3}	&	0.0 	\\
 M504-2	&&	1902	&	1902	&&	1902	{\,\,\,\,}	&	0.10 	&	0.0 	&&		1902		{\,\,\,\,}	&	0.10 	&		1		&	0.0 	&&		1902		{\,\,\,\,}	&	0.10 	&		1		&	0.0 	\\
 M504-3	&&	1870	&	1870	&&	1870	{\,\,\,\,}	&	0.17 	&	0.0 	&&		1870		{\,\,\,\,}	&	0.17 	&		1		&	0.0 	&&		1870		{\,\,\,\,}	&	0.17 	&		1		&	0.0 	\\
 M505-1	&&	1998	&	1998	&&	1998	{\,\,\,\,}	&	0.57 	&	0.0 	&&		1998		{\,\,\,\,}	&	0.57 	&		1		&	0.0 	&&		1998		{\,\,\,\,}	&	0.57 	&		1		&	0.0 	\\
 M505-2	&&	1781	&	1781	&&	1781	{\,\,\,\,}	&	0.48 	&	0.0 	&&		1781		{\,\,\,\,}	&	0.48 	&		1		&	0.0 	&&		1781		{\,\,\,\,}	&	0.48 	&		1		&	0.0 	\\
 M505-3	&&	1869	&	1869	&&	1869	{\,\,\,\,}	&	0.86 	&	0.0 	&&		1869		{\,\,\,\,}	&	0.86 	&		1		&	0.0 	&&		1869		{\,\,\,\,}	&	0.86 	&		1		&	0.0 	\\
 M506-1	&&	1803	&	1803	&&	1803	{\,\,\,\,}	&	0.11 	&	0.0 	&&		1803		{\,\,\,\,}	&	0.11 	&		1		&	0.0 	&&		1803		{\,\,\,\,}	&	0.11 	&		1		&	0.0 	\\
 M506-2	&&	1943	&	1943	&&	1943	{\,\,\,\,}	&	0.33 	&	0.0 	&&		1943		{\,\,\,\,}	&	0.33 	&		1		&	0.0 	&&		1943		{\,\,\,\,}	&	0.33 	&		1		&	0.0 	\\
 M506-3	&&	2001	&	2001	&&	2001	{\,\,\,\,}	&	0.44 	&	0.0 	&&		2001		{\,\,\,\,}	&	0.44 	&		1		&	0.0 	&&		2001		{\,\,\,\,}	&	0.44 	&		1		&	0.0 	\\
 M507-1	&&	2075	&	2075	&&	2075	{\,\,\,\,}	&	0.16 	&	0.0 	&&		2075		{\,\,\,\,}	&	0.16 	&		1		&	0.0 	&&		2075		{\,\,\,\,}	&	0.16 	&		1		&	0.0 	\\
 M507-2	&&	1878	&	1878	&&	1878	{\,\,\,\,}	&	0.19 	&	0.0 	&&		1878		{\,\,\,\,}	&	0.19 	&		1		&	0.0 	&&		1878		{\,\,\,\,}	&	0.19 	&		1		&	0.0 	\\
 M507-3	&&	1791	&	1791	&&	1791	{\,\,\,\,}	&	1.04 	&	0.0 	&&		1791		{\,\,\,\,}	&	1.04 	&		1		&	0.0 	&&		1791		{\,\,\,\,}	&	1.04 	&		1		&	0.0 	\\
 M508-1	&&	1656	&	1656	&&	1656	{\,\,\,\,}	&	0.32 	&	0.0 	&&		1656		{\,\,\,\,}	&	0.32 	&		1		&	0.0 	&&		1656		{\,\,\,\,}	&	0.32 	&		1		&	0.0 	\\
 M508-2	&&	1570	&	1570	&&	1570	{\,\,\,\,}	&	0.47 	&	0.0 	&&		1570		{\,\,\,\,}	&	0.47 	&		1		&	0.0 	&&		1570		{\,\,\,\,}	&	0.47 	&		1		&	0.0 	\\
 M508-3	&&	1765	&	1765	&&	1765	{\,\,\,\,}	&	0.31 	&	0.0 	&&		1765		{\,\,\,\,}	&	0.31 	&		1		&	0.0 	&&		1765		{\,\,\,\,}	&	0.31 	&		1		&	0.0 	\\
 M509-1	&&	1710	&	1710	&&	1710	{\,\,\,\,}	&	0.29 	&	0.0 	&&		1710		{\,\,\,\,}	&	0.29 	&		1		&	0.0 	&&		1710		{\,\,\,\,}	&	0.29 	&		1		&	0.0 	\\
 M509-2	&&	1769	&	1769	&&	1769	{\,\,\,\,}	&	0.53 	&	0.0 	&&		1769		{\,\,\,\,}	&	0.53 	&		1		&	0.0 	&&		1769		{\,\,\,\,}	&	0.53 	&		1		&	0.0 	\\
 M509-3	&&	2003	&	2003	&&	2003	{\,\,\,\,}	&	0.19 	&	0.0 	&&		2003		{\,\,\,\,}	&	0.19 	&		1		&	0.0 	&&		2003		{\,\,\,\,}	&	0.19 	&		1		&	0.0 	\\
 M510-1	&&	2217	&	2217	&&	2217	{\,\,\,\,}	&	0.13 	&	0.0 	&&		2217		{\,\,\,\,}	&	0.13 	&		1		&	0.0 	&&		2217		{\,\,\,\,}	&	0.13 	&		1		&	0.0 	\\
 M510-2	&&	1937	&	1937	&&	1937	{\,\,\,\,}	&	0.49 	&	0.0 	&&		1937		{\,\,\,\,}	&	0.49 	&		1		&	0.0 	&&		1937		{\,\,\,\,}	&	0.49 	&		1		&	0.0 	\\
 M510-3	&&	1838	&	1838	&&	1838	{\,\,\,\,}	&	0.25 	&	0.0 	&&		1838		{\,\,\,\,}	&	0.25 	&		1		&	0.0 	&&		1838		{\,\,\,\,}	&	0.25 	&		1		&	0.0 	\\
 M601-1	&&	800	&	800	&&	800	{\,\,\,\,}	&	0.55 	&	0.0 	&&		800		{\,\,\,\,}	&	0.55 	&		1		&	0.0 	&&		800		{\,\,\,\,}	&	0.55 	&		1		&	0.0 	\\
 M601-2	&&	1340	&	1340	&&	1340	{\,\,\,\,}	&	0.82 	&	0.0 	&&		1340		{\,\,\,\,}	&	0.82 	&		1		&	0.0 	&&		1340		{\,\,\,\,}	&	0.82 	&		1		&	0.0 	\\
 M601-3	&&	1091	&	1091	&&	1096	\,$\uparrow$	&	0.90 	&	0.5 	&&		1091		{\,\,\,\,}	&	2.73 	&		3		&	0.0 	&&		1091		{\,\,\,\,}	&	3.18 	&		3		&	0.0 	\\
 M602-1	&&	1112	&	1112	&&	1114	\,$\uparrow$	&	0.74 	&	0.2 	&&		1112		{\,\,\,\,}	&	3.55 	&		4		&	0.0 	&&		1112		{\,\,\,\,}	&	2.55 	&	\textbf{3}	&	0.0 	\\
 M602-2	&&	1052	&	1052	&&	1052	{\,\,\,\,}	&	1.01 	&	0.0 	&&		1052		{\,\,\,\,}	&	1.01 	&		1		&	0.0 	&&		1052		{\,\,\,\,}	&	1.01 	&		1		&	0.0 	\\
 M602-3	&&	1264	&	1264	&&	1268	\,$\uparrow$	&	0.89 	&	0.3 	&&		1264		{\,\,\,\,}	&	4.46 	&		5		&	0.0 	&&		1264		{\,\,\,\,}	&	2.65 	&	\textbf{3}	&	0.0 	\\
 M603-1	&&	683	&	683	&&	683	{\,\,\,\,}	&	0.29 	&	0.0 	&&		683		{\,\,\,\,}	&	0.29 	&		1		&	0.0 	&&		683		{\,\,\,\,}	&	0.29 	&		1		&	0.0 	\\
 M603-2	&&	966	&	966	&&	966	{\,\,\,\,}	&	0.42 	&	0.0 	&&		966		{\,\,\,\,}	&	0.42 	&		1		&	0.0 	&&		966		{\,\,\,\,}	&	0.42 	&		1		&	0.0 	\\
 M603-3	&&	1200	&	1200	&&	1200	{\,\,\,\,}	&	0.68 	&	0.0 	&&		1200		{\,\,\,\,}	&	0.68 	&		1		&	0.0 	&&		1200		{\,\,\,\,}	&	0.68 	&		1		&	0.0 	\\
 M604-1	&&	1022	&	1022	&&	1022	{\,\,\,\,}	&	0.18 	&	0.0 	&&		1022		{\,\,\,\,}	&	0.18 	&		1		&	0.0 	&&		1022		{\,\,\,\,}	&	0.18 	&		1		&	0.0 	\\
 M604-2	&&	1055	&	1055	&&	1055	{\,\,\,\,}	&	0.70 	&	0.0 	&&		1055		{\,\,\,\,}	&	0.70 	&		1		&	0.0 	&&		1055		{\,\,\,\,}	&	0.70 	&		1		&	0.0 	\\
 M604-3	&&	971	&	971	&&	971	{\,\,\,\,}	&	0.62 	&	0.0 	&&		971		{\,\,\,\,}	&	0.62 	&		1		&	0.0 	&&		971		{\,\,\,\,}	&	0.62 	&		1		&	0.0 	\\
 M605-1	&&	1083	&	1083	&&	1083	{\,\,\,\,}	&	0.59 	&	0.0 	&&		1083		{\,\,\,\,}	&	0.59 	&		1		&	0.0 	&&		1083		{\,\,\,\,}	&	0.59 	&		1		&	0.0 	\\
 M605-2	&&	1053	&	1053	&&	1053	{\,\,\,\,}	&	0.54 	&	0.0 	&&		1053		{\,\,\,\,}	&	0.54 	&		1		&	0.0 	&&		1053		{\,\,\,\,}	&	0.54 	&		1		&	0.0 	\\
 M605-3	&&	1031	&	1031	&&	1031	{\,\,\,\,}	&	0.35 	&	0.0 	&&		1031		{\,\,\,\,}	&	0.35 	&		1		&	0.0 	&&		1031		{\,\,\,\,}	&	0.35 	&		1		&	0.0 	\\
\hline
\end{tabular}
\end{table}

\begin{table}[htbp]
\centering \fontsize{7.5pt}{8.2pt}\selectfont
\setlength{\tabcolsep}{4.7pt}
\caption{MMR-SCP results for type-K instances}
\label{tbl:scp-k}
\begin{tabular}{ll*{16}{r}}
\hline
                             && \multicolumn{2}{c}{Best Known}                  && \multicolumn{3}{c}{DS}                  && \multicolumn{4}{c}{iDS-H}                                                                  && \multicolumn{4}{c}{iDS-B}                                                                      \\ \cline{3-4} \cline{6-8} \cline{10-13} \cline{15-18} 
\multicolumn{1}{c}{instance} && \multicolumn{1}{c}{LB} & \multicolumn{1}{c}{UB} &&
\multicolumn{1}{c}{obj} & \multicolumn{1}{c}{time} & \multicolumn{1}{c}{\%gap} &&
\multicolumn{1}{c}{obj} & \multicolumn{1}{c}{time} & \multicolumn{1}{c}{iter} & \multicolumn{1}{c}{\%gap} &&
\multicolumn{1}{c}{obj} & \multicolumn{1}{c}{time} & \multicolumn{1}{c}{iter} & \multicolumn{1}{c}{\%gap} \\ \hline
 K401-1	&&	13365	&	14440	&&	14440	{\,\,\,\,}	&	20.403 	&	7.4 	&&		14440		{\,\,\,\,}	&	20.4 	&		1		&	7.4 	&&		14440		{\,\,\,\,}	&	20.4 	&		1		&	7.4 	\\
 K401-2	&&	14026	&	16372	&&	16247	\,$\downarrow$	&	419.292 	&	13.7 	&&		16243		\,$\downarrow$	&	897.8 	&		2		&	13.6 	&&		16243		\,$\downarrow$	&	912.1 	&		2		&	13.6 	\\
 K401-3	&&	11803	&	12974	&&	12974	{\,\,\,\,}	&	45.133 	&	9.0 	&&		12974		{\,\,\,\,}	&	45.2 	&		1		&	9.0 	&&		12974		{\,\,\,\,}	&	45.1 	&		1		&	9.0 	\\
 K402-1	&&	12628	&	14235	&&	14235	{\,\,\,\,}	&	77.207 	&	11.3 	&&		14235		{\,\,\,\,}	&	77.2 	&		1		&	11.3 	&&		14235		{\,\,\,\,}	&	77.2 	&		1		&	11.3 	\\
 K402-2	&&	14325	&	16335	&&	16351	\,$\uparrow$	&	155.660 	&	12.4 	&&		16335		{\,\,\,\,}	&	496.7 	&		3		&	12.3 	&&		16335		{\,\,\,\,}	&	488.5 	&		3		&	12.3 	\\
 K402-3	&&	12445	&	14306	&&	14306	{\,\,\,\,}	&	351.183 	&	13.0 	&&		14306		{\,\,\,\,}	&	351.2 	&		1		&	13.0 	&&		14306		{\,\,\,\,}	&	351.4 	&		1		&	13.0 	\\
 K403-1	&&	14031	&	15139	&&	15139	{\,\,\,\,}	&	28.780 	&	7.3 	&&		15139		{\,\,\,\,}	&	28.8 	&		1		&	7.3 	&&		15139		{\,\,\,\,}	&	28.8 	&		1		&	7.3 	\\
 K403-2	&&	14349	&	16523	&&	16523	{\,\,\,\,}	&	176.895 	&	13.2 	&&		16523		{\,\,\,\,}	&	177.3 	&		1		&	13.2 	&&		16523		{\,\,\,\,}	&	177.6 	&		1		&	13.2 	\\
 K403-3	&&	12436	&	13613	&&	13613	{\,\,\,\,}	&	30.240 	&	8.6 	&&		13613		{\,\,\,\,}	&	30.2 	&		1		&	8.6 	&&		13613		{\,\,\,\,}	&	30.2 	&		1		&	8.6 	\\
 K404-1	&&	13046	&	13472	&&	13472	{\,\,\,\,}	&	9.412 	&	3.2 	&&		13472		{\,\,\,\,}	&	9.4 	&		1		&	3.2 	&&		13472		{\,\,\,\,}	&	9.4 	&		1		&	3.2 	\\
 K404-2	&&	13291	&	15871	&&	15820	\,$\downarrow$	&	580.562 	&	16.0 	&&		15820		\,$\downarrow$	&	580.6 	&		1		&	16.0 	&&	\textbf{15813}	\,$\downarrow$	&	4110.3 	&		7		&	15.9 	\\
 K404-3	&&	13255	&	14898	&&	14954	\,$\uparrow$	&	38.057 	&	11.4 	&&		14881		\,$\downarrow$	&	152.8 	&		4		&	10.9 	&&		14881		\,$\downarrow$	&	146.6 	&		4		&	10.9 	\\
 K405-1	&&	13518	&	16242	&&	16242	{\,\,\,\,}	&	296.749 	&	16.8 	&&		16242		{\,\,\,\,}	&	297.8 	&		1		&	16.8 	&&		16242		{\,\,\,\,}	&	296.7 	&		1		&	16.8 	\\
 K405-2	&&	14204	&	16311	&&	16355	\,$\uparrow$	&	77.639 	&	13.2 	&&		16311		{\,\,\,\,}	&	372.7 	&		4		&	12.9 	&&		16311		{\,\,\,\,}	&	348.7 	&		4		&	12.9 	\\
 K405-3	&&	12933	&	14549	&&	14549	{\,\,\,\,}	&	71.876 	&	11.1 	&&		14539		\,$\downarrow$	&	147.4 	&		2		&	11.0 	&&		14539		\,$\downarrow$	&	167.1 	&		2		&	11.0 	\\
 K406-1	&&	12885	&	14067	&&	14067	{\,\,\,\,}	&	90.892 	&	8.4 	&&		14067		{\,\,\,\,}	&	91.1 	&		1		&	8.4 	&&		14067		{\,\,\,\,}	&	90.9 	&		1		&	8.4 	\\
 K406-2	&&	12522	&	14551	&&	14555	\,$\uparrow$	&	199.734 	&	14.0 	&&		14551		{\,\,\,\,}	&	427.5 	&		2		&	13.9 	&&		14551		{\,\,\,\,}	&	390.6 	&		2		&	13.9 	\\
 K406-3	&&	12312	&	13941	&&	13941	{\,\,\,\,}	&	177.886 	&	11.7 	&&		13941		{\,\,\,\,}	&	178.4 	&		1		&	11.7 	&&		13941		{\,\,\,\,}	&	177.9 	&		1		&	11.7 	\\
 K407-1	&&	13055	&	14928	&&	14908	\,$\downarrow$	&	123.242 	&	12.4 	&&		14908		\,$\downarrow$	&	123.9 	&		1		&	12.4 	&&		14908		\,$\downarrow$	&	123.2 	&		1		&	12.4 	\\
 K407-2	&&	13509	&	15875	&&	15901	\,$\uparrow$	&	380.145 	&	15.0 	&&		15875		{\,\,\,\,}	&	1166.1 	&		3		&	14.9 	&&		15875		{\,\,\,\,}	&	1151.9 	&		3		&	14.9 	\\
 K407-3	&&	13465	&	15262	&&	15268	\,$\uparrow$	&	41.457 	&	11.8 	&&		15262		{\,\,\,\,}	&	246.8 	&		5		&	11.8 	&&		15262		{\,\,\,\,}	&	108.8 	&	\textbf{2}	&	11.8 	\\
 K408-1	&&	12193	&	13752	&&	13654	\,$\downarrow$	&	91.841 	&	10.7 	&&		13654		\,$\downarrow$	&	91.8 	&		1		&	10.7 	&&		13654		\,$\downarrow$	&	92.0 	&		1		&	10.7 	\\
 K408-2	&&	13907	&	15843	&&	15843	{\,\,\,\,}	&	47.826 	&	12.2 	&&		15843		{\,\,\,\,}	&	47.9 	&		1		&	12.2 	&&		15843		{\,\,\,\,}	&	47.8 	&		1		&	12.2 	\\
 K408-3	&&	12430	&	13566	&&	13566	{\,\,\,\,}	&	34.834 	&	8.4 	&&		13566		{\,\,\,\,}	&	34.8 	&		1		&	8.4 	&&		13566		{\,\,\,\,}	&	34.8 	&		1		&	8.4 	\\
 K409-1	&&	13283	&	14872	&&	14872	{\,\,\,\,}	&	160.887 	&	10.7 	&&		14872		{\,\,\,\,}	&	161.2 	&		1		&	10.7 	&&		14872		{\,\,\,\,}	&	160.9 	&		1		&	10.7 	\\
 K409-2	&&	12285	&	14020	&&	13971	\,$\downarrow$	&	117.825 	&	12.1 	&&		13971		\,$\downarrow$	&	117.8 	&		1		&	12.1 	&&		13971		\,$\downarrow$	&	117.8 	&		1		&	12.1 	\\
 K409-3	&&	12949	&	14414	&&	14414	{\,\,\,\,}	&	50.856 	&	10.2 	&&		14414		{\,\,\,\,}	&	50.9 	&		1		&	10.2 	&&		14414		{\,\,\,\,}	&	50.9 	&		1		&	10.2 	\\
 K410-1	&&	13049	&	15367	&&	15287	\,$\downarrow$	&	93.437 	&	14.6 	&&		15254		\,$\downarrow$	&	316.7 	&		3		&	14.5 	&&		15254		\,$\downarrow$	&	291.5 	&		3		&	14.5 	\\
 K410-2	&&	14265	&	15903	&&	15903	{\,\,\,\,}	&	168.399 	&	10.3 	&&		15903		{\,\,\,\,}	&	168.4 	&		1		&	10.3 	&&		15903		{\,\,\,\,}	&	168.7 	&		1		&	10.3 	\\
 K410-3	&&	13983	&	16635	&&	16635	{\,\,\,\,}	&	344.680 	&	15.9 	&&		16635		{\,\,\,\,}	&	346.0 	&		1		&	15.9 	&&		16635		{\,\,\,\,}	&	346.3 	&		1		&	15.9 	\\
 K501-1	&&	10925	&	11743	&&	11743	{\,\,\,\,}	&	36.964 	&	7.0 	&&		11743		{\,\,\,\,}	&	37.0 	&		1		&	7.0 	&&		11743		{\,\,\,\,}	&	37.0 	&		1		&	7.0 	\\
 K501-2	&&	10997	&	11899	&&	11899	{\,\,\,\,}	&	16.887 	&	7.6 	&&		11899		{\,\,\,\,}	&	16.9 	&		1		&	7.6 	&&		11899		{\,\,\,\,}	&	16.9 	&		1		&	7.6 	\\
 K501-3	&&	11328	&	11631	&&	11631	{\,\,\,\,}	&	12.344 	&	2.6 	&&		11631		{\,\,\,\,}	&	12.3 	&		1		&	2.6 	&&		11631		{\,\,\,\,}	&	12.3 	&		1		&	2.6 	\\
 K502-1	&&	9891	&	10144	&&	10152	\,$\uparrow$	&	15.681 	&	2.6 	&&		10144		{\,\,\,\,}	&	67.2 	&		3		&	2.5 	&&		10144		{\,\,\,\,}	&	58.5 	&		3		&	2.5 	\\
 K502-2	&&	11450	&	12422	&&	12422	{\,\,\,\,}	&	125.340 	&	7.8 	&&		12422		{\,\,\,\,}	&	125.3 	&		1		&	7.8 	&&		12422		{\,\,\,\,}	&	125.3 	&		1		&	7.8 	\\
 K502-3	&&	10465	&	11016	&&	11016	{\,\,\,\,}	&	18.901 	&	5.0 	&&		11016		{\,\,\,\,}	&	18.9 	&		1		&	5.0 	&&		11016		{\,\,\,\,}	&	18.9 	&		1		&	5.0 	\\
 K503-1	&&	10265	&	10265	&&	10291	\,$\uparrow$	&	4.596 	&	0.3 	&&		10265		{\,\,\,\,}	&	9.2 	&		2		&	0.0 	&&		10265		{\,\,\,\,}	&	9.5 	&		2		&	0.0 	\\
 K503-2	&&	11305	&	12321	&&	12321	{\,\,\,\,}	&	22.899 	&	8.2 	&&		12321		{\,\,\,\,}	&	22.9 	&		1		&	8.2 	&&		12321		{\,\,\,\,}	&	22.9 	&		1		&	8.2 	\\
 K503-3	&&	10650	&	11957	&&	11919	\,$\downarrow$	&	69.742 	&	10.6 	&&		11912		\,$\downarrow$	&	134.9 	&		2		&	10.6 	&&		11912		\,$\downarrow$	&	128.8 	&		2		&	10.6 	\\
 K504-1	&&	10931	&	11429	&&	11429	{\,\,\,\,}	&	7.445 	&	4.4 	&&		11429		{\,\,\,\,}	&	7.4 	&		1		&	4.4 	&&		11429		{\,\,\,\,}	&	7.4 	&		1		&	4.4 	\\
 K504-2	&&	12234	&	13388	&&	13388	{\,\,\,\,}	&	30.963 	&	8.6 	&&		13388		{\,\,\,\,}	&	31.0 	&		1		&	8.6 	&&		13388		{\,\,\,\,}	&	31.0 	&		1		&	8.6 	\\
 K504-3	&&	11281	&	11943	&&	11943	{\,\,\,\,}	&	22.659 	&	5.5 	&&		11943		{\,\,\,\,}	&	22.7 	&		1		&	5.5 	&&		11943		{\,\,\,\,}	&	22.7 	&		1		&	5.5 	\\
 K505-1	&&	11342	&	12102	&&	12116	\,$\uparrow$	&	10.666 	&	6.4 	&&		12102		{\,\,\,\,}	&	20.6 	&		2		&	6.3 	&&		12102		{\,\,\,\,}	&	21.4 	&		2		&	6.3 	\\
 K505-2	&&	11847	&	13663	&&	13663	{\,\,\,\,}	&	342.522 	&	13.3 	&&		13663		{\,\,\,\,}	&	342.5 	&		1		&	13.3 	&&		13663		{\,\,\,\,}	&	342.7 	&		1		&	13.3 	\\
 K505-3	&&	10815	&	12159	&&	12159	{\,\,\,\,}	&	253.301 	&	11.1 	&&		12159		{\,\,\,\,}	&	253.3 	&		1		&	11.1 	&&		12159		{\,\,\,\,}	&	253.3 	&		1		&	11.1 	\\
 K506-1	&&	10047	&	10232	&&	10232	{\,\,\,\,}	&	10.340 	&	1.8 	&&		10232		{\,\,\,\,}	&	10.3 	&		1		&	1.8 	&&		10232		{\,\,\,\,}	&	10.3 	&		1		&	1.8 	\\
 K506-2	&&	11645	&	12236	&&	12237	\,$\uparrow$	&	31.781 	&	4.8 	&&		12236		{\,\,\,\,}	&	64.9 	&		2		&	4.8 	&&		12236		{\,\,\,\,}	&	63.6 	&		2		&	4.8 	\\
 K506-3	&&	9858	&	10291	&&	10291	{\,\,\,\,}	&	10.506 	&	4.2 	&&		10291		{\,\,\,\,}	&	10.5 	&		1		&	4.2 	&&		10291		{\,\,\,\,}	&	10.5 	&		1		&	4.2 	\\
 K507-1	&&	10605	&	10661	&&	10661	{\,\,\,\,}	&	3.128 	&	0.5 	&&		10661		{\,\,\,\,}	&	3.1 	&		1		&	0.5 	&&		10661		{\,\,\,\,}	&	3.1 	&		1		&	0.5 	\\
 K507-2	&&	11591	&	12200	&&	12200	{\,\,\,\,}	&	27.002 	&	5.0 	&&		12200		{\,\,\,\,}	&	27.0 	&		1		&	5.0 	&&		12200		{\,\,\,\,}	&	27.0 	&		1		&	5.0 	\\
 K507-3	&&	10506	&	11105	&&	11105	{\,\,\,\,}	&	16.122 	&	5.4 	&&		11105		{\,\,\,\,}	&	16.1 	&		1		&	5.4 	&&		11105		{\,\,\,\,}	&	16.1 	&		1		&	5.4 	\\
 K508-1	&&	10641	&	11095	&&	11095	{\,\,\,\,}	&	4.337 	&	4.1 	&&		11095		{\,\,\,\,}	&	4.3 	&		1		&	4.1 	&&		11095		{\,\,\,\,}	&	4.3 	&		1		&	4.1 	\\
 K508-2	&&	11352	&	12557	&&	12527	\,$\downarrow$	&	74.643 	&	9.4 	&&		12514		\,$\downarrow$	&	158.0 	&		2		&	9.3 	&&		12514		\,$\downarrow$	&	138.8 	&		2		&	9.3 	\\
 K508-3	&&	11179	&	11554	&&	11554	{\,\,\,\,}	&	9.094 	&	3.2 	&&		11554		{\,\,\,\,}	&	9.1 	&		1		&	3.2 	&&		11554		{\,\,\,\,}	&	9.1 	&		1		&	3.2 	\\
 K509-1	&&	11362	&	12151	&&	12111	\,$\downarrow$	&	49.222 	&	6.2 	&&		12111		\,$\downarrow$	&	49.2 	&		1		&	6.2 	&&		12111		\,$\downarrow$	&	49.2 	&		1		&	6.2 	\\
 K509-2	&&	12101	&	13236	&&	13246	\,$\uparrow$	&	66.195 	&	8.6 	&&		13236		{\,\,\,\,}	&	197.6 	&		3		&	8.6 	&&		13236		{\,\,\,\,}	&	168.0 	&		3		&	8.6 	\\
 K509-3	&&	11061	&	11862	&&	11871	\,$\uparrow$	&	9.010 	&	6.8 	&&		11862		{\,\,\,\,}	&	36.8 	&		3		&	6.8 	&&		11862		{\,\,\,\,}	&	26.2 	&		3		&	6.8 	\\
 K510-1	&&	10443	&	10969	&&	11006	\,$\uparrow$	&	23.418 	&	5.1 	&&		10969		{\,\,\,\,}	&	256.6 	&		9		&	4.8 	&&		10969		{\,\,\,\,}	&	201.6 	&		9		&	4.8 	\\
 K510-2	&&	11437	&	12298	&&	12300	\,$\uparrow$	&	39.650 	&	7.0 	&&		12298		{\,\,\,\,}	&	128.0 	&		3		&	7.0 	&&		12298		{\,\,\,\,}	&	116.1 	&		3		&	7.0 	\\
 K510-3	&&	11276	&	12253	&&	12253	{\,\,\,\,}	&	41.684 	&	8.0 	&&		12253		{\,\,\,\,}	&	41.7 	&		1		&	8.0 	&&		12253		{\,\,\,\,}	&	41.7 	&		1		&	8.0 	\\
 K601-1	&&	7099	&	7127	&&	7136	\,$\uparrow$	&	110.892 	&	0.5 	&&		7127		{\,\,\,\,}	&	208.5 	&		2		&	0.4 	&&		7127		{\,\,\,\,}	&	199.3 	&		2		&	0.4 	\\
 K601-2	&&	7190	&	7242	&&	7242	{\,\,\,\,}	&	344.573 	&	0.7 	&&		7242		{\,\,\,\,}	&	344.6 	&		1		&	0.7 	&&		7242		{\,\,\,\,}	&	344.6 	&		1		&	0.7 	\\
 K601-3	&&	5834	&	5834	&&	5834	{\,\,\,\,}	&	73.339 	&	0.0 	&&		5834		{\,\,\,\,}	&	73.3 	&		1		&	0.0 	&&		5834		{\,\,\,\,}	&	73.3 	&		1		&	0.0 	\\
 K602-1	&&	6640	&	6640	&&	6640	{\,\,\,\,}	&	77.797 	&	0.0 	&&		6640		{\,\,\,\,}	&	77.8 	&		1		&	0.0 	&&		6640		{\,\,\,\,}	&	77.8 	&		1		&	0.0 	\\
 K602-2	&&	7146	&	7146	&&	7165	\,$\uparrow$	&	41.927 	&	0.3 	&&		7146		{\,\,\,\,}	&	81.3 	&		2		&	0.0 	&&		7146		{\,\,\,\,}	&	74.1 	&		2		&	0.0 	\\
 K602-3	&&	6605	&	6733	&&	6733	{\,\,\,\,}	&	72.275 	&	1.9 	&&		6733		{\,\,\,\,}	&	72.3 	&		1		&	1.9 	&&		6733		{\,\,\,\,}	&	72.3 	&		1		&	1.9 	\\
 K603-1	&&	6986	&	6986	&&	7023	\,$\uparrow$	&	43.893 	&	0.5 	&&		6986		{\,\,\,\,}	&	98.3 	&		2		&	0.0 	&&		6986		{\,\,\,\,}	&	81.9 	&		2		&	0.0 	\\
 K603-2	&&	7404	&	7828	&&	7828	{\,\,\,\,}	&	108.521 	&	5.4 	&&		7828		{\,\,\,\,}	&	108.5 	&		1		&	5.4 	&&		7828		{\,\,\,\,}	&	108.5 	&		1		&	5.4 	\\
 K603-3	&&	6943	&	6943	&&	6943	{\,\,\,\,}	&	64.947 	&	0.0 	&&		6943		{\,\,\,\,}	&	64.9 	&		1		&	0.0 	&&		6943		{\,\,\,\,}	&	64.9 	&		1		&	0.0 	\\
 K604-1	&&	6343	&	6343	&&	6343	{\,\,\,\,}	&	53.900 	&	0.0 	&&		6343		{\,\,\,\,}	&	53.9 	&		1		&	0.0 	&&		6343		{\,\,\,\,}	&	53.9 	&		1		&	0.0 	\\
 K604-2	&&	7271	&	7822	&&	7822	{\,\,\,\,}	&	616.776 	&	7.0 	&&		7822		{\,\,\,\,}	&	616.8 	&		1		&	7.0 	&&		7822		{\,\,\,\,}	&	616.8 	&		1		&	7.0 	\\
 K604-3	&&	5673	&	5673	&&	5673	{\,\,\,\,}	&	4.224 	&	0.0 	&&		5673		{\,\,\,\,}	&	4.2 	&		1		&	0.0 	&&		5673		{\,\,\,\,}	&	4.2 	&		1		&	0.0 	\\
 K605-1	&&	6641	&	6641	&&	6641	{\,\,\,\,}	&	34.390 	&	0.0 	&&		6641		{\,\,\,\,}	&	34.4 	&		1		&	0.0 	&&		6641		{\,\,\,\,}	&	34.4 	&		1		&	0.0 	\\
 K605-2	&&	7589	&	7674	&&	7674	{\,\,\,\,}	&	146.931 	&	1.1 	&&		7674		{\,\,\,\,}	&	146.9 	&		1		&	1.1 	&&		7674		{\,\,\,\,}	&	146.9 	&		1		&	1.1 	\\
 K605-3	&&	7189	&	7462	&&	7462	{\,\,\,\,}	&	216.769 	&	3.7 	&&		7462		{\,\,\,\,}	&	216.8 	&		1		&	3.7 	&&		7462		{\,\,\,\,}	&	216.8 	&		1		&	3.7 	\\
\hline
\end{tabular}
\end{table}

%% file: tableGAP.tex
\begin{table}[htbp]
\centering \fontsize{8.5pt}{10.1pt}\selectfont
\setlength{\tabcolsep}{4.3pt}
\caption{MMR-GAP results for type-A instances}
\label{tbl:gap-a}
\begin{tabular}{ll*{16}{r}}
\hline
                             && \multicolumn{2}{c}{Best Known}                  && \multicolumn{3}{c}{DS}                  && \multicolumn{4}{c}{iDS-H}                                                                  && \multicolumn{4}{c}{iDS-B}                                                                      \\ \cline{3-4} \cline{6-8} \cline{10-13} \cline{15-18} 
\multicolumn{1}{c}{instance} && \multicolumn{1}{c}{LB} & \multicolumn{1}{c}{UB} &&
\multicolumn{1}{c}{obj} & \multicolumn{1}{c}{time} & \multicolumn{1}{c}{\%gap} &&
\multicolumn{1}{c}{obj} & \multicolumn{1}{c}{time} & \multicolumn{1}{c}{iter} & \multicolumn{1}{c}{\%gap} &&
\multicolumn{1}{c}{obj} & \multicolumn{1}{c}{time} & \multicolumn{1}{c}{iter} & \multicolumn{1}{c}{\%gap} \\ \hline
a0504010-01	&&	16	&	16	&&	16	{\,\,\,\,}	&	0.01 	&	0.0 	&&		16		{\,\,\,\,}	&	0.01 	&		1		&	0.0 	&&		16		{\,\,\,\,}	&	0.01 	&		1		&	0.0 	\\
a0504010-02	&&	18	&	18	&&	18	{\,\,\,\,}	&	0.02 	&	0.0 	&&		18		{\,\,\,\,}	&	0.02 	&		1		&	0.0 	&&		18		{\,\,\,\,}	&	0.02 	&		1		&	0.0 	\\
a0504010-03	&&	5	&	5	&&	5	{\,\,\,\,}	&	0.01 	&	0.0 	&&		5		{\,\,\,\,}	&	0.01 	&		1		&	0.0 	&&		5		{\,\,\,\,}	&	0.01 	&		1		&	0.0 	\\
a0504010-04	&&	18	&	18	&&	18	{\,\,\,\,}	&	0.02 	&	0.0 	&&		18		{\,\,\,\,}	&	0.02 	&		1		&	0.0 	&&		18		{\,\,\,\,}	&	0.02 	&		1		&	0.0 	\\
a0504010-05	&&	12	&	12	&&	12	{\,\,\,\,}	&	0.02 	&	0.0 	&&		12		{\,\,\,\,}	&	0.02 	&		1		&	0.0 	&&		12		{\,\,\,\,}	&	0.02 	&		1		&	0.0 	\\
a0504025-01	&&	80	&	80	&&	81	\,$\uparrow$	&	0.07 	&	1.2 	&&		80		{\,\,\,\,}	&	0.13 	&		2		&	0.0 	&&		80		{\,\,\,\,}	&	0.13 	&		2		&	0.0 	\\
a0504025-02	&&	67	&	67	&&	67	{\,\,\,\,}	&	0.03 	&	0.0 	&&		67		{\,\,\,\,}	&	0.03 	&		1		&	0.0 	&&		67		{\,\,\,\,}	&	0.03 	&		1		&	0.0 	\\
a0504025-03	&&	59	&	59	&&	59	{\,\,\,\,}	&	0.03 	&	0.0 	&&		59		{\,\,\,\,}	&	0.03 	&		1		&	0.0 	&&		59		{\,\,\,\,}	&	0.03 	&		1		&	0.0 	\\
a0504025-04	&&	71	&	71	&&	71	{\,\,\,\,}	&	0.07 	&	0.0 	&&		71		{\,\,\,\,}	&	0.07 	&		1		&	0.0 	&&		71		{\,\,\,\,}	&	0.07 	&		1		&	0.0 	\\
a0504025-05	&&	84	&	84	&&	86	\,$\uparrow$	&	0.21 	&	2.3 	&&		84		{\,\,\,\,}	&	0.49 	&		2		&	0.0 	&&		84		{\,\,\,\,}	&	0.39 	&		2		&	0.0 	\\
a0504050-01	&&	261	&	261	&&	261	{\,\,\,\,}	&	0.21 	&	0.0 	&&		261		{\,\,\,\,}	&	0.21 	&		1		&	0.0 	&&		261		{\,\,\,\,}	&	0.21 	&		1		&	0.0 	\\
a0504050-02	&&	204	&	204	&&	204	{\,\,\,\,}	&	0.12 	&	0.0 	&&		204		{\,\,\,\,}	&	0.12 	&		1		&	0.0 	&&		204		{\,\,\,\,}	&	0.12 	&		1		&	0.0 	\\
a0504050-03	&&	164	&	164	&&	164	{\,\,\,\,}	&	0.05 	&	0.0 	&&		164		{\,\,\,\,}	&	0.05 	&		1		&	0.0 	&&		164		{\,\,\,\,}	&	0.05 	&		1		&	0.0 	\\
a0504050-04	&&	227	&	227	&&	227	{\,\,\,\,}	&	0.13 	&	0.0 	&&		227		{\,\,\,\,}	&	0.13 	&		1		&	0.0 	&&		227		{\,\,\,\,}	&	0.13 	&		1		&	0.0 	\\
a0504050-05	&&	219	&	219	&&	219	{\,\,\,\,}	&	0.22 	&	0.0 	&&		219		{\,\,\,\,}	&	0.22 	&		1		&	0.0 	&&		219		{\,\,\,\,}	&	0.22 	&		1		&	0.0 	\\
a0508010-01	&&	23	&	23	&&	23	{\,\,\,\,}	&	0.03 	&	0.0 	&&		23		{\,\,\,\,}	&	0.03 	&		1		&	0.0 	&&		23		{\,\,\,\,}	&	0.03 	&		1		&	0.0 	\\
a0508010-02	&&	29	&	29	&&	29	{\,\,\,\,}	&	0.03 	&	0.0 	&&		29		{\,\,\,\,}	&	0.03 	&		1		&	0.0 	&&		29		{\,\,\,\,}	&	0.03 	&		1		&	0.0 	\\
a0508010-03	&&	19	&	19	&&	20	\,$\uparrow$	&	0.07 	&	5.0 	&&		19		{\,\,\,\,}	&	0.18 	&		3		&	0.0 	&&		19		{\,\,\,\,}	&	0.14 	&	\textbf{2}	&	0.0 	\\
a0508010-04	&&	27	&	27	&&	27	{\,\,\,\,}	&	0.02 	&	0.0 	&&		27		{\,\,\,\,}	&	0.02 	&		1		&	0.0 	&&		27		{\,\,\,\,}	&	0.02 	&		1		&	0.0 	\\
a0508010-05	&&	27	&	27	&&	27	{\,\,\,\,}	&	0.06 	&	0.0 	&&		27		{\,\,\,\,}	&	0.06 	&		1		&	0.0 	&&		27		{\,\,\,\,}	&	0.06 	&		1		&	0.0 	\\
a0508025-01	&&	118	&	141	&&	141	{\,\,\,\,}	&	0.15 	&	16.3 	&&		141		{\,\,\,\,}	&	0.15 	&		1		&	16.3 	&&		141		{\,\,\,\,}	&	0.15 	&		1		&	16.3 	\\
a0508025-02	&&	108	&	108	&&	108	{\,\,\,\,}	&	0.07 	&	0.0 	&&		108		{\,\,\,\,}	&	0.07 	&		1		&	0.0 	&&		108		{\,\,\,\,}	&	0.07 	&		1		&	0.0 	\\
a0508025-03	&&	91	&	91	&&	91	{\,\,\,\,}	&	0.23 	&	0.0 	&&		91		{\,\,\,\,}	&	0.23 	&		1		&	0.0 	&&		91		{\,\,\,\,}	&	0.23 	&		1		&	0.0 	\\
a0508025-04	&&	104	&	116	&&	116	{\,\,\,\,}	&	0.07 	&	10.3 	&&		116		{\,\,\,\,}	&	0.07 	&		1		&	10.3 	&&		116		{\,\,\,\,}	&	0.07 	&		1		&	10.3 	\\
a0508025-05	&&	105	&	105	&&	105	{\,\,\,\,}	&	0.12 	&	0.0 	&&		105		{\,\,\,\,}	&	0.12 	&		1		&	0.0 	&&		105		{\,\,\,\,}	&	0.12 	&		1		&	0.0 	\\
a0508050-01	&&	300	&	427	&&	427	{\,\,\,\,}	&	1.03 	&	29.7 	&&		427		{\,\,\,\,}	&	1.05 	&		1		&	29.7 	&&		427		{\,\,\,\,}	&	1.05 	&		1		&	29.7 	\\
a0508050-02	&&	291	&	388	&&	388	{\,\,\,\,}	&	0.31 	&	25.0 	&&		388		{\,\,\,\,}	&	0.31 	&		1		&	25.0 	&&		388		{\,\,\,\,}	&	0.31 	&		1		&	25.0 	\\
a0508050-03	&&	348	&	487	&&	487	{\,\,\,\,}	&	1.21 	&	28.5 	&&		487		{\,\,\,\,}	&	1.24 	&		1		&	28.5 	&&		487		{\,\,\,\,}	&	1.24 	&		1		&	28.5 	\\
a0508050-04	&&	307	&	390	&&	390	{\,\,\,\,}	&	0.25 	&	21.3 	&&		390		{\,\,\,\,}	&	0.25 	&		1		&	21.3 	&&		390		{\,\,\,\,}	&	0.25 	&		1		&	21.3 	\\
a0508050-05	&&	307	&	418	&&	418	{\,\,\,\,}	&	0.52 	&	26.6 	&&		418		{\,\,\,\,}	&	0.52 	&		1		&	26.6 	&&		418		{\,\,\,\,}	&	0.52 	&		1		&	26.6 	\\
a1004010-01	&&	14	&	14	&&	14	{\,\,\,\,}	&	0.04 	&	0.0 	&&		14		{\,\,\,\,}	&	0.04 	&		1		&	0.0 	&&		14		{\,\,\,\,}	&	0.04 	&		1		&	0.0 	\\
a1004010-02	&&	16	&	16	&&	16	{\,\,\,\,}	&	0.02 	&	0.0 	&&		16		{\,\,\,\,}	&	0.02 	&		1		&	0.0 	&&		16		{\,\,\,\,}	&	0.02 	&		1		&	0.0 	\\
a1004010-03	&&	14	&	14	&&	14	{\,\,\,\,}	&	0.01 	&	0.0 	&&		14		{\,\,\,\,}	&	0.01 	&		1		&	0.0 	&&		14		{\,\,\,\,}	&	0.01 	&		1		&	0.0 	\\
a1004010-04	&&	13	&	13	&&	13	{\,\,\,\,}	&	0.02 	&	0.0 	&&		13		{\,\,\,\,}	&	0.02 	&		1		&	0.0 	&&		13		{\,\,\,\,}	&	0.02 	&		1		&	0.0 	\\
a1004010-05	&&	16	&	16	&&	17	\,$\uparrow$	&	0.05 	&	5.9 	&&		16		{\,\,\,\,}	&	0.15 	&		3		&	0.0 	&&		16		{\,\,\,\,}	&	0.11 	&	\textbf{2}	&	0.0 	\\
a1004025-01	&&	78	&	78	&&	78	{\,\,\,\,}	&	0.16 	&	0.0 	&&		78		{\,\,\,\,}	&	0.16 	&		1		&	0.0 	&&		78		{\,\,\,\,}	&	0.16 	&		1		&	0.0 	\\
a1004025-02	&&	54	&	54	&&	54	{\,\,\,\,}	&	0.03 	&	0.0 	&&		54		{\,\,\,\,}	&	0.03 	&		1		&	0.0 	&&		54		{\,\,\,\,}	&	0.03 	&		1		&	0.0 	\\
a1004025-03	&&	64	&	64	&&	64	{\,\,\,\,}	&	0.08 	&	0.0 	&&		64		{\,\,\,\,}	&	0.08 	&		1		&	0.0 	&&		64		{\,\,\,\,}	&	0.08 	&		1		&	0.0 	\\
a1004025-04	&&	45	&	45	&&	45	{\,\,\,\,}	&	0.02 	&	0.0 	&&		45		{\,\,\,\,}	&	0.02 	&		1		&	0.0 	&&		45		{\,\,\,\,}	&	0.02 	&		1		&	0.0 	\\
a1004025-05	&&	73	&	73	&&	75	\,$\uparrow$	&	0.17 	&	2.7 	&&		73		{\,\,\,\,}	&	0.69 	&		4		&	0.0 	&&		73		{\,\,\,\,}	&	0.79 	&		4		&	0.0 	\\
a1004050-01	&&	184	&	211	&&	211	{\,\,\,\,}	&	0.24 	&	12.8 	&&		210		\,$\downarrow$	&	2.10 	&		9		&	12.4 	&&		210		\,$\downarrow$	&	2.40 	&		9		&	12.4 	\\
a1004050-02	&&	165	&	182	&&	182	{\,\,\,\,}	&	0.11 	&	9.3 	&&		182		{\,\,\,\,}	&	0.11 	&		1		&	9.3 	&&		182		{\,\,\,\,}	&	0.11 	&		1		&	9.3 	\\
a1004050-03	&&	181	&	206	&&	206	{\,\,\,\,}	&	0.29 	&	12.1 	&&		206		{\,\,\,\,}	&	0.29 	&		1		&	12.1 	&&		206		{\,\,\,\,}	&	0.30 	&		1		&	12.1 	\\
a1004050-04	&&	161	&	169	&&	169	{\,\,\,\,}	&	0.06 	&	4.7 	&&		169		{\,\,\,\,}	&	0.06 	&		1		&	4.7 	&&		169		{\,\,\,\,}	&	0.06 	&		1		&	4.7 	\\
a1004050-05	&&	178	&	206	&&	206	{\,\,\,\,}	&	0.55 	&	13.6 	&&		205		\,$\downarrow$	&	1.87 	&		3		&	13.2 	&&		205		\,$\downarrow$	&	1.92 	&		3		&	13.2 	\\
a1008010-01	&&	22	&	22	&&	22	{\,\,\,\,}	&	0.02 	&	0.0 	&&		22		{\,\,\,\,}	&	0.02 	&		1		&	0.0 	&&		22		{\,\,\,\,}	&	0.02 	&		1		&	0.0 	\\
a1008010-02	&&	34	&	34	&&	34	{\,\,\,\,}	&	0.10 	&	0.0 	&&		34		{\,\,\,\,}	&	0.10 	&		1		&	0.0 	&&		34		{\,\,\,\,}	&	0.10 	&		1		&	0.0 	\\
a1008010-03	&&	36	&	39	&&	39	{\,\,\,\,}	&	0.20 	&	7.7 	&&		39		{\,\,\,\,}	&	0.20 	&		1		&	7.7 	&&		39		{\,\,\,\,}	&	0.20 	&		1		&	7.7 	\\
a1008010-04	&&	16	&	16	&&	16	{\,\,\,\,}	&	0.04 	&	0.0 	&&		16		{\,\,\,\,}	&	0.04 	&		1		&	0.0 	&&		16		{\,\,\,\,}	&	0.04 	&		1		&	0.0 	\\
a1008010-05	&&	29	&	29	&&	29	{\,\,\,\,}	&	0.12 	&	0.0 	&&		29		{\,\,\,\,}	&	0.12 	&		1		&	0.0 	&&		29		{\,\,\,\,}	&	0.12 	&		1		&	0.0 	\\
a1008025-01	&&	77	&	87	&&	87	{\,\,\,\,}	&	0.05 	&	11.5 	&&		87		{\,\,\,\,}	&	0.05 	&		1		&	11.5 	&&	\textbf{86}	\,$\downarrow$	&	0.17 	&		3		&	10.5 	\\
a1008025-02	&&	103	&	117	&&	117	{\,\,\,\,}	&	0.11 	&	12.0 	&&		117		{\,\,\,\,}	&	0.11 	&		1		&	12.0 	&&		117		{\,\,\,\,}	&	0.11 	&		1		&	12.0 	\\
a1008025-03	&&	102	&	126	&&	126	{\,\,\,\,}	&	0.69 	&	19.0 	&&		126		{\,\,\,\,}	&	0.69 	&		1		&	19.0 	&&		126		{\,\,\,\,}	&	0.69 	&		1		&	19.0 	\\
a1008025-04	&&	98	&	113	&&	113	{\,\,\,\,}	&	0.08 	&	13.3 	&&		113		{\,\,\,\,}	&	0.08 	&		1		&	13.3 	&&		113		{\,\,\,\,}	&	0.08 	&		1		&	13.3 	\\
a1008025-05	&&	91	&	109	&&	109	{\,\,\,\,}	&	0.12 	&	16.5 	&&		109		{\,\,\,\,}	&	0.12 	&		1		&	16.5 	&&		109		{\,\,\,\,}	&	0.12 	&		1		&	16.5 	\\
a1008050-01	&&	248	&	367	&&	367	{\,\,\,\,}	&	0.26 	&	32.4 	&&		366		\,$\downarrow$	&	164.77 	&		275		&	32.2 	&&		366		\,$\downarrow$	&	459.16 	&		275		&	32.2 	\\
a1008050-02	&&	296	&	433	&&	433	{\,\,\,\,}	&	2.56 	&	31.6 	&&		433		{\,\,\,\,}	&	2.56 	&		1		&	31.6 	&&		433		{\,\,\,\,}	&	2.56 	&		1		&	31.6 	\\
a1008050-03	&&	306	&	419	&&	419	{\,\,\,\,}	&	8.63 	&	27.0 	&&		418		\,$\downarrow$	&	291.91 	&		27		&	26.6 	&&		418		\,$\downarrow$	&	221.88 	&		27		&	26.8 	\\
a1008050-04	&&	240	&	321	&&	321	{\,\,\,\,}	&	0.11 	&	25.2 	&&		321		{\,\,\,\,}	&	0.11 	&		1		&	25.2 	&&		321		{\,\,\,\,}	&	0.11 	&		1		&	25.2 	\\
a1008050-05	&&	264	&	376	&&	376	{\,\,\,\,}	&	0.20 	&	29.8 	&&		376		{\,\,\,\,}	&	0.20 	&		1		&	29.8 	&&		376		{\,\,\,\,}	&	0.20 	&		1		&	29.8 	\\
\hline
\end{tabular}
\end{table}

\begin{table}[htbp]
\centering \fontsize{8.5pt}{10.1pt}\selectfont
\setlength{\tabcolsep}{3.6pt}
\caption{MMR-GAP results for type-B instances}
\label{tbl:gap-b}
\begin{tabular}{ll*{16}{r}}
\hline
                             && \multicolumn{2}{c}{Best Known}                  && \multicolumn{3}{c}{DS}                  && \multicolumn{4}{c}{iDS-H}                                                                  && \multicolumn{4}{c}{iDS-B}                                                                      \\ \cline{3-4} \cline{6-8} \cline{10-13} \cline{15-18} 
\multicolumn{1}{c}{instance} && \multicolumn{1}{c}{LB} & \multicolumn{1}{c}{UB} &&
\multicolumn{1}{c}{obj} & \multicolumn{1}{c}{time} & \multicolumn{1}{c}{\%gap} &&
\multicolumn{1}{c}{obj} & \multicolumn{1}{c}{time} & \multicolumn{1}{c}{iter} & \multicolumn{1}{c}{\%gap} &&
\multicolumn{1}{c}{obj} & \multicolumn{1}{c}{time} & \multicolumn{1}{c}{iter} & \multicolumn{1}{c}{\%gap} \\ \hline
b0504010-01	&&	18 	&	18 	&&	24 	\,$\uparrow$	&	0.12 	&	25.0 	&&		18 		{\,\,\,\,}	&	0.23 	&		2 		&	0.0 	&&		18 		{\,\,\,\,}	&	0.23 	&		2 		&	0.0 	\\
b0504010-02	&&	19 	&	19 	&&	19 	{\,\,\,\,}	&	0.11 	&	0.0 	&&		19 		{\,\,\,\,}	&	0.11 	&		1 		&	0.0 	&&		19 		{\,\,\,\,}	&	0.11 	&		1 		&	0.0 	\\
b0504010-03	&&	12 	&	12 	&&	12 	{\,\,\,\,}	&	0.06 	&	0.0 	&&		12 		{\,\,\,\,}	&	0.06 	&		1 		&	0.0 	&&		12 		{\,\,\,\,}	&	0.06 	&		1 		&	0.0 	\\
b0504010-04	&&	25 	&	25 	&&	28 	\,$\uparrow$	&	0.09 	&	10.7 	&&		25 		{\,\,\,\,}	&	0.42 	&		4 		&	0.0 	&&		25 		{\,\,\,\,}	&	0.29 	&		4 		&	0.0 	\\
b0504010-05	&&	22 	&	22 	&&	22 	{\,\,\,\,}	&	0.14 	&	0.0 	&&		22 		{\,\,\,\,}	&	0.14 	&		1 		&	0.0 	&&		22 		{\,\,\,\,}	&	0.14 	&		1 		&	0.0 	\\
b0504025-01	&&	84 	&	84 	&&	84 	{\,\,\,\,}	&	0.28 	&	0.0 	&&		84 		{\,\,\,\,}	&	0.28 	&		1 		&	0.0 	&&		84 		{\,\,\,\,}	&	0.28 	&		1 		&	0.0 	\\
b0504025-02	&&	64 	&	64 	&&	64 	{\,\,\,\,}	&	0.14 	&	0.0 	&&		64 		{\,\,\,\,}	&	0.14 	&		1 		&	0.0 	&&		64 		{\,\,\,\,}	&	0.14 	&		1 		&	0.0 	\\
b0504025-03	&&	38 	&	38 	&&	38 	{\,\,\,\,}	&	0.12 	&	0.0 	&&		38 		{\,\,\,\,}	&	0.12 	&		1 		&	0.0 	&&		38 		{\,\,\,\,}	&	0.12 	&		1 		&	0.0 	\\
b0504025-04	&&	66 	&	66 	&&	68 	\,$\uparrow$	&	0.13 	&	2.9 	&&		66 		{\,\,\,\,}	&	1.89 	&		11 		&	0.0 	&&		66 		{\,\,\,\,}	&	1.62 	&	\textbf{10}	&	0.0 	\\
b0504025-05	&&	77 	&	77 	&&	77 	{\,\,\,\,}	&	0.20 	&	0.0 	&&		77 		{\,\,\,\,}	&	0.20 	&		1 		&	0.0 	&&		77 		{\,\,\,\,}	&	0.20 	&		1 		&	0.0 	\\
b0504050-01	&&	248 	&	248 	&&	248 	{\,\,\,\,}	&	0.59 	&	0.0 	&&		248 		{\,\,\,\,}	&	0.59 	&		1 		&	0.0 	&&		248 		{\,\,\,\,}	&	0.59 	&		1 		&	0.0 	\\
b0504050-02	&&	227 	&	227 	&&	227 	{\,\,\,\,}	&	1.06 	&	0.0 	&&		227 		{\,\,\,\,}	&	1.06 	&		1 		&	0.0 	&&		227 		{\,\,\,\,}	&	1.06 	&		1 		&	0.0 	\\
b0504050-03	&&	108 	&	108 	&&	108 	{\,\,\,\,}	&	0.15 	&	0.0 	&&		108 		{\,\,\,\,}	&	0.15 	&		1 		&	0.0 	&&		108 		{\,\,\,\,}	&	0.15 	&		1 		&	0.0 	\\
b0504050-04	&&	215 	&	215 	&&	215 	{\,\,\,\,}	&	0.68 	&	0.0 	&&		215 		{\,\,\,\,}	&	0.68 	&		1 		&	0.0 	&&		215 		{\,\,\,\,}	&	0.68 	&		1 		&	0.0 	\\
b0504050-05	&&	191 	&	191 	&&	191 	{\,\,\,\,}	&	0.43 	&	0.0 	&&		191 		{\,\,\,\,}	&	0.43 	&		1 		&	0.0 	&&		191 		{\,\,\,\,}	&	0.43 	&		1 		&	0.0 	\\
b0508010-01	&&	31 	&	31 	&&	31 	{\,\,\,\,}	&	0.87 	&	0.0 	&&		31 		{\,\,\,\,}	&	0.87 	&		1 		&	0.0 	&&		31 		{\,\,\,\,}	&	0.87 	&		1 		&	0.0 	\\
b0508010-02	&&	29 	&	29 	&&	29 	{\,\,\,\,}	&	0.19 	&	0.0 	&&		29 		{\,\,\,\,}	&	0.19 	&		1 		&	0.0 	&&		29 		{\,\,\,\,}	&	0.19 	&		1 		&	0.0 	\\
b0508010-03	&&	30 	&	30 	&&	30 	{\,\,\,\,}	&	0.12 	&	0.0 	&&		30 		{\,\,\,\,}	&	0.12 	&		1 		&	0.0 	&&		30 		{\,\,\,\,}	&	0.12 	&		1 		&	0.0 	\\
b0508010-04	&&	45 	&	45 	&&	45 	{\,\,\,\,}	&	2.25 	&	0.0 	&&		45 		{\,\,\,\,}	&	2.25 	&		1 		&	0.0 	&&		45 		{\,\,\,\,}	&	2.25 	&		1 		&	0.0 	\\
b0508010-05	&&	31 	&	31 	&&	31 	{\,\,\,\,}	&	0.20 	&	0.0 	&&		31 		{\,\,\,\,}	&	0.20 	&		1 		&	0.0 	&&		31 		{\,\,\,\,}	&	0.20 	&		1 		&	0.0 	\\
b0508025-01	&&	94 	&	111 	&&	111 	{\,\,\,\,}	&	15.28 	&	15.3 	&&		111 		{\,\,\,\,}	&	33.58 	&		2 		&	15.3 	&&		111 		{\,\,\,\,}	&	34.29 	&		2 		&	15.3 	\\
b0508025-02	&&	113 	&	122 	&&	122 	{\,\,\,\,}	&	4.20 	&	7.4 	&&		122 		{\,\,\,\,}	&	4.20 	&		1 		&	7.4 	&&		122 		{\,\,\,\,}	&	4.20 	&		1 		&	7.4 	\\
b0508025-03	&&	123 	&	147 	&&	147 	{\,\,\,\,}	&	10.61 	&	16.3 	&&		147 		{\,\,\,\,}	&	10.61 	&		1 		&	16.3 	&&		147 		{\,\,\,\,}	&	10.61 	&		1 		&	16.3 	\\
b0508025-04	&&	113 	&	131 	&&	131 	{\,\,\,\,}	&	7.99 	&	13.7 	&&	\textbf{130}	\,$\downarrow$	&	48.99 	&		4 		&	13.1 	&&		131 		{\,\,\,\,}	&	7.99 	&		1 		&	13.7 	\\
b0508025-05	&&	95 	&	107 	&&	107 	{\,\,\,\,}	&	3.37 	&	11.2 	&&		107 		{\,\,\,\,}	&	3.37 	&		1 		&	11.2 	&&		107 		{\,\,\,\,}	&	3.37 	&		1 		&	11.2 	\\
b0508050-01	&&	258 	&	346 	&&	346 	{\,\,\,\,}	&	54.15 	&	25.4 	&&		345 		\,$\downarrow$	&	99.58 	&		3 		&	25.2 	&&		345 		\,$\downarrow$	&	135.25 	&		3 		&	25.2 	\\
b0508050-02	&&	298 	&	415 	&&	415 	{\,\,\,\,}	&	49.21 	&	28.2 	&&		415 		{\,\,\,\,}	&	49.21 	&		1 		&	28.2 	&&		415 		{\,\,\,\,}	&	49.21 	&		1 		&	28.2 	\\
b0508050-03	&&	292 	&	402 	&&	402 	{\,\,\,\,}	&	63.67 	&	27.4 	&&		402 		{\,\,\,\,}	&	63.67 	&		1 		&	27.4 	&&		402 		{\,\,\,\,}	&	63.67 	&		1 		&	27.4 	\\
b0508050-04	&&	293 	&	421 	&&	421 	{\,\,\,\,}	&	70.98 	&	30.4 	&&		421 		{\,\,\,\,}	&	70.98 	&		1 		&	30.4 	&&		421 		{\,\,\,\,}	&	70.98 	&		1 		&	30.4 	\\
b0508050-05	&&	269 	&	398 	&&	398 	{\,\,\,\,}	&	55.11 	&	32.4 	&&		398 		{\,\,\,\,}	&	55.13 	&		1 		&	32.4 	&&		398 		{\,\,\,\,}	&	55.13 	&		1 		&	32.4 	\\
b1004010-01	&&	21 	&	21 	&&	24 	\,$\uparrow$	&	0.17 	&	12.5 	&&		21 		{\,\,\,\,}	&	0.56 	&		4 		&	0.0 	&&		21 		{\,\,\,\,}	&	0.66 	&		4 		&	0.0 	\\
b1004010-02	&&	16 	&	16 	&&	18 	\,$\uparrow$	&	0.11 	&	11.1 	&&		16 		{\,\,\,\,}	&	0.34 	&		3 		&	0.0 	&&		16 		{\,\,\,\,}	&	0.31 	&		3 		&	0.0 	\\
b1004010-03	&&	16 	&	16 	&&	16 	{\,\,\,\,}	&	0.17 	&	0.0 	&&		16 		{\,\,\,\,}	&	0.17 	&		1 		&	0.0 	&&		16 		{\,\,\,\,}	&	0.17 	&		1 		&	0.0 	\\
b1004010-04	&&	14 	&	14 	&&	14 	{\,\,\,\,}	&	0.12 	&	0.0 	&&		14 		{\,\,\,\,}	&	0.12 	&		1 		&	0.0 	&&		14 		{\,\,\,\,}	&	0.12 	&		1 		&	0.0 	\\
b1004010-05	&&	28 	&	28 	&&	31 	\,$\uparrow$	&	0.27 	&	9.7 	&&		28 		{\,\,\,\,}	&	0.86 	&	\textbf{3}	&	0.0 	&&		28 		{\,\,\,\,}	&	1.08 	&		4 		&	0.0 	\\
b1004025-01	&&	100 	&	103 	&&	103 	{\,\,\,\,}	&	2.42 	&	2.9 	&&		102 		\,$\downarrow$	&	7.80 	&		3 		&	2.0 	&&		102 		\,$\downarrow$	&	9.53 	&		3 		&	2.0 	\\
b1004025-02	&&	71 	&	71 	&&	72 	\,$\uparrow$	&	0.38 	&	1.4 	&&		71 		{\,\,\,\,}	&	1.01 	&		3 		&	0.0 	&&		71 		{\,\,\,\,}	&	1.06 	&		3 		&	0.0 	\\
b1004025-03	&&	63 	&	63 	&&	69 	\,$\uparrow$	&	0.34 	&	8.7 	&&		63 		{\,\,\,\,}	&	0.78 	&		2 		&	0.0 	&&		63 		{\,\,\,\,}	&	0.72 	&		2 		&	0.0 	\\
b1004025-04	&&	58 	&	58 	&&	58 	{\,\,\,\,}	&	0.31 	&	0.0 	&&		58 		{\,\,\,\,}	&	0.31 	&		1 		&	0.0 	&&		58 		{\,\,\,\,}	&	0.31 	&		1 		&	0.0 	\\
b1004025-05	&&	76 	&	76 	&&	79 	\,$\uparrow$	&	1.18 	&	3.8 	&&		76 		{\,\,\,\,}	&	4.15 	&		3 		&	0.0 	&&		76 		{\,\,\,\,}	&	3.89 	&		3 		&	0.0 	\\
b1004050-01	&&	184 	&	194 	&&	194 	{\,\,\,\,}	&	3.55 	&	5.2 	&&		193 		\,$\downarrow$	&	6.94 	&		2 		&	4.7 	&&		193 		\,$\downarrow$	&	7.17 	&		2 		&	4.7 	\\
b1004050-02	&&	189 	&	189 	&&	189 	{\,\,\,\,}	&	0.74 	&	0.0 	&&		189 		{\,\,\,\,}	&	0.74 	&		1 		&	0.0 	&&		189 		{\,\,\,\,}	&	0.74 	&		1 		&	0.0 	\\
b1004050-03	&&	186 	&	188 	&&	192 	\,$\uparrow$	&	5.46 	&	3.1 	&&		187 		\,$\downarrow$	&	39.88 	&		7 		&	0.5 	&&		187 		\,$\downarrow$	&	39.95 	&		7 		&	0.5 	\\
b1004050-04	&&	190 	&	199 	&&	199 	{\,\,\,\,}	&	4.20 	&	4.5 	&&		199 		{\,\,\,\,}	&	4.20 	&		1 		&	4.5 	&&		199 		{\,\,\,\,}	&	4.20 	&		1 		&	4.5 	\\
b1004050-05	&&	185 	&	197 	&&	197 	{\,\,\,\,}	&	9.64 	&	6.1 	&&		195 		\,$\downarrow$	&	21.82 	&		2 		&	5.1 	&&		195 		\,$\downarrow$	&	18.94 	&		2 		&	5.1 	\\
b1008010-01	&&	24 	&	24 	&&	24 	{\,\,\,\,}	&	0.34 	&	0.0 	&&		24 		{\,\,\,\,}	&	0.34 	&		1 		&	0.0 	&&		24 		{\,\,\,\,}	&	0.34 	&		1 		&	0.0 	\\
b1008010-02	&&	40 	&	40 	&&	41 	\,$\uparrow$	&	4.05 	&	2.4 	&&		40 		{\,\,\,\,}	&	10.21 	&		2 		&	0.0 	&&		40 		{\,\,\,\,}	&	10.14 	&		2 		&	0.0 	\\
b1008010-03	&&	42 	&	42 	&&	44 	\,$\uparrow$	&	3.65 	&	4.5 	&&		42 		{\,\,\,\,}	&	130.41 	&		19 		&	0.0 	&&		42 		{\,\,\,\,}	&	123.66 	&	\textbf{18}	&	0.0 	\\
b1008010-04	&&	25 	&	25 	&&	26 	\,$\uparrow$	&	0.60 	&	3.8 	&&		25 		{\,\,\,\,}	&	1.49 	&		2 		&	0.0 	&&		25 		{\,\,\,\,}	&	1.31 	&		2 		&	0.0 	\\
b1008010-05	&&	25 	&	25 	&&	25 	{\,\,\,\,}	&	0.72 	&	0.0 	&&		25 		{\,\,\,\,}	&	0.72 	&		1 		&	0.0 	&&		25 		{\,\,\,\,}	&	0.72 	&		1 		&	0.0 	\\
b1008025-01	&&	83 	&	91 	&&	91 	{\,\,\,\,}	&	5.04 	&	8.8 	&&		91 		{\,\,\,\,}	&	5.11 	&		1 		&	8.8 	&&		91 		{\,\,\,\,}	&	5.09 	&		1 		&	8.8 	\\
b1008025-02	&&	123 	&	171 	&&	171 	{\,\,\,\,}	&	829.66 	&	28.1 	&&		170 		\,$\downarrow$	&	1967.46 	&		2 		&	27.6 	&&		170 		\,$\downarrow$	&	1785.74 	&		2 		&	27.6 	\\
b1008025-03	&&	113 	&	124 	&&	125 	\,$\uparrow$	&	5.77 	&	9.6 	&&		124 		{\,\,\,\,}	&	11.90 	&		2 		&	8.9 	&&		124 		{\,\,\,\,}	&	13.02 	&		2 		&	8.9 	\\
b1008025-04	&&	97 	&	116 	&&	116 	{\,\,\,\,}	&	23.35 	&	16.4 	&&		116 		{\,\,\,\,}	&	23.35 	&		1 		&	16.4 	&&		116 		{\,\,\,\,}	&	23.35 	&		1 		&	16.4 	\\
b1008025-05	&&	96 	&	116 	&&	116 	{\,\,\,\,}	&	29.19 	&	17.2 	&&		116 		{\,\,\,\,}	&	29.19 	&		1 		&	17.2 	&&		116 		{\,\,\,\,}	&	29.19 	&		1 		&	17.2 	\\
b1008050-01	&&	236 	&	338 	&&	338 	{\,\,\,\,}	&	1083.19 	&	30.2 	&&		338 		{\,\,\,\,}	&	1083.19 	&		1 		&	30.2 	&&		338 		{\,\,\,\,}	&	1083.19 	&		1 		&	30.2 	\\
b1008050-02	&&	293 	&	444 	&&	445 	\,$\uparrow$	&	3600.00 	&	34.2 	&&		445 		\,$\uparrow$	&	3600.00 	&		1 		&	34.2 	&&		445 		\,$\uparrow$	&	3600.00 	&		1 		&	34.2 	\\
b1008050-03	&&	291 	&	409 	&&	409 	{\,\,\,\,}	&	2621.89 	&	28.9 	&&		409 		{\,\,\,\,}	&	2621.89 	&		1 		&	28.9 	&&		409 		{\,\,\,\,}	&	2621.89 	&		1 		&	28.9 	\\
b1008050-04	&&	259 	&	381 	&&	381 	{\,\,\,\,}	&	3600.00 	&	32.0 	&&		381 		{\,\,\,\,}	&	3600.00 	&		1 		&	32.0 	&&		381 		{\,\,\,\,}	&	3600.00 	&		1 		&	32.0 	\\
b1008050-05	&&	238 	&	346 	&&	346 	{\,\,\,\,}	&	3600.00 	&	31.2 	&&		346 		{\,\,\,\,}	&	3600.00 	&		1 		&	31.2 	&&		346 		{\,\,\,\,}	&	3600.00 	&		1 		&	31.2 	\\
\hline
\end{tabular}
\end{table}

\begin{table}[htbp]
\centering \fontsize{8.5pt}{10.1pt}\selectfont
\setlength{\tabcolsep}{3.7pt}
\caption{ MMR-GAP results for type-C instances}
\label{tbl:gap-c}
\begin{tabular}{ll*{16}{r}}
\hline
                             && \multicolumn{2}{c}{Best Known}                  && \multicolumn{3}{c}{DS}                  && \multicolumn{4}{c}{iDS-H}                                                                  && \multicolumn{4}{c}{iDS-B}                                                                      \\ \cline{3-4} \cline{6-8} \cline{10-13} \cline{15-18} 
\multicolumn{1}{c}{instance} && \multicolumn{1}{c}{LB} & \multicolumn{1}{c}{UB} &&
\multicolumn{1}{c}{obj} & \multicolumn{1}{c}{time} & \multicolumn{1}{c}{\%gap} &&
\multicolumn{1}{c}{obj} & \multicolumn{1}{c}{time} & \multicolumn{1}{c}{iter} & \multicolumn{1}{c}{\%gap} &&
\multicolumn{1}{c}{obj} & \multicolumn{1}{c}{time} & \multicolumn{1}{c}{iter} & \multicolumn{1}{c}{\%gap} \\ \hline
c0504010-01	&&	18 	&	18 	&&	18 	{\,\,\,\,}	&	0.14 	&	0.0 	&&		18 		{\,\,\,\,}	&	0.14 	&		1 		&	0.0 	&&		18 		{\,\,\,\,}	&	0.14 	&		1 		&	0.0 	\\
c0504010-02	&&	16 	&	16 	&&	16 	{\,\,\,\,}	&	0.06 	&	0.0 	&&		16 		{\,\,\,\,}	&	0.06 	&		1 		&	0.0 	&&		16 		{\,\,\,\,}	&	0.06 	&		1 		&	0.0 	\\
c0504010-03	&&	14 	&	14 	&&	14 	{\,\,\,\,}	&	0.08 	&	0.0 	&&		14 		{\,\,\,\,}	&	0.08 	&		1 		&	0.0 	&&		14 		{\,\,\,\,}	&	0.08 	&		1 		&	0.0 	\\
c0504010-04	&&	17 	&	17 	&&	17 	{\,\,\,\,}	&	0.15 	&	0.0 	&&		17 		{\,\,\,\,}	&	0.15 	&		1 		&	0.0 	&&		17 		{\,\,\,\,}	&	0.15 	&		1 		&	0.0 	\\
c0504010-05	&&	23 	&	23 	&&	23 	{\,\,\,\,}	&	0.15 	&	0.0 	&&		23 		{\,\,\,\,}	&	0.15 	&		1 		&	0.0 	&&		23 		{\,\,\,\,}	&	0.15 	&		1 		&	0.0 	\\
c0504025-01	&&	47 	&	47 	&&	47 	{\,\,\,\,}	&	0.18 	&	0.0 	&&		47 		{\,\,\,\,}	&	0.18 	&		1 		&	0.0 	&&		47 		{\,\,\,\,}	&	0.18 	&		1 		&	0.0 	\\
c0504025-02	&&	60 	&	60 	&&	60 	{\,\,\,\,}	&	0.17 	&	0.0 	&&		60 		{\,\,\,\,}	&	0.17 	&		1 		&	0.0 	&&		60 		{\,\,\,\,}	&	0.17 	&		1 		&	0.0 	\\
c0504025-03	&&	103 	&	103 	&&	103 	{\,\,\,\,}	&	0.27 	&	0.0 	&&		103 		{\,\,\,\,}	&	0.27 	&		1 		&	0.0 	&&		103 		{\,\,\,\,}	&	0.27 	&		1 		&	0.0 	\\
c0504025-04	&&	83 	&	83 	&&	83 	{\,\,\,\,}	&	0.19 	&	0.0 	&&		83 		{\,\,\,\,}	&	0.19 	&		1 		&	0.0 	&&		83 		{\,\,\,\,}	&	0.19 	&		1 		&	0.0 	\\
c0504025-05	&&	85 	&	85 	&&	85 	{\,\,\,\,}	&	0.34 	&	0.0 	&&		85 		{\,\,\,\,}	&	0.34 	&		1 		&	0.0 	&&		85 		{\,\,\,\,}	&	0.34 	&		1 		&	0.0 	\\
c0504050-01	&&	160 	&	160 	&&	160 	{\,\,\,\,}	&	1.37 	&	0.0 	&&		160 		{\,\,\,\,}	&	1.37 	&		1 		&	0.0 	&&		160 		{\,\,\,\,}	&	1.37 	&		1 		&	0.0 	\\
c0504050-02	&&	202 	&	202 	&&	202 	{\,\,\,\,}	&	2.85 	&	0.0 	&&		202 		{\,\,\,\,}	&	2.85 	&		1 		&	0.0 	&&		202 		{\,\,\,\,}	&	2.85 	&		1 		&	0.0 	\\
c0504050-03	&&	217 	&	217 	&&	217 	{\,\,\,\,}	&	2.76 	&	0.0 	&&		217 		{\,\,\,\,}	&	2.76 	&		1 		&	0.0 	&&		217 		{\,\,\,\,}	&	2.76 	&		1 		&	0.0 	\\
c0504050-04	&&	173 	&	173 	&&	173 	{\,\,\,\,}	&	0.49 	&	0.0 	&&		173 		{\,\,\,\,}	&	0.49 	&		1 		&	0.0 	&&		173 		{\,\,\,\,}	&	0.49 	&		1 		&	0.0 	\\
c0504050-05	&&	244 	&	269 	&&	269 	{\,\,\,\,}	&	3.58 	&	9.3 	&&		269 		{\,\,\,\,}	&	3.59 	&		1 		&	9.3 	&&		269 		{\,\,\,\,}	&	3.59 	&		1 		&	9.3 	\\
c0508010-01	&&	25 	&	25 	&&	25 	{\,\,\,\,}	&	0.70 	&	0.0 	&&		25 		{\,\,\,\,}	&	0.70 	&		1 		&	0.0 	&&		25 		{\,\,\,\,}	&	0.70 	&		1 		&	0.0 	\\
c0508010-02	&&	28 	&	28 	&&	28 	{\,\,\,\,}	&	0.18 	&	0.0 	&&		28 		{\,\,\,\,}	&	0.18 	&		1 		&	0.0 	&&		28 		{\,\,\,\,}	&	0.18 	&		1 		&	0.0 	\\
c0508010-03	&&	35 	&	35 	&&	35 	{\,\,\,\,}	&	0.69 	&	0.0 	&&		35 		{\,\,\,\,}	&	0.69 	&		1 		&	0.0 	&&		35 		{\,\,\,\,}	&	0.69 	&		1 		&	0.0 	\\
c0508010-04	&&	39 	&	39 	&&	39 	{\,\,\,\,}	&	0.89 	&	0.0 	&&		39 		{\,\,\,\,}	&	0.89 	&		1 		&	0.0 	&&		39 		{\,\,\,\,}	&	0.89 	&		1 		&	0.0 	\\
c0508010-05	&&	28 	&	28 	&&	28 	{\,\,\,\,}	&	0.29 	&	0.0 	&&		28 		{\,\,\,\,}	&	0.29 	&		1 		&	0.0 	&&		28 		{\,\,\,\,}	&	0.29 	&		1 		&	0.0 	\\
c0508025-01	&&	118 	&	132 	&&	132 	{\,\,\,\,}	&	0.88 	&	10.6 	&&		132 		{\,\,\,\,}	&	0.88 	&		1 		&	10.6 	&&		132 		{\,\,\,\,}	&	0.88 	&		1 		&	10.6 	\\
c0508025-02	&&	107 	&	112 	&&	112 	{\,\,\,\,}	&	2.38 	&	4.5 	&&		112 		{\,\,\,\,}	&	2.38 	&		1 		&	4.5 	&&		112 		{\,\,\,\,}	&	2.38 	&		1 		&	4.5 	\\
c0508025-03	&&	114 	&	131 	&&	131 	{\,\,\,\,}	&	5.61 	&	13.0 	&&		131 		{\,\,\,\,}	&	5.66 	&		1 		&	13.0 	&&		131 		{\,\,\,\,}	&	5.66 	&		1 		&	13.0 	\\
c0508025-04	&&	119 	&	142 	&&	142 	{\,\,\,\,}	&	20.82 	&	16.2 	&&		141 		\,$\downarrow$	&	62.26 	&		3 		&	15.6 	&&		141 		\,$\downarrow$	&	72.20 	&		3 		&	15.6 	\\
c0508025-05	&&	103 	&	116 	&&	116 	{\,\,\,\,}	&	3.78 	&	11.2 	&&		113 		\,$\downarrow$	&	9.39 	&		2 		&	8.8 	&&		113 		\,$\downarrow$	&	8.91 	&		2 		&	8.8 	\\
c0508050-01	&&	275 	&	402 	&&	402 	{\,\,\,\,}	&	59.66 	&	31.6 	&&		402 		{\,\,\,\,}	&	59.66 	&		1 		&	31.6 	&&		402 		{\,\,\,\,}	&	59.66 	&		1 		&	31.6 	\\
c0508050-02	&&	283 	&	389 	&&	389 	{\,\,\,\,}	&	18.80 	&	27.2 	&&		389 		{\,\,\,\,}	&	18.80 	&		1 		&	27.2 	&&		389 		{\,\,\,\,}	&	18.80 	&		1 		&	27.2 	\\
c0508050-03	&&	284 	&	418 	&&	419 	\,$\uparrow$	&	41.25 	&	32.2 	&&		418 		{\,\,\,\,}	&	96.07 	&		2 		&	32.1 	&&		418 		{\,\,\,\,}	&	107.59 	&		2 		&	32.1 	\\
c0508050-04	&&	326 	&	503 	&&	503 	{\,\,\,\,}	&	88.67 	&	35.2 	&&		503 		{\,\,\,\,}	&	88.67 	&		1 		&	35.2 	&&		503 		{\,\,\,\,}	&	88.67 	&		1 		&	35.2 	\\
c0508050-05	&&	261 	&	365 	&&	365 	{\,\,\,\,}	&	15.01 	&	28.5 	&&		365 		{\,\,\,\,}	&	15.05 	&		1 		&	28.5 	&&		365 		{\,\,\,\,}	&	15.05 	&		1 		&	28.5 	\\
c1004010-01	&&	20 	&	20 	&&	21 	\,$\uparrow$	&	0.21 	&	4.8 	&&		20 		{\,\,\,\,}	&	2.12 	&		9 		&	0.0 	&&		20 		{\,\,\,\,}	&	1.84 	&	\textbf{8}	&	0.0 	\\
c1004010-02	&&	15 	&	15 	&&	15 	{\,\,\,\,}	&	0.19 	&	0.0 	&&		15 		{\,\,\,\,}	&	0.19 	&		1 		&	0.0 	&&		15 		{\,\,\,\,}	&	0.19 	&		1 		&	0.0 	\\
c1004010-03	&&	24 	&	24 	&&	28 	\,$\uparrow$	&	0.23 	&	14.3 	&&		24 		{\,\,\,\,}	&	0.48 	&		2 		&	0.0 	&&		24 		{\,\,\,\,}	&	0.46 	&		2 		&	0.0 	\\
c1004010-04	&&	26 	&	26 	&&	26 	{\,\,\,\,}	&	0.11 	&	0.0 	&&		26 		{\,\,\,\,}	&	0.11 	&		1 		&	0.0 	&&		26 		{\,\,\,\,}	&	0.11 	&		1 		&	0.0 	\\
c1004010-05	&&	25 	&	25 	&&	25 	{\,\,\,\,}	&	0.11 	&	0.0 	&&		25 		{\,\,\,\,}	&	0.11 	&		1 		&	0.0 	&&		25 		{\,\,\,\,}	&	0.11 	&		1 		&	0.0 	\\
c1004025-01	&&	67 	&	67 	&&	68 	\,$\uparrow$	&	0.30 	&	1.5 	&&		67 		{\,\,\,\,}	&	1.84 	&		5 		&	0.0 	&&		67 		{\,\,\,\,}	&	1.77 	&		5 		&	0.0 	\\
c1004025-02	&&	58 	&	58 	&&	61 	\,$\uparrow$	&	0.28 	&	4.9 	&&		58 		{\,\,\,\,}	&	1.24 	&		4 		&	0.0 	&&		58 		{\,\,\,\,}	&	1.15 	&		4 		&	0.0 	\\
c1004025-03	&&	70 	&	70 	&&	70 	{\,\,\,\,}	&	0.39 	&	0.0 	&&		70 		{\,\,\,\,}	&	0.39 	&		1 		&	0.0 	&&		70 		{\,\,\,\,}	&	0.39 	&		1 		&	0.0 	\\
c1004025-04	&&	65 	&	65 	&&	66 	\,$\uparrow$	&	0.35 	&	1.5 	&&		65 		{\,\,\,\,}	&	7.84 	&		17 		&	0.0 	&&		65 		{\,\,\,\,}	&	8.36 	&		17 		&	0.0 	\\
c1004025-05	&&	83 	&	83 	&&	85 	\,$\uparrow$	&	0.63 	&	2.4 	&&		83 		{\,\,\,\,}	&	4.64 	&		6 		&	0.0 	&&		83 		{\,\,\,\,}	&	4.60 	&		6 		&	0.0 	\\
c1004050-01	&&	202 	&	212 	&&	216 	\,$\uparrow$	&	5.52 	&	6.5 	&&		212 		{\,\,\,\,}	&	318.00 	&		34 		&	4.7 	&&		212 		{\,\,\,\,}	&	343.69 	&		34 		&	4.7 	\\
c1004050-02	&&	219 	&	219 	&&	219 	{\,\,\,\,}	&	1.61 	&	0.0 	&&		219 		{\,\,\,\,}	&	1.61 	&		1 		&	0.0 	&&		219 		{\,\,\,\,}	&	1.61 	&		1 		&	0.0 	\\
c1004050-03	&&	189 	&	199 	&&	201 	\,$\uparrow$	&	6.23 	&	6.0 	&&		197 		\,$\downarrow$	&	17.99 	&		3 		&	4.1 	&&		197 		\,$\downarrow$	&	17.38 	&		3 		&	4.1 	\\
c1004050-04	&&	187 	&	187 	&&	187 	{\,\,\,\,}	&	1.31 	&	0.0 	&&		187 		{\,\,\,\,}	&	1.31 	&		1 		&	0.0 	&&		187 		{\,\,\,\,}	&	1.31 	&		1 		&	0.0 	\\
c1004050-05	&&	202 	&	226 	&&	230 	\,$\uparrow$	&	14.72 	&	12.2 	&&		225 		\,$\downarrow$	&	1013.36 	&		48 		&	10.2 	&&		225 		\,$\downarrow$	&	1056.08 	&	\textbf{46}	&	10.2 	\\
c1008010-01	&&	29 	&	29 	&&	29 	{\,\,\,\,}	&	0.67 	&	0.0 	&&		29 		{\,\,\,\,}	&	0.67 	&		1 		&	0.0 	&&		29 		{\,\,\,\,}	&	0.67 	&		1 		&	0.0 	\\
c1008010-02	&&	40 	&	40 	&&	43 	\,$\uparrow$	&	19.64 	&	7.0 	&&		40 		{\,\,\,\,}	&	913.48 	&		28 		&	0.0 	&&		40 		{\,\,\,\,}	&	889.42 	&	\textbf{27}	&	0.0 	\\
c1008010-03	&&	33 	&	33 	&&	34 	\,$\uparrow$	&	0.56 	&	2.9 	&&		33 		{\,\,\,\,}	&	1.28 	&		2 		&	0.0 	&&		33 		{\,\,\,\,}	&	1.31 	&		2 		&	0.0 	\\
c1008010-04	&&	37 	&	37 	&&	41 	\,$\uparrow$	&	2.36 	&	9.8 	&&		37 		{\,\,\,\,}	&	11.14 	&		3 		&	0.0 	&&		37 		{\,\,\,\,}	&	7.77 	&		3 		&	0.0 	\\
c1008010-05	&&	32 	&	32 	&&	32 	{\,\,\,\,}	&	1.09 	&	0.0 	&&		32 		{\,\,\,\,}	&	1.09 	&		1 		&	0.0 	&&		32 		{\,\,\,\,}	&	1.09 	&		1 		&	0.0 	\\
c1008025-01	&&	85 	&	87 	&&	87 	{\,\,\,\,}	&	4.70 	&	2.3 	&&		87 		{\,\,\,\,}	&	4.70 	&		1 		&	2.3 	&&		87 		{\,\,\,\,}	&	4.70 	&		1 		&	2.3 	\\
c1008025-02	&&	121 	&	157 	&&	159 	\,$\uparrow$	&	329.82 	&	23.9 	&&		155 		\,$\downarrow$	&	1024.08 	&		3 		&	21.9 	&&		155 		\,$\downarrow$	&	908.12 	&		3 		&	21.9 	\\
c1008025-03	&&	111 	&	129 	&&	129 	{\,\,\,\,}	&	52.73 	&	14.0 	&&		128 		\,$\downarrow$	&	146.62 	&		3 		&	13.3 	&&		128 		\,$\downarrow$	&	135.10 	&		3 		&	13.3 	\\
c1008025-04	&&	112 	&	120 	&&	120 	{\,\,\,\,}	&	7.15 	&	6.7 	&&		120 		{\,\,\,\,}	&	7.15 	&		1 		&	6.7 	&&		120 		{\,\,\,\,}	&	7.15 	&		1 		&	6.7 	\\
c1008025-05	&&	111 	&	142 	&&	142 	{\,\,\,\,}	&	257.09 	&	21.8 	&&		142 		{\,\,\,\,}	&	257.09 	&		1 		&	21.8 	&&		142 		{\,\,\,\,}	&	257.09 	&		1 		&	21.8 	\\
c1008050-01	&&	250 	&	346 	&&	346 	{\,\,\,\,}	&	804.54 	&	27.7 	&&		346 		{\,\,\,\,}	&	805.55 	&		1 		&	27.7 	&&		346 		{\,\,\,\,}	&	805.55 	&		1 		&	27.7 	\\
c1008050-02	&&	285 	&	436 	&&	436 	{\,\,\,\,}	&	3600.00 	&	34.6 	&&		436 		{\,\,\,\,}	&	3600.00 	&		1 		&	34.6 	&&		436 		{\,\,\,\,}	&	3600.00 	&		1 		&	34.6 	\\
c1008050-03	&&	283 	&	393 	&&	393 	{\,\,\,\,}	&	1334.99 	&	28.0 	&&		393 		{\,\,\,\,}	&	1334.99 	&		1 		&	28.0 	&&		393 		{\,\,\,\,}	&	1334.99 	&		1 		&	28.0 	\\
c1008050-04	&&	293 	&	418 	&&	418 	{\,\,\,\,}	&	1355.04 	&	29.9 	&&		418 		{\,\,\,\,}	&	1355.04 	&		1 		&	29.9 	&&		418 		{\,\,\,\,}	&	1355.04 	&		1 		&	29.9 	\\
c1008050-05	&&	264 	&	390 	&&	390 	{\,\,\,\,}	&	3600.00 	&	32.3 	&&		390 		{\,\,\,\,}	&	3600.00 	&		1 		&	32.3 	&&		390 		{\,\,\,\,}	&	3600.00 	&		1 		&	32.3 	\\
\hline
\end{tabular}
\end{table}

\begin{table}[htbp]
\centering \fontsize{8.5pt}{10.1pt}\selectfont
\setlength{\tabcolsep}{3.2pt}
\caption{MMR-GAP results for type-E instances}
\label{tbl:gap-e}
\begin{tabular}{ll*{16}{r}}
\hline
                             && \multicolumn{2}{c}{Best Known}                  && \multicolumn{3}{c}{DS}                  && \multicolumn{4}{c}{iDS-H}                                                                  && \multicolumn{4}{c}{iDS-B}                                                                      \\ \cline{3-4} \cline{6-8} \cline{10-13} \cline{15-18} 
\multicolumn{1}{c}{instance} && \multicolumn{1}{c}{LB} & \multicolumn{1}{c}{UB} &&
\multicolumn{1}{c}{obj} & \multicolumn{1}{c}{time} & \multicolumn{1}{c}{\%gap} &&
\multicolumn{1}{c}{obj} & \multicolumn{1}{c}{time} & \multicolumn{1}{c}{iter} & \multicolumn{1}{c}{\%gap} &&
\multicolumn{1}{c}{obj} & \multicolumn{1}{c}{time} & \multicolumn{1}{c}{iter} & \multicolumn{1}{c}{\%gap} \\ \hline
e0504010-01	&&	224 	&	224 	&&	224 	{\,\,\,\,}	&	0.44 	&	0.0 	&&		224 		{\,\,\,\,}	&	0.44 	&		1 		&	0.0 	&&		224 		{\,\,\,\,}	&	0.44 	&		1 		&	0.0 	\\
e0504010-02	&&	184 	&	184 	&&	184 	{\,\,\,\,}	&	0.70 	&	0.0 	&&		184 		{\,\,\,\,}	&	0.70 	&		1 		&	0.0 	&&		184 		{\,\,\,\,}	&	0.70 	&		1 		&	0.0 	\\
e0504010-03	&&	133 	&	133 	&&	133 	{\,\,\,\,}	&	0.16 	&	0.0 	&&		133 		{\,\,\,\,}	&	0.16 	&		1 		&	0.0 	&&		133 		{\,\,\,\,}	&	0.16 	&		1 		&	0.0 	\\
e0504010-04	&&	190 	&	190 	&&	190 	{\,\,\,\,}	&	0.79 	&	0.0 	&&		190 		{\,\,\,\,}	&	0.79 	&		1 		&	0.0 	&&		190 		{\,\,\,\,}	&	0.79 	&		1 		&	0.0 	\\
e0504010-05	&&	221 	&	221 	&&	225 	\,$\uparrow$	&	2.87 	&	1.8 	&&		221 		{\,\,\,\,}	&	12.36 	&		4 		&	0.0 	&&		221 		{\,\,\,\,}	&	14.77 	&		4 		&	0.0 	\\
e0504025-01	&&	826 	&	826 	&&	826 	{\,\,\,\,}	&	11.01 	&	0.0 	&&		826 		{\,\,\,\,}	&	11.01 	&		1 		&	0.0 	&&		826 		{\,\,\,\,}	&	11.01 	&		1 		&	0.0 	\\
e0504025-02	&&	706 	&	706 	&&	706 	{\,\,\,\,}	&	6.51 	&	0.0 	&&		706 		{\,\,\,\,}	&	6.51 	&		1 		&	0.0 	&&		706 		{\,\,\,\,}	&	6.51 	&		1 		&	0.0 	\\
e0504025-03	&&	659 	&	659 	&&	660 	\,$\uparrow$	&	3.61 	&	0.2 	&&		659 		{\,\,\,\,}	&	77.88 	&		13 		&	0.0 	&&		659 		{\,\,\,\,}	&	79.42 	&		13 		&	0.0 	\\
e0504025-04	&&	639 	&	639 	&&	639 	{\,\,\,\,}	&	7.37 	&	0.0 	&&		639 		{\,\,\,\,}	&	7.37 	&		1 		&	0.0 	&&		639 		{\,\,\,\,}	&	7.37 	&		1 		&	0.0 	\\
e0504025-05	&&	670 	&	670 	&&	670 	{\,\,\,\,}	&	15.90 	&	0.0 	&&		670 		{\,\,\,\,}	&	15.90 	&		1 		&	0.0 	&&		670 		{\,\,\,\,}	&	15.90 	&		1 		&	0.0 	\\
e0504050-01	&&	1689 	&	2056 	&&	2056 	{\,\,\,\,}	&	20.71 	&	17.9 	&&		2056 		{\,\,\,\,}	&	20.71 	&		1 		&	17.9 	&&		2056 		{\,\,\,\,}	&	20.71 	&		1 		&	17.9 	\\
e0504050-02	&&	1562 	&	2028 	&&	2028 	{\,\,\,\,}	&	129.15 	&	23.0 	&&		2028 		{\,\,\,\,}	&	129.15 	&		1 		&	23.0 	&&		2028 		{\,\,\,\,}	&	129.15 	&		1 		&	23.0 	\\
e0504050-03	&&	1312 	&	1519 	&&	1519 	{\,\,\,\,}	&	4.42 	&	13.6 	&&		1519 		{\,\,\,\,}	&	4.42 	&		1 		&	13.6 	&&		1519 		{\,\,\,\,}	&	4.42 	&		1 		&	13.6 	\\
e0504050-04	&&	1365 	&	1689 	&&	1689 	{\,\,\,\,}	&	18.42 	&	19.2 	&&		1670 		\,$\downarrow$	&	362.72 	&		9 		&	18.3 	&&		1670 		\,$\downarrow$	&	358.70 	&		9 		&	18.3 	\\
e0504050-05	&&	1420 	&	1731 	&&	1731 	{\,\,\,\,}	&	13.56 	&	18.0 	&&		1731 		{\,\,\,\,}	&	13.56 	&		1 		&	18.0 	&&		1731 		{\,\,\,\,}	&	13.56 	&		1 		&	18.0 	\\
e0508010-01	&&	290 	&	331 	&&	337 	\,$\uparrow$	&	19.46 	&	13.9 	&&		331 		{\,\,\,\,}	&	124.75 	&		3 		&	12.4 	&&		331 		{\,\,\,\,}	&	105.86 	&		3 		&	12.4 	\\
e0508010-02	&&	240 	&	314 	&&	318 	\,$\uparrow$	&	29.24 	&	24.5 	&&		309 		\,$\downarrow$	&	113.63 	&		3 		&	22.3 	&&		309 		\,$\downarrow$	&	79.86 	&		3 		&	22.3 	\\
e0508010-03	&&	274 	&	343 	&&	343 	{\,\,\,\,}	&	31.18 	&	20.1 	&&		334 		\,$\downarrow$	&	68.43 	&		2 		&	18.0 	&&		334 		\,$\downarrow$	&	77.30 	&		2 		&	18.0 	\\
e0508010-04	&&	282 	&	400 	&&	400 	{\,\,\,\,}	&	154.01 	&	29.5 	&&		400 		{\,\,\,\,}	&	154.01 	&		1 		&	29.5 	&&		400 		{\,\,\,\,}	&	154.01 	&		1 		&	29.5 	\\
e0508010-05	&&	264 	&	264 	&&	264 	{\,\,\,\,}	&	3.44 	&	0.0 	&&		264 		{\,\,\,\,}	&	3.44 	&		1 		&	0.0 	&&		264 		{\,\,\,\,}	&	3.44 	&		1 		&	0.0 	\\
e0508025-01	&&	833 	&	1230 	&&	1230 	{\,\,\,\,}	&	1897.28 	&	32.3 	&&		1230 		{\,\,\,\,}	&	1905.65 	&		1 		&	32.3 	&&		1230 		{\,\,\,\,}	&	1905.65 	&		1 		&	32.3 	\\
e0508025-02	&&	775 	&	1254 	&&	1251 	\,$\downarrow$	&	3600.00 	&	38.0 	&&		1251 		\,$\downarrow$	&	3600.00 	&		1 		&	38.0 	&&		1251 		\,$\downarrow$	&	3600.00 	&		1 		&	38.0 	\\
e0508025-03	&&	828 	&	1271 	&&	1271 	{\,\,\,\,}	&	268.21 	&	34.9 	&&		1271 		{\,\,\,\,}	&	269.27 	&		1 		&	34.9 	&&		1271 		{\,\,\,\,}	&	269.27 	&		1 		&	34.9 	\\
e0508025-04	&&	941 	&	1529 	&&	1541 	\,$\uparrow$	&	3600.00 	&	38.9 	&&		1541 		\,$\uparrow$	&	3600.00 	&		1 		&	38.9 	&&		1541 		\,$\uparrow$	&	3600.00 	&		1 		&	38.9 	\\
e0508025-05	&&	783 	&	1122 	&&	1122 	{\,\,\,\,}	&	31.56 	&	30.2 	&&		1122 		{\,\,\,\,}	&	31.56 	&		1 		&	30.2 	&&		1122 		{\,\,\,\,}	&	31.56 	&		1 		&	30.2 	\\
e0508050-01	&&	2447 	&	3985 	&&	4009 	\,$\uparrow$	&	3600.00 	&	39.0 	&&		4009 		\,$\uparrow$	&	3600.00 	&		1 		&	39.0 	&&		4009 		\,$\uparrow$	&	3600.00 	&		1 		&	39.0 	\\
e0508050-02	&&	2173 	&	3596 	&&	3603 	\,$\uparrow$	&	3600.00 	&	39.7 	&&		3603 		\,$\uparrow$	&	3600.00 	&		1 		&	39.7 	&&		3603 		\,$\uparrow$	&	3600.00 	&		1 		&	39.7 	\\
e0508050-03	&&	2194 	&	3655 	&&	3655 	{\,\,\,\,}	&	3600.00 	&	40.0 	&&		3655 		{\,\,\,\,}	&	3600.00 	&		1 		&	40.0 	&&		3655 		{\,\,\,\,}	&	3600.00 	&		1 		&	40.0 	\\
e0508050-04	&&	2337 	&	3711 	&&	3711 	{\,\,\,\,}	&	3600.00 	&	37.0 	&&		3711 		{\,\,\,\,}	&	3600.00 	&		1 		&	37.0 	&&		3711 		{\,\,\,\,}	&	3600.00 	&		1 		&	37.0 	\\
e0508050-05	&&	2103 	&	3448 	&&	3448 	{\,\,\,\,}	&	1899.93 	&	39.0 	&&		3448 		{\,\,\,\,}	&	1899.93 	&		1 		&	39.0 	&&		3448 		{\,\,\,\,}	&	1899.93 	&		1 		&	39.0 	\\
e1004010-01	&&	257 	&	266 	&&	273 	\,$\uparrow$	&	94.38 	&	5.9 	&&		272 		\,$\uparrow$	&	1757.52 	&		14 		&	5.5 	&&		272 		\,$\uparrow$	&	1678.94 	&		14 		&	5.5 	\\
e1004010-02	&&	223 	&	223 	&&	278 	\,$\uparrow$	&	21.43 	&	19.8 	&&		278 		\,$\uparrow$	&	21.43 	&		1 		&	19.8 	&&		278 		\,$\uparrow$	&	21.43 	&		1 		&	19.8 	\\
e1004010-03	&&	259 	&	270 	&&	277 	\,$\uparrow$	&	58.88 	&	6.5 	&&		270 		{\,\,\,\,}	&	190.75 	&		3 		&	4.1 	&&		270 		{\,\,\,\,}	&	190.62 	&		3 		&	4.1 	\\
e1004010-04	&&	250 	&	250 	&&	255 	\,$\uparrow$	&	20.99 	&	2.0 	&&		255 		\,$\uparrow$	&	20.99 	&		1 		&	2.0 	&&		255 		\,$\uparrow$	&	20.99 	&		1 		&	2.0 	\\
e1004010-05	&&	191 	&	191 	&&	201 	\,$\uparrow$	&	4.83 	&	5.0 	&&		201 		\,$\uparrow$	&	4.83 	&		1 		&	5.0 	&&		201 		\,$\uparrow$	&	4.83 	&		1 		&	5.0 	\\
e1004025-01	&&	629 	&	778 	&&	793 	\,$\uparrow$	&	708.83 	&	20.7 	&&		779 		\,$\uparrow$	&	2102.07 	&		2 		&	19.3 	&&		779 		\,$\uparrow$	&	2488.24 	&		2 		&	19.3 	\\
e1004025-02	&&	573 	&	674 	&&	674 	{\,\,\,\,}	&	113.52 	&	15.0 	&&		642 		\,$\downarrow$	&	1030.07 	&		5 		&	10.7 	&&		642 		\,$\downarrow$	&	844.25 	&		5 		&	10.7 	\\
e1004025-03	&&	566 	&	676 	&&	676 	{\,\,\,\,}	&	1029.30 	&	16.3 	&&		676 		{\,\,\,\,}	&	1029.31 	&		1 		&	16.3 	&&		676 		{\,\,\,\,}	&	1029.31 	&		1 		&	16.3 	\\
e1004025-04	&&	690 	&	829 	&&	829 	{\,\,\,\,}	&	515.01 	&	16.8 	&&		824 		\,$\downarrow$	&	957.66 	&		2 		&	16.3 	&&		824 		\,$\downarrow$	&	1126.31 	&		2 		&	16.3 	\\
e1004025-05	&&	562 	&	629 	&&	629 	{\,\,\,\,}	&	187.29 	&	10.7 	&&		629 		{\,\,\,\,}	&	187.29 	&		1 		&	10.7 	&&		629 		{\,\,\,\,}	&	187.29 	&		1 		&	10.7 	\\
e1004050-01	&&	1312 	&	1651 	&&	1651 	{\,\,\,\,}	&	1005.86 	&	20.5 	&&		1651 		{\,\,\,\,}	&	1005.86 	&		1 		&	20.5 	&&	\textbf{1647}	\,$\downarrow$	&	3600.00 	&		4 		&	20.3 	\\
e1004050-02	&&	1318 	&	1656 	&&	1656 	{\,\,\,\,}	&	836.70 	&	20.4 	&&		1656 		{\,\,\,\,}	&	836.70 	&		1 		&	20.4 	&&		1656 		{\,\,\,\,}	&	836.70 	&		1 		&	20.4 	\\
e1004050-03	&&	1289 	&	1667 	&&	1667 	{\,\,\,\,}	&	1963.44 	&	22.7 	&&		1667 		{\,\,\,\,}	&	1963.44 	&		1 		&	22.7 	&&		1667 		{\,\,\,\,}	&	1963.44 	&		1 		&	22.7 	\\
e1004050-04	&&	1470 	&	1988 	&&	2006 	\,$\uparrow$	&	3600.00 	&	26.7 	&&		2006 		\,$\uparrow$	&	3600.00 	&		1 		&	26.7 	&&		2006 		\,$\uparrow$	&	3600.00 	&		1 		&	26.7 	\\
e1004050-05	&&	1222 	&	1456 	&&	1456 	{\,\,\,\,}	&	136.91 	&	16.1 	&&		1456 		{\,\,\,\,}	&	136.91 	&		1 		&	16.1 	&&		1456 		{\,\,\,\,}	&	136.91 	&		1 		&	16.1 	\\
e1008010-01	&&	391 	&	572 	&&	577 	\,$\uparrow$	&	3600.00 	&	32.2 	&&		577 		\,$\uparrow$	&	3600.00 	&		1 		&	32.2 	&&		577 		\,$\uparrow$	&	3600.00 	&		1 		&	32.2 	\\
e1008010-02	&&	362 	&	574 	&&	578 	\,$\uparrow$	&	3600.00 	&	37.4 	&&		578 		\,$\uparrow$	&	3600.00 	&		1 		&	37.4 	&&		578 		\,$\uparrow$	&	3600.00 	&		1 		&	37.4 	\\
e1008010-03	&&	312 	&	427 	&&	428 	\,$\uparrow$	&	3600.00 	&	27.1 	&&		428 		\,$\uparrow$	&	3600.00 	&		1 		&	27.1 	&&		428 		\,$\uparrow$	&	3600.00 	&		1 		&	27.1 	\\
e1008010-04	&&	396 	&	578 	&&	583 	\,$\uparrow$	&	3600.00 	&	32.1 	&&		583 		\,$\uparrow$	&	3600.00 	&		1 		&	32.1 	&&		583 		\,$\uparrow$	&	3600.00 	&		1 		&	32.1 	\\
e1008010-05	&&	370 	&	538 	&&	534 	\,$\downarrow$	&	3600.00 	&	30.7 	&&		534 		\,$\downarrow$	&	3600.00 	&		1 		&	30.7 	&&		534 		\,$\downarrow$	&	3600.00 	&		1 		&	30.7 	\\
e1008025-01	&&	1210 	&	1893 	&&	1889 	\,$\downarrow$	&	3600.00 	&	35.9 	&&		1889 		\,$\downarrow$	&	3600.00 	&		1 		&	35.9 	&&		1889 		\,$\downarrow$	&	3600.00 	&		1 		&	35.9 	\\
e1008025-02	&&	1107 	&	1867 	&&	1895 	\,$\uparrow$	&	3600.00 	&	41.6 	&&		1895 		\,$\uparrow$	&	3600.00 	&		1 		&	41.6 	&&		1895 		\,$\uparrow$	&	3600.00 	&		1 		&	41.6 	\\
e1008025-03	&&	1012 	&	1616 	&&	1616 	{\,\,\,\,}	&	3600.00 	&	37.4 	&&		1616 		{\,\,\,\,}	&	3600.00 	&		1 		&	37.4 	&&		1616 		{\,\,\,\,}	&	3600.00 	&		1 		&	37.4 	\\
e1008025-04	&&	1058 	&	1713 	&&	1713 	{\,\,\,\,}	&	3600.00 	&	38.2 	&&		1713 		{\,\,\,\,}	&	3600.00 	&		1 		&	38.2 	&&		1713 		{\,\,\,\,}	&	3600.00 	&		1 		&	38.2 	\\
e1008025-05	&&	1090 	&	1726 	&&	1726 	{\,\,\,\,}	&	3600.00 	&	36.8 	&&		1726 		{\,\,\,\,}	&	3600.00 	&		1 		&	36.8 	&&		1726 		{\,\,\,\,}	&	3600.00 	&		1 		&	36.8 	\\
e1008050-01	&&	2775 	&	4588 	&&	4660 	\,$\uparrow$	&	3600.00 	&	40.5 	&&		4660 		\,$\uparrow$	&	3600.00 	&		1 		&	40.5 	&&		4660 		\,$\uparrow$	&	3600.00 	&		1 		&	40.5 	\\
e1008050-02	&&	2707 	&	4349 	&&	4351 	\,$\uparrow$	&	3600.00 	&	37.8 	&&		4351 		\,$\uparrow$	&	3600.00 	&		1 		&	37.8 	&&		4351 		\,$\uparrow$	&	3600.00 	&		1 		&	37.8 	\\
e1008050-03	&&	2581 	&	4233 	&&	4239 	\,$\uparrow$	&	3600.00 	&	39.1 	&&		4239 		\,$\uparrow$	&	3600.00 	&		1 		&	39.1 	&&		4239 		\,$\uparrow$	&	3600.00 	&		1 		&	39.1 	\\
e1008050-04	&&	2390 	&	3947 	&&	3953 	\,$\uparrow$	&	3600.00 	&	39.5 	&&		3953 		\,$\uparrow$	&	3600.00 	&		1 		&	39.5 	&&		3953 		\,$\uparrow$	&	3600.00 	&		1 		&	39.5 	\\
e1008050-05	&&	2394 	&	4040 	&&	4062 	\,$\uparrow$	&	3600.00 	&	41.1 	&&		4062 		\,$\uparrow$	&	3600.00 	&		1 		&	41.1 	&&		4062 		\,$\uparrow$	&	3600.00 	&		1 		&	41.1 	\\
\hline
\end{tabular}
\end{table}

%% file: ms.bbl
\begin{thebibliography}{40}
\providecommand{\natexlab}[1]{#1}
\providecommand{\url}[1]{\texttt{#1}}
\expandafter\ifx\csname urlstyle\endcsname\relax
  \providecommand{\doi}[1]{doi: #1}\else
  \providecommand{\doi}{doi: \begingroup \urlstyle{rm}\Url}\fi

\bibitem[Aissi et~al.(2009)Aissi, Bazgan, and Vanderpooten]{ABV09}
H.~Aissi, C.~Bazgan, and D.~Vanderpooten.
\newblock Min--max and min--max regret versions of combinatorial optimization
  problems: A survey.
\newblock \emph{European journal of operational research}, 197\penalty0
  (2):\penalty0 427--438, 2009.

\bibitem[Ayvaz-Cavdaroglu et~al.(2016)Ayvaz-Cavdaroglu, Kachani, and
  Maglaras]{AKM16}
N.~Ayvaz-Cavdaroglu, S.~Kachani, and C.~Maglaras.
\newblock Revenue management with minimax regret negotiations.
\newblock \emph{Omega}, 63:\penalty0 11--22, 2016.

\bibitem[Beasley(1990)]{B90}
J.~E. Beasley.
\newblock Or-library: distributing test problems by electronic mail.
\newblock \emph{Journal of the operational research society}, 41\penalty0
  (11):\penalty0 1069--1072, 1990.

\bibitem[Cacchiani et~al.(2014)Cacchiani, Hemmelmayr, and Tricoire]{CHT14}
V.~Cacchiani, V.~C. Hemmelmayr, and F.~Tricoire.
\newblock A set-covering based heuristic algorithm for the periodic vehicle
  routing problem.
\newblock \emph{Discrete Applied Mathematics}, 163:\penalty0 53--64, 2014.

\bibitem[Candia-V{\'e}jar et~al.(2011)Candia-V{\'e}jar, {\'A}lvarez-Miranda,
  and Maculan]{CAM11}
A.~Candia-V{\'e}jar, E.~{\'A}lvarez-Miranda, and N.~Maculan.
\newblock Minmax regret combinatorial optimization problems: an algorithmic
  perspective.
\newblock \emph{RAIRO-Operations Research-Recherche Op{\'e}rationnelle},
  45\penalty0 (2):\penalty0 101--129, 2011.

\bibitem[Chien and Zheng(2012)]{CZJ12}
C.-F. Chien and J.-N. Zheng.
\newblock Mini--max regret strategy for robust capacity expansion decisions in
  semiconductor manufacturing.
\newblock \emph{Journal of Intelligent Manufacturing}, 23\penalty0
  (6):\penalty0 2151--2159, 2012.

\bibitem[Chu and Beasley(1997)]{CB97}
P.~C. Chu and J.~E. Beasley.
\newblock A genetic algorithm for the generalised assignment problem.
\newblock \emph{Computers \& Operations Research}, 24\penalty0 (1):\penalty0
  17--23, 1997.

\bibitem[Chu and Beasley(1998)]{CB98}
P.~C. Chu and J.~E. Beasley.
\newblock A genetic algorithm for the multidimensional knapsack problem.
\newblock \emph{Journal of heuristics}, 4\penalty0 (1):\penalty0 63--86, 1998.

\bibitem[Deineko and Woeginger(2010)]{DW10}
V.~G. Deineko and G.~J. Woeginger.
\newblock Pinpointing the complexity of the interval min--max regret knapsack
  problem.
\newblock \emph{Discrete Optimization}, 7\penalty0 (4):\penalty0 191--196,
  2010.

\bibitem[Dong et~al.(2011)Dong, Huang, Cai, and Xu]{DHCX11}
C.~Dong, G.~Huang, Y.~Cai, and Y.~Xu.
\newblock An interval-parameter minimax regret programming approach for power
  management systems planning under uncertainty.
\newblock \emph{Applied Energy}, 88\penalty0 (8):\penalty0 2835--2845, 2011.

\bibitem[Farahani et~al.(2012)Farahani, Asgari, Heidari, Hosseininia, and
  Goh]{FAHHG12}
R.~Z. Farahani, N.~Asgari, N.~Heidari, M.~Hosseininia, and M.~Goh.
\newblock Covering problems in facility location: A review.
\newblock \emph{Computers \& Industrial Engineering}, 62\penalty0 (1):\penalty0
  368--407, 2012.

\bibitem[Fisher and Jaikumar(1981)]{FJ81}
M.~L. Fisher and R.~Jaikumar.
\newblock A generalized assignment heuristic for vehicle routing.
\newblock \emph{Networks}, 11\penalty0 (2):\penalty0 109--124, 1981.

\bibitem[Furini et~al.(2015)Furini, Iori, Martello, and Yagiura]{FIMY15}
F.~Furini, M.~Iori, S.~Martello, and M.~Yagiura.
\newblock Heuristic and exact algorithms for the interval min--max regret
  knapsack problem.
\newblock \emph{INFORMS Journal on Computing}, 27\penalty0 (2):\penalty0
  392--405, 2015.

\bibitem[Garey and Johnson(1979)]{GJ79}
M.~R. Garey and D.~S. Johnson.
\newblock \emph{Computers and intractability: A Guide to the Theory of
  NP-completeness}, volume 174.
\newblock W.H. Freeman, San Francisco, 1979.

\bibitem[Kara\c{s}an et~al.(2001)Kara\c{s}an, P{\i}nar, and Yaman]{KPY01}
O.~E. Kara\c{s}an, M.~c. P{\i}nar, and H.~Yaman.
\newblock The robust shortest path problem with interval data.
\newblock Technical report, Bilkent University, Ankara, Turkey, 2001.
\newblock available from the web site of Optimization Online
  (http://www.optimization-online.org/DB\_HTML/2001/08/361.html, accesed on
  June 1, 2019).

\bibitem[Karp(1972)]{K72}
R.~M. Karp.
\newblock Reducibility among combinatorial problems.
\newblock In \emph{Complexity of computer computations}, pages 85--103.
  Springer, 1972.

\bibitem[Kasperski(2008)]{K08}
A.~Kasperski.
\newblock \emph{Discrete Optimization with Interval Data}.
\newblock Springer, Berlin, 2008.

\bibitem[Kasperski and Zieli{\'n}ski(2006)]{KZ06}
A.~Kasperski and P.~Zieli{\'n}ski.
\newblock An approximation algorithm for interval data minmax regret
  combinatorial optimization problems.
\newblock \emph{Information Processing Letters}, 97\penalty0 (5):\penalty0
  177--180, 2006.

\bibitem[Kellerer et~al.(2004)Kellerer, Pferschy, and Pisinger]{KPP04}
H.~Kellerer, U.~Pferschy, and D.~Pisinger.
\newblock \emph{Knapsack problems}.
\newblock Springer, Berlin, 2004.

\bibitem[Kouvelis and Yu(2013)]{KY97}
P.~Kouvelis and G.~Yu.
\newblock \emph{Robust discrete optimization and its applications}, volume~14.
\newblock Springer Science \& Business Media, Dordrecht, 2013.

\bibitem[Kulik and Shachnai(2010)]{KS2010}
A.~Kulik and H.~Shachnai.
\newblock There is no eptas for two-dimensional knapsack.
\newblock \emph{Information Processing Letters}, 110\penalty0 (16):\penalty0
  707--710, 2010.

\bibitem[Laabadi et~al.(2018)Laabadi, Naimi, El~Amri, Achchab, et~al.]{LNAA18}
S.~Laabadi, M.~Naimi, H.~El~Amri, B.~Achchab, et~al.
\newblock The 0/1 multidimensional knapsack problem and its variants: a survey
  of practical models and heuristic approaches.
\newblock \emph{American Journal of Operations Research}, 8\penalty0
  (05):\penalty0 395, 2018.

\bibitem[Laguna et~al.(1995)Laguna, Kelly, Gonz{\'a}lez-Velarde, and
  Glover]{LKGG95}
M.~Laguna, J.~P. Kelly, J.~Gonz{\'a}lez-Velarde, and F.~Glover.
\newblock Tabu search for the multilevel generalized assignment problem.
\newblock \emph{European Journal of Operational Research}, 82:\penalty0
  176--189, 1995.

\bibitem[Loulou and Kanudia(1999)]{LK99}
R.~Loulou and A.~Kanudia.
\newblock Minimax regret strategies for greenhouse gas abatement: methodology
  and application.
\newblock \emph{Operations Research Letters}, 25\penalty0 (5):\penalty0
  219--230, 1999.

\bibitem[Martello and Toth(1990)]{MT90}
S.~Martello and P.~Toth.
\newblock \emph{Knapsack problems: Algorithms and computer implementations}.
\newblock John Wiley \& Sons, Chichester, 1990.

\bibitem[Montemanni(2006)]{M06}
R.~Montemanni.
\newblock A benders decomposition approach for the robust spanning tree problem
  with interval data.
\newblock \emph{European Journal of Operational Research}, 174\penalty0
  (3):\penalty0 1479--1490, 2006.

\bibitem[Montemanni and Gambardella(2005)]{MG05-2}
R.~Montemanni and L.~M. Gambardella.
\newblock A branch and bound algorithm for the robust spanning tree problem
  with interval data.
\newblock \emph{European Journal of Operational Research}, 161\penalty0
  (3):\penalty0 771--779, 2005.

\bibitem[Montemanni et~al.(2007)Montemanni, Barta, Mastrolilli, and
  Gambardella]{MBMG07}
R.~Montemanni, J.~Barta, M.~Mastrolilli, and L.~M. Gambardella.
\newblock The robust traveling salesman problem with interval data.
\newblock \emph{Transportation Science}, 41\penalty0 (3):\penalty0 366--381,
  2007.

\bibitem[Pereira(2018)]{Pe18}
J.~Pereira.
\newblock The robust (minmax regret) assembly line worker assignment and
  balancing problem.
\newblock \emph{Computers \& Operations Research}, 93:\penalty0 27--40, 2018.

\bibitem[Pereira and Averbakh(2011)]{PA11}
J.~Pereira and I.~Averbakh.
\newblock Exact and heuristic algorithms for the interval data robust
  assignment problem.
\newblock \emph{Computers \& Operations Research}, 38\penalty0 (8):\penalty0
  1153--1163, 2011.

\bibitem[Pereira and Averbakh(2013)]{PA13}
J.~Pereira and I.~Averbakh.
\newblock The robust set covering problem with interval data.
\newblock \emph{Annals of Operations Research}, 207\penalty0 (1):\penalty0
  217--235, 2013.

\bibitem[Poorsepahy-Samian et~al.(2012)Poorsepahy-Samian, Kerachian, and
  Nikoo]{PKN12}
H.~Poorsepahy-Samian, R.~Kerachian, and M.~R. Nikoo.
\newblock Water and pollution discharge permit allocation to agricultural
  zones: application of game theory and min-max regret analysis.
\newblock \emph{Water resources management}, 26\penalty0 (14):\penalty0
  4241--4257, 2012.

\bibitem[Puchinger et~al.(2010)Puchinger, Raidl, and Pferschy]{PRP10}
J.~Puchinger, G.~R. Raidl, and U.~Pferschy.
\newblock The multidimensional knapsack problem: Structure and algorithms.
\newblock \emph{INFORMS Journal on Computing}, 22\penalty0 (2):\penalty0
  250--265, 2010.

\bibitem[Ruland(1999)]{R99}
K.~S. Ruland.
\newblock A model for aeromedical routing and scheduling.
\newblock \emph{International Transactions in Operational Research}, 6\penalty0
  (1):\penalty0 57--73, 1999.

\bibitem[Wu et~al.(2016{\natexlab{a}})Wu, Hu, Hashimoto, Ando, Shiraki, and
  Yagiura]{WJJASU16}
W.~Wu, Y.~Hu, H.~Hashimoto, T.~Ando, T.~Shiraki, and M.~Yagiura.
\newblock A column generation approach to the airline crew pairing problem to
  minimize the total person-days.
\newblock \emph{Journal of Advanced Mechanical Design, Systems, and
  Manufacturing}, 10:\penalty0 JAMDSM0040, 2016{\natexlab{a}}.

\bibitem[Wu et~al.(2016{\natexlab{b}})Wu, Iori, Martello, and Yagiura]{WIMY16}
W.~Wu, M.~Iori, S.~Martello, and M.~Yagiura.
\newblock An iterated dual substitution approach for the min-max regret
  multidimensional knapsack problem.
\newblock In \emph{2016 IEEE International Conference on Industrial Engineering
  and Engineering Management (IEEM)}, pages 726--730, 2016{\natexlab{b}}.

\bibitem[Wu et~al.(2018{\natexlab{a}})Wu, Iori, Martello, and Yagiura]{WIMY18}
W.~Wu, M.~Iori, S.~Martello, and M.~Yagiura.
\newblock Exact and heuristic algorithms for the interval min-max regret
  generalized assignment problem.
\newblock \emph{Computers \& Industrial Engineering}, 125:\penalty0 98--110,
  2018{\natexlab{a}}.

\bibitem[Wu et~al.(2018{\natexlab{b}})Wu, Yagiura, and Ibaraki]{WYI18}
W.~Wu, M.~Yagiura, and T.~Ibaraki.
\newblock Generalized assignment problem.
\newblock In T.~F. Gonzalez, editor, \emph{Handbook of Approximation Algorithms
  and Metaheuristics: Methologies and Traditional Applications}, volume~1,
  chapter~40, pages 713--736. Chapman \& Hall/CRC, 2 edition,
  2018{\natexlab{b}}.

\bibitem[Xidonas et~al.(2017)Xidonas, Mavrotas, Hassapis, and
  Zopounidis]{XMHZ17}
P.~Xidonas, G.~Mavrotas, C.~Hassapis, and C.~Zopounidis.
\newblock Robust multiobjective portfolio optimization: A minimax regret
  approach.
\newblock \emph{European Journal of Operational Research}, 262\penalty0
  (1):\penalty0 299--305, 2017.

\bibitem[Yaman et~al.(2001)Yaman, Kara\c{s}an, and P{\i}nar]{YKP01}
H.~Yaman, O.~E. Kara\c{s}an, and M.~c. P{\i}nar.
\newblock The robust spanning tree problem with interval data.
\newblock \emph{Operations research letters}, 29\penalty0 (1):\penalty0 31--40,
  2001.

\end{thebibliography}
